\documentclass[12pt]{amsart}
\usepackage{amssymb,amsmath,txfonts,mathrsfs}
\newtheorem{theorem}{Theorem}[section]
\newtheorem{prop}[theorem]{Proposition}
\newtheorem{lemma}[theorem]{Lemma}
\newtheorem{remark}[theorem]{Remark}

\newtheorem{definition}[theorem]{Definition}
\newtheorem{cor}[theorem]{Corollary}

\numberwithin{equation}{section}

\def\pf{{\it Proof:}~}
 
\begin{document}

\title[Ricci flow on noncompact Riemannian manifolds]{Short-time existence of the Ricci flow on noncompact Riemannian manifolds}
\author{Guoyi Xu}
\date{\today}
\address{Department of Mathematics\\University of California, Irvine\\CA 92617}
\email{guoyixu@math.uci.edu}

\begin{abstract}
In this paper, we give the first detailed proof of the short-time existence of Deane Yang's local Ricci flow. Then using the local Ricci flow, we prove
the short-time existence of the Ricci flow on noncompact manifolds, whose Ricci
curvature has global lower bound and sectional curvature has only local average
integral bound. The short-time existence of the Ricci flow on noncompact manifolds with bounded curvature was studied by Wan-Xiong Shi in 1990s. As a corollary of our main theorem, we get the short-time existence part of Shi's theorem in this more general context. \\[3mm]
\end{abstract}

\maketitle

\tableofcontents

\section{Introduction}
In his well-known paper \cite{RT}, R. Hamilton introduced the Ricci flow on compact Riemannian manifolds which has proved to be very useful in the research of differential geometry and topology. Let us recall the definition of the Ricci flow, it is the solution of the evolution equation deforming the metric on any $n$-dimensional Riemannian manifold $(M^n, g_{ij})$:
\begin{equation}\label{1.1}
\left\{
\begin{array}{rl}
\frac{\partial}{\partial t}g_{ij}(x,t)&= -2R_{ij}(x,t) \\
g(x,0)&= g_{0}(x)
\end{array} \right.
\end{equation}
where $R_{ij}$ is the Ricci curvature of $M^n$ at time $t$.

The first thing in the study of the Ricci flow on Riemannian manifolds is the short-time existence of the solution. In the case where $M^n$ is compact, it is well known that for any given initial metric $g_{ij}$ on $M^n$, the evolution equation (\ref{1.1}) always has a unique solution for a short time (see \cite{DeTurck}, \cite{RT}). 

To prove Yau's Uniformization Conjecture, W-X Shi initiated studying
the Ricci flow on complete noncompact Riemannian manifolds. The short-time
existence problem of the evolution equation (\ref{1.1}) is more difficult than the compact case. In \cite{Shi}, Shi proved that if $(M^n, g_{ij})$ is complete noncompact with bounded curvature, then the Ricci flow (\ref{1.1}) has a solution with bounded curvature on a short time interval. Note Shi's theorem requires that the initial manifold $(M^n, g_0)$ has the bounded geometry, i.e. the curvature of $M^n$ must have point-wise estimate.  

In \cite{CRI}, Deane Yang introduced the local Ricci flow. We recall the definition of local Ricci flow firstly. Let $M^n$ be a smooth $n$-dimensional manifold with Riemannian metric $g_0$, $\Omega$ an open bounded domain of $M^n$ and $n \geq 3$. Let $\chi$ be a nonnegative smooth compactly supported function on $\Omega$, and $0\leq \chi\leq 1$. The local Ricci flow is the solution of the following evolution equation:
\begin{equation}\label{2.1}
\frac{\partial g}{\partial t}=-2\chi^2 Rc(g),\quad g(0)=g_0 , \quad \chi\in C_0^{\infty}(\Omega)
\end{equation}

In this paper, using the local Ricci flow, we prove
\begin{theorem}\label{thm 1.1}
{Assume $(M^n, g_0)$ is a $n$-dimensional $(n\geq 3)$ complete noncompact Riemannian manifold, satisfying the following conditions:
\begin{equation}\label{1.2}
{\left\{
\begin{array}{rl}
\Big(\fint_{B_x(4r_0)} h^{\frac{2n}{n-2}} dV_{g_0} \Big)^{\frac{n-2}{n}} &\leq  A_0 r_0^2 \fint_{B_x(4r_0)}|\nabla h|^2 dV_{g_0}\ , \\
\Big(\fint_{B_x(r_0)} |Rm(g_0)|^{p_0} dV_{g_0} \Big)^{\frac{1}{p_0}} &\leq K_1
\end{array}\right.
}
\end{equation}
for any $x\in M$ and $h\in C_0^{\infty}(B_x(4r_0))$, where $A_0\geq 1$, $p_0> \frac{n}{2}$, $K_1$ and $r_0$ are positive constants.

Then the Ricci flow (\ref{1.1}) has a smooth solution $g_{ij}(x,t)$ on $[0, \frac{T_{2}}{2}]$ (where $T_2$ is defined in (\ref{5.25}) ), and satisfies the following estimates. For any integer $m\geq 0$, there exists a positive constant $C(A_0,K_1,m,n,p_0,r_0)$ depending only on $A_0$, $K_1$, $m$, $n$, $p_0$ and $r_0$, such that
\begin{equation}\label{1.3}
{\sup_{x\in M}|\nabla^m Rm(x,t)|\leq \frac{C(A_0,K_1,m,n,p_0,r_0)}{t^{\frac{m}{2}}}(t^{-\frac{n}{2p_0}}+ t^{\frac{1}{p_0}}), \quad 0< t\leq \frac{T_{2}}{2}
}
\end{equation}
Especially, 
\begin{equation}\label{1.4}
{\sup_{x\in M}|Rm(x,t)|\leq C_2(t^{-\frac{n}{2p_0}}+ C_1^{\frac{n+ 2}{2p_0}}t^{\frac{1}{p_0}}), \quad 0< t\leq \frac{T_{2}}{2}
}
\end{equation}
where $C_1$ and $C_2$ are define in (\ref{5.9}) and (\ref{5.16}).
}
\end{theorem}


Theorem \ref{thm 1.1} requires only the integral conditions on the curvature tensor of the initial manifold and a local Sobolev inequality. By \cite{Saloff}, local Sobolev inequality is valid on Riemannian manifolds whose Ricci curvature has a lower bound. Hence we can generalize Shi's theorem in more general context in the following corollary.

\begin{cor}\label{cor 1.2}
{Assume $(M, g_0)$ is a $n$-dimensional $(n\geq 3)$ complete noncompact Riemannian manifold, satisfying the following conditions:
\begin{equation}\label{1.5}
{\left\{
\begin{array}{rl}
Rc(g_0) &\geq -Kg_0\ , \\
\Big(\fint_{B_x(r_0)} |Rm(g_0)|^{p_0} dV_{g_0} \Big)^{\frac{1}{p_0}} &\leq K_1
\end{array}\right.
}
\end{equation}
for any $x\in M$, where $p_0> \frac{n}{2}$, $r_0$, $K$ and $K_1$ are positive constants. 

Then the Ricci flow (\ref{1.1}) has a smooth solution $g_{ij}(x,t)$ on $[0, T]$, and satisfies the following estimates. For any integer $m\geq 0$, there exists positive constant $C(K,K_1,m,n,p_0,r_0)$ such that
\begin{equation}\nonumber
{\sup_{x\in M}|\nabla^m Rm(x,t)|\leq \frac{C(K,K_1,m,n,p_0,r_0)}{t^{\frac{m}{2}}} (t^{-\frac{n}{2p_0}}+ t^{\frac{1}{p_0}} ), \quad 0< t\leq T
}
\end{equation}
where $T= T(K,K_1,n,p_0,r_0)$ is a positive constant depending only on $K$, $K_1$, $n$, $p_0$ and $r_0$.
}
\end{cor}

In section $2$, we give the detailed proof of short-time existence of the local Ricci flow following the  brief sketch in \cite{CRI}. In the sketch proof of Theorem $8.2$ of \cite{CRI}, Yang considered a similar but different differential system (see the top of page $91$ of \cite{CRI}) comparing (\ref{2.5}), and the main difference is that our principal coefficients in (\ref{2.5}) are $(\chi^2+ \epsilon)u^{\alpha \beta}$ instead of $(\chi^2+ \epsilon)v^{\alpha \beta}$. This is the key point why we can get $C^0$ a-priori estimate of $v_{ij}$ (see (\ref{2.12}) etc.) and the higher order estimate further in Proposition \ref{prop 2.4}.


More concretely, for (\ref{2.5}) we prove that $\Phi_{\epsilon}:\ U\rightarrow V$ is a contraction map on $[0, T]$, where $T$ is independent of $\epsilon$. Then we use fixed point method to get solution of (\ref{2.5}) on $[0, T]$. Finally, let $\epsilon\rightarrow 0$, we get the solution of (\ref{2.2}) on $[0, T]$. And by regularity of parabolic system, the solution is smooth. All the estimates are about metric tensors and the norms we considered are respect to the initial metric $g_0$ of manifold $M^n$. 

In spirit, our proof is close to the original proof of Ricci flow's short-time existence on compact manifolds by R. Hamilton. In \cite{RT}, R. Hamilton used Nash-Moser inverse function theorem, in this paper we use fixed-point argument in a Sobolev-type inequality.

In section $3$, we discuss the evolution equations and inequalities of curvature tensors under local Ricci flow. In section $4$, extension of local Ricci flow is discussed. The main theorem (Theorem \ref{thm W10}) of this section was proved in \cite{CRI}, recently the complete proof was given by Wang among other things (see \cite{Wang}). The proof is non-trivial and different from the proof of Ricci flow's extension under the bound of curvature. In the proof, the interplay among the estimates of $\nabla^i \chi$, $\nabla^j Rm$ and $\nabla^k(\chi^2 Rm)$ are interesting and subtle. Following the proof of Wang, we prove this extension theorem again to make this paper self-contained. Note all the estimates in section $2$, $3$ and $4$ are valid for any smooth initial manifold $(M^n, g_0)$.

After establishing short-time existence of the local Ricci flow and its extension theorem, we try to get the existence time estimate of the local Ricci flow independent of the domain $\Omega$. This is done in section $5$ and section $6$. Note the time estimate which we get in the proof of short-time existence of the local Ricci flow, depends on $\Omega$. To get the time estimate independent of $\Omega$, we need to add some assumptions on the initial manifold $(M^n, g_0)$.

The key point of section $5$ is: If we assume initial local $L^p (p> \frac{n}{2})$ average integral bound of curvature tensor and local average Sobolev inequality, we can get the parabolic version of ``local local" energy estimate of curvature tensor, where the coefficient of $\fint \xi_i^{2q- 2}\chi^{2p'} |Rm|^p$ is independent of $\Omega$. Here ``local local" means that we times two cut-off functions with curvature tensor. One is $\xi_i$, which is from partition of unity of the good cover mentioned before; the other is $\chi$, which is the cut-off function in the local Ricci flow. 

Then using modified Moser iteration, we can get point-wise estimate of $\chi^2 Rm$ on $[0, T]$, where $T$ is independent of $\Omega$. By the local Ricci flow's extension theorem proved in section $4$, we get our existence time estimate independent of $\Omega$. The precise time estimate we get is expressed in explicit form of curvature bound, diameter of balls in good cover and local Sobolev constant, and this explicit estimate has its independent interest and application (see the remarks \ref{rem 5.15} and \ref{rem 7.1}). Another reason we get time estimate in such explicit expression (with $A_0$, $r_0$, $K_1$, see (\ref{5.21}) and (\ref{5.24}) etc.) is that we are interested in $Rc\geq 0$ case, where we can allow $r_0\rightarrow \infty$ if there is suitable initial curvature integral bound on $B(r_0)$ as $r_0\rightarrow \infty$.

In many situations, critical metrics are interesting, and we have only the $L^{\frac{n}{2}}$ bound of $Rm$. In section $6$, using different energy estimates we get similar existence time estimate of the local Ricci flow on manifolds with small $L^{\frac{n}{2}}$ bound of $Rm$, $L^{p} (p> \frac{n}{2})$ bound of $Rc$, local Sobolev inequality and volume doubling property.
 
In section 7, using a family of local Ricci flows with respect to a family of
expanding domains $\{\Omega_i\}_{i= 1}^{\infty}$ and the curvature estimates we get in section $5$ and $6$, we can apply compactness results in Ricci flow to show that a subsequence of this
family of local Ricci flows converges to the Ricci flow on the whole manifold when $i\rightarrow \infty$.

Finally, it is worth mentioning that our arguments in section $5$ and $6$ also apply
to compact manifolds. On compact manifolds, we can choose $\chi\equiv 1$ in (\ref{2.1}), then local Ricci flow is just Ricci flow. Note the minimal existence time dependent on curvature bound is well-known (see Corollary $7.7$ in \cite{RI}), the minimal existence time we get depends on the integral bound of curvature, which has its own interest.

Acknowledgement: The author is grateful to his advisor Professor Robert
Gulliver for reading this paper and providing comments and suggestions. He would
like to thank his teacher Professor Jiaping Wang for suggesting this problem and
bringing his attention to the article [15], also for his encouragement and many
helpful discussions. He would like to thank King-Yeung Lam, Weiwei Wu and
Weiyi Zhang for useful discussions.

 \section{Short-time existence of the local Ricci flow}

To prove the short-time existence of the local Ricci flow, we follow the brief sketch in Deane Yang's \cite{CRI}. By DeTurck's trick, we can equivalently consider the following local Ricci-DeTurck flow:
\begin{equation}\label{2.2}
\left\{
\begin{array}{rl}
\frac{\partial }{\partial t}g_{ij}&= -2\chi^2 R_{ij} + \nabla_{i}(\chi^2 \mathscr{W})_{j} + \nabla_{j}(\chi^2 \mathscr{W})_{i}\\
\frac{\partial }{\partial t} \phi_t(p)&= -\chi^2(p) \mathscr{W}(\phi_t(p), t), \quad p\in M\\
\phi_0 &= Id_{M}, \quad g_{ij}(x, 0)= \tilde{g}_{ij}(x)\triangleq (g_0)_{ij}(x)
\end{array} \right.
\end{equation}
where $\mathscr{W}_j= g_{jk}g^{pq}(\Gamma_{pq}^{k}- \tilde{\Gamma}_{pq}^{k})$. The second equation in (\ref{2.2}) is a quasilinear ODE, the solution on $M\setminus \Omega$ is $Id_{M\setminus \Omega}$, so one only need to consider the first family of equations in (\ref{2.2}).
 
\begin{lemma} \label{lem 2.1}
{The following modified evolution equations (\ref{2.4}) is a degenerate parabolic system.
\begin{equation}\label{2.3}
\left\{
\begin{array}{rl}
\frac{\partial }{\partial t}g_{ij}&= -2\chi^2 R_{ij} + \nabla_{i}(\chi^2 \mathscr{W})_{j} + \nabla_{j}(\chi^2 \mathscr{W})_{i}\\
g_{ij}(x, 0)&= \tilde{g}_{ij}(x)
\end{array} \right.
\end{equation}
}
\end{lemma}

\pf
{On $M\setminus \Omega$, $g_{ij}(x,t)\equiv \tilde{g}_{ij}(x)$, one only need to consider the evolution equations on $\Omega$. Calculate (\ref{2.3}) directly, use Lemma $2.1$ in \cite{Shi}, we get 
\begin{equation}\label{2.4}
\left\{
\begin{array}{rl}
\frac{\partial }{\partial t}g_{ij}&= \chi^2 g^{\alpha \beta} \tilde{\nabla}_{\alpha}\tilde{\nabla}_{\beta} g_{ij}- \chi^2 (g^{\alpha \beta} g_{ip} \tilde{g}^{pq} \tilde{R}_{j \alpha q \beta} + g^{\alpha \beta} g_{jp} \tilde{g}^{pq} \tilde{R}_{i \alpha q \beta})\\
& +\frac{1}{2}\chi^2 g^{\alpha \beta}g^{pq} (\tilde{\nabla}_i g_{p \alpha}\tilde{\nabla}_j g_{q \beta} +
2 \tilde{\nabla}_{\alpha} g_{jp}\tilde{\nabla}_q g_{i \beta} - 2 \tilde{\nabla}_{\alpha} g_{jp}\tilde{\nabla}_{\beta} g_{iq}\\
& \quad - 2 \tilde{\nabla}_{j} g_{p \alpha}\tilde{\nabla}_{\beta} g_{iq} -2 \tilde{\nabla}_{i} g_{p \alpha}\tilde{\nabla}_{\beta} g_{jq} )\\
& + \chi\chi_i g^{pq}(\tilde{\nabla}_p g_{jq}+ \tilde{\nabla}_{q} g_{pj}- \tilde{\nabla}_{j} g_{pq})+ \chi\chi_j g^{pq}(\tilde{\nabla}_p g_{iq}+ \tilde{\nabla}_{q} g_{pi}- \tilde{\nabla}_{i} g_{pq})\\
g_{ij}(x, 0)&= \tilde{g}_{ij}(x), \quad x\in \Omega.\\
g_{ij}(x, t)&= \tilde{g}_{ij}(x), \quad x\in \partial\Omega , \  t\in [0, T].
\end{array} \right.
\end{equation}
where $\chi_i= \tilde{\nabla}_i \chi$. Note the principal coefficients of (\ref{2.4}) are $(\chi^2 g^{\alpha \beta})$, it is a degenerate parabolic system. 
}
\qed

We define the following tensor-valued function spaces and norms:
\begin{definition} \label{def 2.2}
{\begin{align}
\mathfrak A_{T}&= \{U \big| \frac{1}{2}\tilde{g}_{ij}\leq u_{ij}\leq 2\tilde{g}_{ij}; \ U(x,0)=\tilde{g}(x), x\in \Omega ; \nonumber \\
U(x,t)&= \tilde{g}, x\in \partial\Omega ;\  U\in L^{\infty} \big(\bar{\Omega}\times [0, T], \otimes^2 T^{*}(\bar{\Omega})\big) \}  \nonumber \\
\big| U(x,t)\big|_{k}&= \Big(\sum_{|\alpha|\leq k} \int_{\Omega} |\tilde{\nabla}^{\alpha} U(x,t)|^2 d\mu_{\tilde{g}}\Big)^{\frac{1}{2}}, \ \ \big| U(x,t)\big|_{T, k}= \sup_{t\in [0, T]}(\big| U(x,t)\big|_{k}) \nonumber \\
\Vert f\Vert_k &\triangleq \sum_{|\alpha|\leq k}|\tilde{\nabla}^{\alpha} f|_{L^{\infty}(\Omega)}, \quad k= 0, 1, 2, \cdots \nonumber
\end{align}
where  $U: \big(\bar{\Omega}\times [0, T]\rightarrow \otimes^2 T^{*}(\bar{\Omega})\big)$, $f$ is a tensor field or function defined on $\Omega$ and $|\nabla^{\alpha} U(x,t)|^2= \tilde{g}^{i_1 j_1}\cdots \tilde{g}^{i_s j_s} \tilde{g}^{pk}\tilde{g}^{ql} (\tilde{\nabla}_{i_1}\cdots \tilde{\nabla}_{i_s} u_{pq})(\tilde{\nabla}_{j_1}\cdots \tilde{\nabla}_{j_s} u_{kl}) $, where $|\alpha|= s\geq 0$. And 
\[\mathbb B_m(K)_{T}= \{U \big| \ | U|_{T, m}\leq K; U\in L^{\infty} \big(\bar{\Omega}\times [0, T], \otimes^2 T^{*}(\bar{\Omega})\big) \} \]
where $m\geq 0$, $K> 0$.
}
\end{definition}

We consider the following strictly parabolic system by adding $\epsilon$ to $\chi^2$ in the principal coefficients, and $\epsilon\leq 1$ is some positive constant.
\begin{equation}\label{2.5}
\left\{
\begin{array}{rl}
\frac{\partial }{\partial t}v_{ij}&= (\chi^2 + \epsilon)u^{\alpha \beta}(v_{ij})_{\alpha \beta} + D_{ij}(U, V)+ B_{ij}(U, \tilde{\nabla}U) \\
v_{ij}(x, 0)&= \tilde{g}_{ij}(x) \quad x\in \Omega.\\
v_{ij}(x, t)&= \tilde{g}_{ij}(x) \quad x\in \partial\Omega , \ t\in [0, T].
\end{array} \right.
\end{equation}

where $U= (u_{ij})_{i,j= 1}^{n}$, $V= (v_{ij})_{i,j= 1}^{n}$, $(v_{ij})_{\alpha \beta}= \widetilde{\nabla}_{\alpha} \widetilde{\nabla}_{\beta} (v_{ij})$,
\begin{align}
D_{ij}(U, V)&= -\chi^2[(u^{\alpha \beta} \tilde{g}^{pq} \tilde{R}_{j \alpha q \beta})v_{ip} + (u^{\alpha \beta} \tilde{g}^{pq} \tilde{R}_{i \alpha q \beta})v_{jp}]\ , \nonumber \\
B_{ij}(U, \tilde{\nabla}U)&= E_{ij} + F_{ij}\ ,  \nonumber \\
E_{ij}&= \frac{1}{2}\chi^2 u^{\alpha \beta} u^{pq} (\tilde{\nabla}_i u_{p \alpha}\tilde{\nabla}_j u_{q \beta} + 2 \tilde{\nabla}_{\alpha} u_{jp}\tilde{\nabla}_q u_{i \beta} - 2 \tilde{\nabla}_{\alpha} u_{jp}\tilde{\nabla}_{\beta} u_{iq} \nonumber \\
&\quad - 2 \tilde{\nabla}_{j} u_{p \alpha}\tilde{\nabla}_{\beta} u_{iq} -2 \tilde{\nabla}_{i} u_{p \alpha}\tilde{\nabla}_{\beta} u_{jq} )\ ,  \nonumber \\
F_{ij}&=  \chi\chi_i u^{pq}(\tilde{\nabla}_p u_{jq}+ \tilde{\nabla}_{q} u_{pj}- \tilde{\nabla}_{j} u_{pq})+ \chi\chi_j u^{pq}(\tilde{\nabla}_p u_{iq}+ \tilde{\nabla}_{q} u_{pi}- \tilde{\nabla}_{i} u_{pq})\ , \nonumber 
\end{align}

We define $\Phi_{\epsilon}(U)= V$, if $V$ is the solution of (\ref{2.5}). And $A_{\epsilon}^{\alpha \beta}\triangleq (\chi^2+ \epsilon)u^{\alpha \beta}$. Then
\begin{equation} \label{2.6}
{\frac{\partial }{\partial t}v_{ij}= A_{\epsilon}^{\alpha \beta}(v_{ij})_{\alpha \beta} + D_{ij}(U, V)+ B_{ij}(U, \tilde{\nabla} U)
}
\end{equation}

\begin{remark}\label{rem 2.1}
{In (\ref{2.6}), there is no $\nabla V$ term, and this is one of the reasons we can get $C^0$-estimate of $V$ in Proposition \ref{prop 2.4}.
}
\end{remark}

We firstly prove a technical lemma about the norm of $U^{-1}$ using interpolation inequalities. We define the Sobolev constant $C_s(\Omega,g(t))$ for $\Omega\subset (M,g(t))$ as the following:
\begin{align}
|h|_{C^0(\Omega)}\leq C_s(\Omega, g(t))\Big(\int_{\Omega}|\nabla h|_{g(t)}^2 dV_{g(t)}\Big)^{\frac{1}{2}}, \quad \forall h\in C_0^{\infty}(\Omega) \label{2.7}
\end{align}
In this section we only need to use $C_s(\Omega, g_0)$ (which are equalt to $C_s(\Omega, g(0))= C_s(\Omega, \tilde{g})$), because the covariant derivatives, tensor norm considered in this section are all respect to $\tilde{g}$.  

Choose $G= 2\sqrt{n} V_{\tilde{g}}(\Omega)^{\frac{1}{2}}$ and some fixed $\tilde{p}> \min\{\frac{n}{2}, 3\}$ in the rest of this section. Note $\big| \tilde{g}\big|_{m}= \sqrt{n} V_{\tilde{g}}(\Omega)^{\frac{1}{2}}$ for any $m\geq 0$.

\begin{lemma}\label{lem 2.3}
{Assume $U\in \mathfrak A_T$ and $|U^{-1}(x, t)|_{L^{\infty}}\leq C_0$ for some $t\in [0, T]$, then 
\[ \big|U^{-1}(x, t)\big|_{k+1}\leq C(k,C_0,C_s(\Omega, g_0)) \big(\big|U(x, t)\big|_{\tilde{p}+ 1}+ 1 \big)^{k+1}(\big|U(x, t) \big|_{k+1}+ 1), \] 
where $k= 0, 1, 2, \cdots$.
}
\end{lemma}

\pf
{\begin{align}
&\big|\tilde{\nabla}^{k+1} U^{-1}\big|_{L^2}= |\tilde{\nabla}^k(U^{-1}U^{-1}\tilde{\nabla} U)|_{L^2}= 
|\tilde{\nabla}^k(U^{-2}\tilde{\nabla} U)|_{L^2} \nonumber \\
&\leq C(k)\big((|U^{-2}|_{L^{\infty}}+ 1) |\tilde{\nabla} U|_k + |U^{-2}|_k (|\tilde{\nabla} U|_{L^{\infty}}+ 1) \big) \nonumber \\
&\leq C(k, C_0)|\tilde{\nabla} U|_k + C(k, C_0, C_s(\Omega, g_0))|U^{-1}|_k (|U^{-1}|_{L^{\infty}}+ 1) (|U|_{\tilde{p}+ 1}+ 1) \nonumber \\
&\leq C(k, C_0, C_s(\Omega, g_0))|U|_{k+1} +C(k, C_0, C_s(\Omega, g_0))|U^{-1}|_k (|U|_{\tilde{p}+ 1}+ 1) \nonumber
\end{align}

Let $a_k= \big|U^{-1} \big|_k$, $c_k=\big|U \big|_k$, $b= (\big| U\big|_{\tilde{p}+ 1}+ 1)$, then $a_{k+1}\leq C(k, C_0, C_s(\Omega, g_0))c_{k+1}+C(k, C_0, C_s(\Omega, g_0))ba_k$. By induction, it is easy to get
\[a_{k+1}\leq C(k, C_0, C_s(\Omega, g_0))(b+ 1)^{k+1}(c_{k+1}+1)\]
Hence
\[\big|U^{-1}\big|_{k+1}\leq C(k, C_0, C_s(\Omega, g_0))(\big|U\big|_{\tilde{p}+ 1}+ 1)^{k+1}(\big|U \big|_{k+1}+ 1)\]
}
\qed

\begin{prop} \label{prop 2.4}
{There exists some $t_3= t_3(n,\tilde{p}, G, \Vert\chi\Vert_{\tilde{p}+ 2}, \Vert\widetilde{Rm}\Vert_{\tilde{p}+ 2}, C_s(\Omega, g_0))>0$, such that $\Phi_{\epsilon}: \mathfrak A_{t_3}\cap \mathbb B_{\tilde{p}+ 2}(G)_{t_3} \rightarrow \mathfrak A_{t_3}\cap \mathbb B_{\tilde{p}+ 2}(G)_{t_3}$.
}
\end{prop}

\pf
{Choose $U\in \mathfrak A_{t_1}\cap \mathbb B_{\tilde{p}+ 2}(G)_{t_1}$, $t_1> 0$ is to be determined later, we firstly show that $V\in \mathfrak A_{t_1}$ if $V= \Phi_{\epsilon}(U)$.

Choose a positive integer $m$ such that $2^m> 2n$, and choose a normal coordinate system $\{x^i\}$, such that at one point 
\begin{equation}\nonumber
\{\tilde{g}_{ij}\} = \left(
\begin{array}{cccc}
1 &   &  &\mathbf 0 \\
  & 1 &  & \\
  &   & \ddots & \\
\mathbf 0 & & & 1  
\end{array} \right)
\end{equation} 
\begin{equation}\nonumber
\{v_{ij}\} = \left(
\begin{array}{cccc}
\lambda_1 &   &  &\mathbf 0 \\
  & \lambda_2 &  & \\
  &   & \ddots & \\
\mathbf 0 & & & \lambda_n  
\end{array} \right)
\end{equation}

then
\begin{equation}\nonumber
\{v^{ij}\} = \left(
\begin{array}{cccc}
\lambda_1^{-1} &   &  &\mathbf 0 \\
  & \lambda_2^{-1} &  & \\
  &   & \ddots & \\
\mathbf 0 & & & \lambda_n^{-1}
\end{array} \right)
\end{equation}

Now one can define:
\begin{equation}\nonumber
{\phi_m \triangleq  v^{\alpha_1 \beta_1} \tilde{g}_{\beta_1 \alpha_2}\cdots v^{\alpha_m \beta_m} \tilde{g}_{\beta_m \alpha_1}= \sum_{i= 1}^{n} \lambda_i^{-m}
}
\end{equation}

By (\ref{2.6}) and $\frac{\partial}{\partial t}v^{ij}= -v^{ik}v^{jl}\frac{\partial}{\partial t}v_{kl}$, 
\begin{equation}\label{2.8}
{\frac{\partial}{\partial t}v^{ij}= -v^{ik}v^{jl}[A_{\epsilon}^{\alpha \beta}(v_{kl})_{\alpha \beta}+ D_{kl}(U,V)+ B_{kl}(U,\tilde{\nabla} U)]
}
\end{equation}

By $(61)$ on page $236$ of \cite{Shi}, 
\begin{equation}\label{2.9}
{(v^{ij})_{\alpha \beta}= \tilde{\nabla}_{\alpha}\tilde{\nabla}_{\beta}(v^{ij})= -v^{ik}v^{jl}\tilde{\nabla}_{\alpha}\tilde{\nabla}_{\beta}(v_{kl})- \tilde{\nabla}_{\alpha}(v^{ik}v^{jl})\tilde{\nabla}_{\beta}v_{kl}
}
\end{equation}

By (\ref{2.8}) and (\ref{2.9}), 
\begin{align}
\frac{\partial}{\partial t}v^{ij}&= A_{\epsilon}^{\alpha \beta}(v^{ij})_{\alpha \beta}+ A_{\epsilon}^{\alpha \beta}(v^{ik}v^{jl})_{\alpha} (v_{kl})_{\beta}- v^{ik}v^{jl}[D_{kl}(U,V)+ B_{kl}(U,\tilde{\nabla} U)] \nonumber \\
(\phi_m)_t &= v^{\alpha_1 \beta_1}\cdots \big(\frac{\partial}{\partial t} v^{\alpha_k \beta_k} \big)\cdots \tilde{g}_{\beta_m \alpha_1} \nonumber \\
&= v^{\alpha_1 \beta_1}\cdots \Big[ A_{\epsilon}^{\alpha \beta}(v^{\alpha_k \beta_k})_{\alpha \beta}+ A_{\epsilon}^{\alpha \beta}(v^{\alpha_k p}v^{\beta_k q})_{\alpha} (v_{pq})_{\beta} \nonumber \\
&\quad \quad - v^{\alpha_k p}v^{\beta_k q}[D_{pq}(U,V)+ B_{pq}(U,\tilde{\nabla} U)] \Big]\cdots \tilde{g}_{\beta_m \alpha_1}  \nonumber \\
&= A_{\epsilon}^{\alpha \beta}(\phi_m)_{\alpha \beta}+ 2\lambda_i^{-m}\chi^2 u^{\alpha \beta}R_{i\alpha i\beta}- \lambda_i^{-m-1}B_{ii}-2A_{\epsilon}^{\alpha \beta}\lambda_p^{-m-1}\lambda_q^{-1}(v_{pq})_{\alpha}(v_{pq})_{\beta} \nonumber \\
&\quad -2\sum_{k,l}\sum_{p< q} A_{\epsilon}^{\alpha \beta}(\lambda_k^{-1})^{m+p-q-3}(\lambda_l^{-1})^{q-p-3}(v_{kl})_{\alpha}(v_{kl})_{\beta}
\label{2.10}
\end{align}
and 
\begin{equation}\label{2.11}
{\left.
\begin{array}{rl}
&\quad -2\sum_{k,l}\sum_{p< q} A_{\epsilon}^{\alpha \beta}(\lambda_k^{-1})^{m+p-q-3}(\lambda_l^{-1})^{q-p-3}(v_{kl})_{\alpha}(v_{kl})_{\beta} \\
&= -m\sum_{k,l} A_{\epsilon}^{\alpha \beta}\sum_{\sigma= -2}^{m-4} (\lambda_k^{-1})^{m-\sigma -6}(\lambda_l^{-1})^{\sigma} (v_{kl})_{\alpha}(v_{kl})_{\beta}
\end{array} \right.
}
\end{equation}

By (\ref{2.10}) and (\ref{2.11}), 
\begin{equation}\nonumber
{\left.
\begin{array}{rl}
(\phi_m)_t &= (\chi^2+ \epsilon)(u^{\alpha \beta})(\phi_m)_{\alpha \beta}- 2(\lambda_p^{-1})^{m+1} \lambda_q^{-1} (\chi^2+ \epsilon) u^{\alpha \beta} (v_{pq})_{\alpha} (v_{pq})_{\beta} \\
& - m(\chi^2+ \epsilon)\sum_{\sigma=-2}^{m-4}u^{\alpha \beta}(v_{kl})_{\alpha} (v_{kl})_{\beta}(\lambda_k^{-1})^{m- \sigma- 6}(\lambda_l^{-1})^{\sigma} \\
&+ 2\lambda_i^{-m}\chi^2 u^{\alpha \beta} \tilde{R}_{i \alpha i \beta} - \lambda_i^{-m-1}B_{ii}
\end{array} \right.
}
\end{equation}

then 
\begin{equation}\label{2.12}
{(\phi_m)_t\leq (\chi^2 + \epsilon)u^{\alpha \beta}(\phi_m)_{\alpha \beta}+ C(n)\chi^2 \big|\widetilde{Rm}\big|\phi_m- \lambda_i^{-m-1}B_{ii} 
}
\end{equation}

Because $U\in \mathfrak A_{t_1}\cap \mathbb B_{\tilde{p}+ 2}(G)_{t_1}$, by definition of $B_{ij}$ and (\ref{2.7}), 
\begin{equation}\label{2.13}
{-\lambda_i^{-m-1}B_{ii}\leq C(n, ||\chi ||_1, C_s(\Omega, g_0))\cdot(G+1)^2 \phi_m^{1+\frac{1}{m}}
}
\end{equation}

From (\ref{2.12}), (\ref{2.13}) and Maximum Principle,
\begin{equation}\nonumber
{\big(\max_{x\in \Omega}\phi_m(x,t)\big)_t\leq C(n, ||\chi ||_1, ||\widetilde{Rm}||_0, C_s(\Omega, g_0), G)\big(\max_{x\in \Omega}\phi_m(x,t) + 1 \big)^2
}
\end{equation}
\begin{equation}\nonumber
{\phi_m(x, t)\leq \frac{1}{C_0- Ct}- 1, \quad C_0=\frac{1}{n+1}
}
\end{equation}

Note $\phi_m(0)= 2n$, one can choose $t_{1,1}= t_{1,1}(n, ||\chi ||_1, ||\widetilde{Rm}||_0, C_s(\Omega, g_0), G)> 0$, such that when $t\in [0, t_{1,1}]$, $\phi_m(x, t)\leq 2^{m}$, then 
\begin{equation}\nonumber
{\lambda_i \geq \frac{1}{2},  \quad i= 1, 2,\cdots, n
}
\end{equation}

Now define
\begin{equation}\nonumber
{\psi_m= v_{\alpha_1 \beta_1} \tilde{g}^{\beta_1 \alpha_2}\cdots v_{\alpha_m \beta_m} \tilde{g}^{\beta_m \alpha_1}= \sum_{i= 1}^{n} \lambda_i^{m}
}
\end{equation}
\begin{equation}\label{2.14}
{\left.
\begin{array}{rl}
(\psi_m)_t &= (\chi^2+ \epsilon)u^{\alpha \beta}(\psi_m)_{\alpha \beta}- 2\chi^2\lambda_i^m u^{\alpha \beta}\tilde{R}_{i \alpha i \beta}+ \lambda_i^{m-1}B_{ii}\\
&\quad - 2(\chi^2+ \epsilon)\sum_{j< k}\lambda_q^{m+j-k-1} \lambda_r^{k-j-1}u^{\alpha \beta}(v_{qr})_{\alpha}(v_{qr})_{\beta}\\
&\leq (\chi^2+ \epsilon)u^{\alpha \beta}(\psi_m)_{\alpha \beta}+ C(\psi_m + 1)^2
\end{array} \right.
}
\end{equation}

Similarly, by (\ref{2.14}) and Maximum Principle, 
\begin{equation}\nonumber
{\phi_m (x, t)\leq \frac{1}{C_0- Ct}- 1, \quad C_0=\frac{1}{n+1}
}
\end{equation}

Choose $t_{1,2}= t_{1,2}(n, ||\chi ||_1, ||\widetilde{Rm}||_0, C_s(\Omega, g_0), G)> 0$, such that when $t\in [0, t_{1,2}]$, $\psi_m(x,t)\leq 2^m$, then
\begin{equation}\nonumber
{\lambda_i\leq 2 \quad i= 1, 2,\cdots, n
}
\end{equation}

Set $t_1= \min\{t_{1,1}, t_{1,2}\}$ (note $t_1= t_1(n, ||\chi ||_1, ||\widetilde{Rm}||_0, C_s(\Omega, g_0), G)$) ,then when $t\in [0, t_1]$, $\frac{1}{2}\leq \lambda_i\leq 2$, hence $V\in \mathfrak A_{t_1}$. In fact we proved that $V\in \mathfrak A_{t_1}$ if $U\in \mathfrak A_{t_1}\cap \mathbb B_{\tilde{p}+ 2}(G)_{t_1}$ and $V= \Phi_{\epsilon}(U)$.

By (\ref{2.6}), for $m\geq \tilde{p}+ 2$, 
\begin{equation}\label{2.15}
{\left.
\begin{array}{rl}
\frac{\partial}{\partial t}\big|V \big|_m^2= 2\sum_{|\alpha |\leq m}\int_{\Omega} (\tilde{g}^{-1})^{|\alpha|+2}\tilde{\nabla}^{\alpha}v_{kl}\tilde{\nabla}^{\alpha}\Big(\frac{\partial}{\partial t}v_{ij} \Big)= I_1+ I_2+ I_3+ I_4 
\end{array}\right.
}
\end{equation}

where
\begin{equation}\nonumber
{I_1= 2\sum_{|\alpha |\leq m}\int_{\Omega} (\tilde{g}^{-1})^{|\alpha|+2}\tilde{\nabla}^{\alpha}v_{kl}\tilde{\nabla}^{\alpha}(v_{ij})_{pq} A_{\epsilon}^{pq}
}
\end{equation}
\begin{equation}\nonumber
{I_2= 2\sum_{|\alpha |\leq m}\int_{\Omega} (\tilde{g}^{-1})^{|\alpha|+2}\tilde{\nabla}^{\alpha}v_{kl} \Big[\tilde{\nabla}^{\alpha}(A_{\epsilon}^{pq}(v_{ij})_{pq})- A_{\epsilon}^{pq}\tilde{\nabla}^{\alpha}(v_{ij})_{pq} \Big]
}
\end{equation}
\begin{equation}\nonumber
{I_3= 2\sum_{|\alpha |\leq m}\int_{\Omega} (\tilde{g}^{-1})^{|\alpha|+2}\tilde{\nabla}^{\alpha}v_{kl}\tilde{\nabla}^{\alpha}D_{ij}, \quad I_4= 2\sum_{|\alpha |\leq m}\int_{\Omega} (\tilde{g}^{-1})^{|\alpha +2}\tilde{\nabla}^{\alpha}v_{kl}\tilde{\nabla}^{\alpha}B_{ij} 
}
\end{equation}
\begin{equation}\label{2.16}
{I_1 \leq -2\sum_{|\alpha |\leq m}\int  (\tilde{g}^{-1})^{|\alpha|+2} \Big((\tilde{\nabla}^{\alpha}v_{kl})_{p} \tilde{\nabla}^{\alpha}(v_{ij})_{q} A_{\epsilon}+ (\tilde{\nabla}^{\alpha}v_{kl}) \tilde{\nabla}^{\alpha}(v_{ij})_{q} (A_{\epsilon})_{p}\Big)
}
\end{equation}
\begin{equation}\nonumber
{\leq -\frac{1}{2}\sum_{|\alpha|= m}\int (\chi^2+ \epsilon)\big|\tilde{\nabla}^{\alpha+1}V\big|^2 + C(n, m, G, C_s(\Omega, g_0), ||\chi ||_1)\big| V\big|_m^2
}
\end{equation}
where we used the fact $|\tilde{\nabla} U^{-1}|\leq |U^{-1}|^2|\tilde{\nabla} U|\leq C(G, C_s(\Omega, g_0))$.
\begin{equation}\label{2.17}
{I_2\leq 2\sum_{|\alpha |\leq m}\Big|\tilde{\nabla}^{\alpha}V \Big|_{L^2}\cdot\Big[ C(||\chi ||_1)\Big(\int (\chi^2+ \epsilon)\big|\tilde{\nabla}^{\alpha+ 1}V\big|^2 \Big)^{\frac{1}{2}} 
}
\end{equation}
\begin{equation}\nonumber
{+ \Big|\tilde{\nabla}^{\alpha-2}[\tilde{\nabla}^2 A_{\epsilon}^{pq}(V)_{pq}] \Big|_{L^2} \Big]
}
\end{equation}

By interpolation inequality, it is easy to get
\begin{align}
\Big|\tilde{\nabla}^{\alpha-2}[\tilde{\nabla}^2 A^{pq}(V)_{pq}] \Big|_{L^2} \leq C(||\chi ||_2, m)|V|_m+ C(m, C_s(\Omega, g_0)) \big(|V|_{\tilde{p}+ 2}+ 1 \big) \big(|A_{\epsilon}^{pq}|_{m}+ 1 \big) \nonumber 
\end{align}

From lemma \ref{lem 2.3} and $|U^{-1}|_{L^{\infty}}\leq 2n$
\begin{equation}\nonumber
{\big|U^{-1}\big|_{k+1}\leq C(k, C_s(\Omega, g_0))(\big| U\big|_{\tilde{p}+ 1}+ 1)^{k+ 1}(\big| U\big|_{k+ 1}+ 1) 
}
\end{equation}

Hence
\begin{equation}\label{2.18}
{\Big|\tilde{\nabla}^{\alpha-2}[\tilde{\nabla}^2 A^{pq}(V)_{pq}] \Big|_{L^2} \leq  C(||\chi ||_2, m)\big| V\big|_m+ C(m, C_s(\Omega, g_0))\big( \big| V\big|_{\tilde{p}+ 2}+ 1\big) (\big| U\big|_m + 1)
}
\end{equation}

By (\ref{2.17}) and (\ref{2.18}), 
\begin{equation}\label{2.19}
{I_2\leq \frac{1}{16}\sum_{|\alpha|= m}\int (\chi^2+ \epsilon)\Big|\tilde{\nabla}^{\alpha+ 1} V\Big|^2 + C(||\chi ||_1)\big|V \big|_m^2 
}
\end{equation}
\begin{equation}\nonumber
{\quad + C(m, ||\chi ||_2, C_s(\Omega, g_0))\big| V\big|_m \big( \big| V\big|_{\tilde{p}+ 2}+ 1\big)(\big| U\big|_{m}+ 1)
}
\end{equation}

Similar with estimate of $I_2$,
\begin{equation}\label{2.20}
{I_3\leq C(m, ||\chi ||_m, ||\widetilde{Rm}||_m, C_s(\Omega, g_0))\Big( |V|_m^2+ |V|_m\big( \big| V\big|_{\tilde{p}}+ 1\big)(|U|_m+ 1) \Big)
}
\end{equation}

\begin{equation}\label{2.21}
{I_4= I_5+ I_6
}
\end{equation}
where 
\begin{align}
I_5&= 2\sum_{|\alpha |\leq m} \int (\tilde{g}^{-1})^{|\alpha|+ 2} \tilde{\nabla}^{\alpha}v_{kl} \tilde{\nabla}^{\alpha}E_{ij} \nonumber \\
I_6&= 2\sum_{|\alpha |\leq m} \int (\tilde{g}^{-1})^{|\alpha|+ 2} \tilde{\nabla}^{\alpha}v_{kl} \tilde{\nabla}^{\alpha} F_{ij} \nonumber 
\end{align}
for $I_5$, we only estimate one term of it as the following, the rest are similar.
\begin{align}
& 2\sum_{|\alpha |\leq m} \int (\tilde{g}^{-1})^{|\alpha|+ 2}\tilde{\nabla}^{\alpha}v_{kl} \tilde{\nabla}^{\alpha}[\frac{1}{2}\chi^2 u^{rs}u^{pq}(u_{rp})_i(u_{sq})_j] \nonumber \\
&\leq C(n)\sum_{\alpha} \Big[ |\tilde{\nabla}^{\alpha}V|_{L^2} \cdot |\tilde{\nabla}^{\alpha- 1}[\tilde{\nabla}(\chi^2 U^{-2})\cdot (\tilde{\nabla} U)^2]|_{L^2} \Big]+ C(n)\sum_{\alpha}\int \tilde{\nabla}^{\alpha}V\cdot (\chi^2 U^{-2})\tilde{\nabla}^{\alpha}(\tilde{\nabla} U\tilde{\nabla} U)  \nonumber \\
&\leq C(n, m, ||\chi ||_m, C_s(\Omega, g_0))|V|_m\cdot \Big[ |U^{-2}|_m |\big( (\tilde{\nabla} U)^2|_{\tilde{p}}+ 1\big)+ \big( |U^{-2}|_{\tilde{p}+ 1}|+ 1\big) (\tilde{\nabla} U)^2|_{m-1} \Big] \nonumber   \\
&+ C(n)\sum_{\alpha}\int \big| \tilde{\nabla}^{\alpha +1}V\cdot (\chi^2 U^{-2})\tilde{\nabla}^{\alpha -1}(\tilde{\nabla} U\tilde{\nabla} U) \big| + C(n)\sum_{\alpha}\int \big| \tilde{\nabla}^{\alpha}V \cdot \tilde{\nabla}(\chi^2 U^{-2})\tilde{\nabla}^{\alpha -1}(\tilde{\nabla} U\tilde{\nabla} U) \big| \nonumber \\
&\leq \frac{1}{64}\sum_{|\alpha|= m}\int \chi^2 |\tilde{\nabla}^{\alpha +1}V|^2 + C(n, m, ||\chi ||_m, C_s(\Omega, g_0), G) \cdot \Big( |V|_m(|U|_m+ 1)+ |U|_m^2 + |V|_m |U|_m \Big) \nonumber 
\end{align}

Hence
\begin{equation}\nonumber
{I_5\leq \frac{9}{64}\sum_{|\alpha|= m}\int \chi^2 |\tilde{\nabla}^{\alpha +1}V|^2 + C(n, m, ||\chi ||_m, C_s(\Omega, g_0), G) \Big( |V|_m(|U|_m+ 1)+ |U|_m^2 \Big)
}
\end{equation}

for $I_6$, we only estimate one term of it as the following, the rest are similar.
\begin{align}
&2\sum_{|\alpha |\leq m} \int (\tilde{g}^{-1})^{|\alpha|+ 2}\tilde{\nabla}^{\alpha}v_{kl} \tilde{\nabla}^{\alpha}[\chi\chi_i u^{pq}(u_{jq})_p]  \nonumber \\
&\leq C(n)\sum_{\alpha} \int \tilde{\nabla}^{\alpha}V \cdot \tilde{\nabla}^{\alpha- 1}[\tilde{\nabla}(\chi\tilde{\nabla}\chi U^{-1})\cdot (\tilde{\nabla} U)] + 2\sum_{\alpha}\int \tilde{\nabla}^{\alpha}V\cdot (\chi\tilde{\nabla}\chi U^{-1})\tilde{\nabla}^{\alpha+ 1}U \nonumber \\
&\leq \frac{1}{64}\sum_{|\alpha|= m}\int \chi^2 |\tilde{\nabla}^{\alpha +1}V|^2+ C(n, m, G, ||\chi ||_m, C_s(\Omega, g_0)) \cdot \Big( |V|_m(|U|_m+ 1)+  |U|_m^2 + |V|_m |U|_m \Big) \nonumber 
\end{align}

Hence
\begin{equation}\nonumber
{I_6\leq \frac{6}{64}\sum_{|\alpha|= m}\int \chi^2 |\tilde{\nabla}^{\alpha +1}V|^2+ C(n, m, G, ||\chi ||_m, C_s(\Omega, g_0)) \Big( |U|_m^2 + |V|_m (|U|_m+ 1) \Big)
}
\end{equation}

Then we get 
\begin{equation}\label{2.22}
{I_4\leq \frac{15}{64}\sum_{|\alpha|= m}\int \chi^2 |\tilde{\nabla}^{\alpha +1}V|^2
+ C(n, m, G, ||\chi ||_m, C_s(\Omega, g_0)) \Big( |U|_m^2 + |V|_m (|U|_m+ 1) \Big)
}
\end{equation}

By (\ref{2.15}), (\ref{2.16}), (\ref{2.19}), (\ref{2.20}) and (\ref{2.22}), and note the sum of the coefficients of $\sum_{|\alpha|\leq m}\int \chi^2 |\tilde{\nabla}^{\alpha+ 1}V|^2$ is negative on the right side, hence
\begin{equation}\label{2.23}
{\frac{\partial}{\partial t}|V|_m^2 \leq  C\Big( |V|_m^2+ |V|_m(|V|_{\tilde{p}+ 2}+ 1)(|U|_m+ 1)+ |U|_m^2 \Big)
}
\end{equation}
where $C= C(n, m, G, C_s(\Omega, g_0), ||\chi ||_m, ||\widetilde{Rm}||_m)$. For simplicity, we use $C$ instead of $C(n, m, G, C_s(\Omega, g_0), ||\chi ||_m, ||\widetilde{Rm}||_m)$ in the rest of the proof. Let $\phi(t)= |V|_m^2(t)$, then
\begin{equation}\nonumber
{\phi_t\leq C\phi+ C\phi^{\frac{1}{2}}Q_1+ Q_2
}
\end{equation}
where $Q_1= (|V|_{\tilde{p}+2}+1)(|U|_m+ 1)$ and $Q_2= C|U|_m^2$. We can get
\begin{equation}\label{2.24}
{\phi_t\leq C(\phi +1)(Q_1^2 + Q_2 +1)\leq CQ_4(\phi +1)
}
\end{equation}
where
\begin{equation}\label{2.25}
{Q_4= (|V|_{\tilde{p}+ 2}^2 +1)(|U|_m +1)^2
}
\end{equation}

Choose $m= \tilde{p}+ 2$, $\phi(t)= |V|_{\tilde{p}+ 2}^2(t)$, define $\psi(t)= |U|_{\tilde{p}+ 2}(t)$, then $Q_4= (\phi+1)(\psi+1)^2$,
\begin{equation}\nonumber
{\phi_t\leq C(\phi+ 1)^2(\psi +1)^2\leq C(G+1)^2(\phi+ 1)^2= C(\phi+ 1)^2
}
\end{equation}
where $\psi\leq G$ because of $U\in \mathbb B_{\tilde{p}+ 2}(G)_{t_1}$ and 
\[C= C(n,\tilde{p}, G, \Vert\chi\Vert_{\tilde{p}+ 2} \Vert\widetilde{Rm}\Vert_{\tilde{p}+ 2}, C_s(\Omega, g_0))\]
\begin{equation}\nonumber
{\phi(t)\leq \frac{1}{C_0- Ct}-1 ,\quad C_0= (\phi(0)+ 1)^{-1}
}
\end{equation}

Note $\phi(0)= (\frac{G}{2})^2$, then there exists $t_2> 0$, such that if $t\in [0, t_2]$, $\phi(t)\leq G^2$, that is 
\begin{equation}\nonumber
{|V|_{\tilde{p}+ 2}\leq G
}
\end{equation}

Choose 
\begin{equation}\label{2.26}
{t_3= \min\{t_1, t_2\}= t_3(n,\tilde{p}, G, \Vert\chi\Vert_{\tilde{p}+ 2}, \Vert\widetilde{Rm}\Vert_{\tilde{p}+ 2}, C_s(\Omega, g_0))
}
\end{equation}

We get the apriori estimate that $V\in \mathfrak A_{t_3}\cap \mathbb B_{\tilde{p}+ 2}(G)_{t_3}$ if $U\in \mathfrak A_{t_3}\cap \mathbb B_{\tilde{p}+ 2}(G)_{t_3}$ and $V= \Phi_{\epsilon}(U)$. By Theorem $8.3$ in \cite{Lieb} (which is for parabolic equations, but similar results apply for parabolic system here), in fact we get the solution of (\ref{2.5}) on $\Omega\times [0, t_3]$ and 
\[\Phi_{\epsilon}:\mathfrak A_{t_3}\cap \mathbb B_{\tilde{p}+ 2}(G)_{t_3}\longrightarrow \mathfrak A_{t_3}\cap \mathbb B_{\tilde{p}+ 2}(G)_{t_3} \] 
}
\qed

\begin{remark}
{The proof of Proposition \ref{prop 2.4} (especially the energy estimate) has the same spirit as in \cite{KL}, but the tensor version makes our estimate more subtle. Also note that $t_3$ is independent of $\epsilon$.
}
\end{remark}

\begin{prop}\label{prop 2.5}
{For any $m\geq \tilde{p}+ 2$, there exists 
\[t_5= t_5(n, m, G, ||\chi ||_m, ||\widetilde{Rm}||_m, C_s(\Omega, g_0))> 0\]
such that 
\[\Phi_{\epsilon}: \mathfrak A_{t_5}\cap \mathbb B_m(G)_{t_5}\longrightarrow \mathfrak A_{t_5}\cap \mathbb B_m(G)_{t_5}\]
}
\end{prop}

\pf
{Choose $U\in \mathfrak A_{T} \cap \mathbb B_m(G)_{T}$, where $T\leq t_3$ is to be determined later. By Proposition \ref{prop 2.4}, we get the solution $V= \Phi_{\epsilon}(U)\in \mathfrak A_{T}\cap \mathbb B_{\tilde{p}+ 2}(G)_{T}$, so $|V|_{\tilde{p}+ 2}\leq G$. And we also have $|U|_m\leq G$, by (\ref{2.24}) and (\ref{2.25}) we get
\begin{equation}\nonumber
{\phi_t\leq C(G^2+ 1)(G +1)^2(\phi +1)\leq C(\phi +1)
}
\end{equation}
where $\phi(t)= |V|_m^2(t)$ and $C= C(n,m,G,\Vert\chi\Vert_m, \Vert\widetilde{Rm}\Vert_m, C_s(\Omega, g_0))$.
\begin{equation}\nonumber
{\phi(t)\leq e^{Ct}[\phi(0)+ 1]- 1
}
\end{equation}
Note $\phi(0)= (\frac{G}{2})^2$. Choose $t_4= t_4(n,m,G,\Vert\chi\Vert_m, \Vert\widetilde{Rm}\Vert_m, C_s(\Omega, g_0))> 0$ such that $e^{Ct_4}[\phi(0)+ 1]- 1= G^2$, let $t_5= \min{t_3, t_4}= t_5(n,m,G,\Vert\chi\Vert_m, \Vert\widetilde{Rm}\Vert_m, C_s(\Omega, g_0))$
If $t\in [0, t_5]$, 
\begin{equation}\nonumber
{\phi(t)\leq e^{C t_5}[(\frac{G}{2})^2+ 1] - 1\leq G^2
}
\end{equation}
then $|V|_m\leq G$. Hence $V\in \mathfrak A_{t_5}\cap \mathbb B_m(G)_{t_5}$.
}
\qed

Next we show that $\Phi_{\epsilon}$ is a contraction map.

\begin{prop}\label{prop 2.6}
{For $0< \delta< 1$, there exists 
\[t_7= t_7(n, \tilde{p}, G, ||\chi ||_{\tilde{p}+ 2}, ||\widetilde{Rm}||_{\tilde{p}+ 2}, C_s(\Omega, g_0), \delta)> 0\] such that 
\begin{equation}\nonumber
{|\Phi_{\epsilon}(U)-\Phi_{\epsilon}(\tilde{U})|_{t_7,\tilde{p}}\leq \delta |U- \tilde{U}|_{t_7, \tilde{p}}
}
\end{equation}
where $U, \tilde{U}\in \mathfrak A_{t_7}\cap \mathbb B_{\tilde{p}+ 2}(G)_{t_7}$.
}
\end{prop}

\pf
{By (\ref{2.6}) we have the following: 
\begin{equation}\nonumber
\left\{
\begin{array}{rl}
V_t&= (\chi^2 + \epsilon)u^{\alpha \beta}(V)_{\alpha \beta} + P(U)V+ B(U, \tilde{\nabla}U)\\
\tilde{V}_t&= (\chi^2 + \epsilon)\tilde{u}^{\alpha \beta}(\tilde{V})_{\alpha \beta} + P(\tilde{U})\tilde{V}+ B(\tilde{U}, \tilde{\nabla}\tilde{U})
\end{array} \right.
\end{equation}

where
\begin{equation}\nonumber
{V= \Phi_{\epsilon}(U) ;\ (P(U)V)_{ij}= D_{ij}(U,V) ;\ (B(U, \tilde{\nabla}U))_{ij}= B_{ij}(U, \tilde{\nabla}U) ;\ (V)_{\alpha \beta}= \tilde{\nabla}_{\alpha}\tilde{\nabla}_{\beta} V
}
\end{equation}
similar notation for $\tilde{V}$ etc. By Proposition \ref{prop 2.4}, if $T\leq t_3$, where $t_3$ is defined in $(\ref{2.26})$, then $V$, $\tilde{V}\in \mathfrak A_{T}\cap \mathbb B_{\tilde{p}+ 2}(G)_{T}$.

Define $W\triangleq V- \tilde{V}= \Phi_{\epsilon}(U)- \Phi_{\epsilon}(\tilde{U})$, then 
\begin{equation}\nonumber
{W_t= (\chi^2+ \epsilon)u^{\alpha \beta}w_{\alpha \beta}+ P(U)W+ H
}
\end{equation}

where 
\begin{equation}\nonumber
{H= (\chi^2+ \epsilon) (u^{\alpha \beta}- \tilde{u}^{\alpha \beta})(\tilde{V})_{\alpha \beta} + (P(U)- P(\tilde{U}))\tilde{V} + \Big(B(U, \tilde{\nabla}U)- B(\tilde{U}, \tilde{\nabla}\tilde{U}) \Big)
}
\end{equation}

We consider the evolution equation of $|W|_{\tilde{p}}^2$:
\begin{equation}\label{2.27}
{\frac{\partial}{\partial t} |W|_{\tilde{p}}^2= J_1+ J_2+ J_3+ J_4
}
\end{equation}

where 
\begin{equation}\nonumber
{J_1= 2\sum_{|\alpha|\leq \tilde{p}}\int (\tilde{g}^{-1})^{|\alpha|+ 2}\tilde{\nabla}^{\alpha} w_{kl} \tilde{\nabla}^{\alpha} (w_{ij})_{pq}A_{\epsilon}^{pq}
}
\end{equation}
\begin{equation}\nonumber
{J_2= 2\sum \int (\tilde{g}^{-1})^{|\alpha|+ 2}\tilde{\nabla}^{\alpha} w_{kl}[\tilde{\nabla}^{\alpha}(A_{\epsilon}^{pq}(w_{ij})_{pq})- A_{\epsilon}^{pq}\tilde{\nabla}^{\alpha}(w_{ij})_{pq}]
}
\end{equation}
\begin{equation}\nonumber
{J_3= 2\sum_{|\alpha|\leq \tilde{p}}\int (\tilde{g}^{-1})^{|\alpha|+ 2}\tilde{\nabla}^{\alpha} w_{kl} \tilde{\nabla}^{\alpha} (P(U)W)_{ij}, \quad J_4= 2\sum_{|\alpha|\leq \tilde{p}}\int (\tilde{g}^{-1})^{|\alpha|+ 2}\tilde{\nabla}^{\alpha} w_{kl} \tilde{\nabla}^{\alpha} H_{ij}
}
\end{equation}

Recall $A_{\epsilon}^{pq}= (\chi^2+ \epsilon)u^{pq}$. Similar with (\ref{2.16}), we have
\begin{equation}\label{2.28}
{J_1\leq -\frac{1}{2}\sum_{|\alpha|= \tilde{p}} \int (\chi^2+ \epsilon)|\tilde{\nabla}^{\alpha+ 1} W|^2+ C(n, \tilde{p}, G, C_s(\Omega, g_0), ||\chi ||_1)|W|_{\tilde{p}}^2
}
\end{equation}

Similar with (\ref{2.17}), we get
\begin{equation}\nonumber
{J_2\leq 2\sum_{|\alpha|\leq \tilde{p}}|\tilde{\nabla}^{\alpha}W|_{L^2}\cdot [C(\int (\chi^2 + \epsilon)|\tilde{\nabla}^{\alpha+1}W|^2)^{\frac{1}{2}}+ |\tilde{\nabla}^{\alpha- 2}[\nabla^2 A_{\epsilon}^{pq}(W)_{pq}]|_{L^2}]
}
\end{equation}

and we have
\begin{equation}\nonumber
{|\tilde{\nabla}^{\alpha- 2}[\tilde{\nabla}^2 A_{\epsilon}^{pq}(W)_{pq}]|_{L^2} \leq C(\tilde{p})\sum_{\beta+ \gamma= \alpha}|\tilde{\nabla}^{\beta}(\tilde{\nabla}^2 A^{pq})\tilde{\nabla}^{\gamma}W|_{L^2}
}
\end{equation}
\begin{equation}\nonumber
{\leq C(n,\tilde{p},C_s(\Omega, g_0))\big[ (|\tilde{\nabla}^2 A_{\epsilon}^{pq}|_{L^{\infty}}+ 1) |\tilde{\nabla}^{\alpha}W|_{L^2}+ |W|_{L^{\infty}} ( |\tilde{\nabla}^{\alpha}(\tilde{\nabla}^2 A_{\epsilon}^{pq})|_{L^2} + 1) \big] 
}
\end{equation}
\begin{equation}\nonumber
{\leq C(n,\tilde{p},C_s(\Omega, g_0), ||\chi ||_2)[(|U|_{\tilde{p}+ 2}+ 1)|W|_{\tilde{p}}+ |W|_{\tilde{p}}(|U|_{\tilde{p}+ 2}+ 1)]
}
\end{equation}
\begin{equation}\nonumber
{\leq C(n,\tilde{p},C_s(\Omega, g_0), ||\chi ||_2, G)|W|_{\tilde{p}}
}
\end{equation}

Hence
\begin{equation}\label{2.29}
{J_2\leq \frac{1}{16}\sum_{|\alpha|= \tilde{p}}\int (\chi^2 + \epsilon)|\tilde{\nabla}^{\alpha+ 1}W|^2 + C(n,\tilde{p},C_s(\Omega, g_0), ||\chi ||_2, G)|W|_{\tilde{p}}^2
}
\end{equation}

Similar to (\ref{2.20}), we get 
\begin{equation}\label{2.30}
{J_3\leq C(\tilde{p},||\chi ||_{\tilde{p}},||\widetilde{Rm} ||_{\tilde{p}},C_s(\Omega, g_0),G) |W|_{\tilde{p}}^2
}
\end{equation}

In the rest of the proof, we use $C$ instead of $C(n,\tilde{p},||\chi ||_{\tilde{p}},||\widetilde{Rm} ||_{\tilde{p}},C_s(\Omega, g_0),G)$ for simplicity.
\begin{equation}\label{2.31}
{J_4= J_5+ J_6+ J_7
}
\end{equation}
where 
\begin{equation}\nonumber
{J_5= 2\sum_{\alpha}\int (\tilde{g}^{-1})^{\alpha+ 2}\tilde{\nabla}^{\alpha}W\cdot \tilde{\nabla}^{\alpha}[(\chi^2+ \epsilon)(U^{-1}- \tilde{U}^{-1})\tilde{\nabla}^2 \tilde{V}]
}
\end{equation}
\begin{equation}\nonumber
{J_6= 2\sum_{\alpha}\int (\tilde{g}^{-1})^{\alpha+ 2}\tilde{\nabla}^{\alpha}W\cdot \tilde{\nabla}^{\alpha}[(P(U)- P(\tilde{U}))\tilde{V}]
}
\end{equation}
\begin{equation}\nonumber
{J_7= 2\sum_{\alpha}\int (\tilde{g}^{-1})^{\alpha+ 2}\tilde{\nabla}^{\alpha}W\cdot \tilde{\nabla}^{\alpha} [B(U,\tilde{\nabla}U)- B(\tilde{U}- \tilde{\nabla}\tilde{U})]
}
\end{equation}
\begin{equation}\label{2.32}
\left.
\begin{array}{rl}
J_5 &\leq C|W|_{\tilde{p}} \Big(\sum_{\alpha}|\tilde{\nabla}^{\alpha}[(\chi^2+ \epsilon)\tilde{\nabla}^2 \tilde{V}\cdot (U^{-1}- \tilde{U}^{-1})]|_{L^2} \Big) \\
&\leq C|W|_{\tilde{p}} ( |\tilde{V}|_{\tilde{p}+ 2}+ 1) |U^{-1}- \tilde{U}^{-1}|_{\tilde{p}}\leq C|W|_{\tilde{p}}|\tilde{U}- U|_{\tilde{p}}
\end{array} \right.
\end{equation}

It is easy to get
\begin{equation}\label{2.33}
{\left.
\begin{array}{rl}
J_6 \leq C|W|_{\tilde{p}} \big[|P(U)- P(\tilde{U})|_{\tilde{p}} (|\tilde{V}|_{\tilde{p}}+ 1) \big] \leq C |W|_{\tilde{p}} |\tilde{U}- U|_{\tilde{p}}
\end{array} \right.
}
\end{equation}

We only need to estimate two terms of $J_7$ like the following $J_8$ and $J_9$, the others are similar.
\begin{equation}\nonumber
{J_8= \sum_{\alpha}\int (\tilde{g}^{-1})^{|\alpha|+ 2}\cdot \tilde{\nabla}^{\alpha}W\cdot \tilde{\nabla}^{\alpha}
[\chi^2(U^{-2}(\tilde{\nabla} U)^2- \tilde{U}^{-2}(\tilde{\nabla}\tilde{U})^2)]
}
\end{equation}
\begin{equation}\nonumber
{J_9= \sum_{\alpha}\int (\tilde{g}^{-1})^{|\alpha|+ 2}\cdot \tilde{\nabla}^{\alpha}W\cdot \tilde{\nabla}^{\alpha}
[\chi \tilde{\nabla}\chi (U^{-1}\tilde{\nabla} U- \tilde{U}^{-1}\tilde{\nabla}\tilde{U})]
}
\end{equation}

We firstly estimate $J_8$:
\begin{equation}\label{2.34}
{J_8\leq J_{10}+ J_{11}+ J_{12}+ J_{13}
}
\end{equation}

where
\begin{equation}\nonumber
{J_{10}= C\sum_{\alpha}\int (\tilde{g}^{-1})^{|\alpha|+ 2}\cdot \tilde{\nabla}^{\alpha}W\cdot \tilde{\nabla}^{\alpha}
[(U^{-1}- \tilde{U}^{-1}) U^{-1} (\tilde{\nabla}U)^2 \chi^2]
}
\end{equation}
\begin{equation}\nonumber
{J_{11}= C\sum_{\alpha}\int (\tilde{g}^{-1})^{|\alpha|+ 2}\cdot \tilde{\nabla}^{\alpha}W\cdot \tilde{\nabla}^{\alpha}
[(U^{-1}- \tilde{U}^{-1}) \tilde{U}^{-1} (\tilde{\nabla}U)^2 \chi^2]
}
\end{equation}
\begin{equation}\nonumber
{J_{12}= C\sum_{\alpha}\int (\tilde{g}^{-1})^{|\alpha|+ 2}\cdot \tilde{\nabla}^{\alpha}W\cdot \tilde{\nabla}^{\alpha}
[(\tilde{U}^{-2}\tilde{\nabla}U\chi^2) \tilde{\nabla}(U- \tilde{U})]
}
\end{equation}
\begin{equation}\nonumber
{J_{13}= C\sum_{\alpha}\int (\tilde{g}^{-1})^{|\alpha|+ 2}\cdot \tilde{\nabla}^{\alpha}W\cdot \tilde{\nabla}^{\alpha}
[(\tilde{U}^{-2}\tilde{\nabla}\tilde{U}\chi^2) \tilde{\nabla}(U- \tilde{U})]
}
\end{equation}

It is easy to get
\begin{equation}\label{2.35}
{J_{10}\leq C|W|_{\tilde{p}} \Big[|U^{-1}- \tilde{U}^{-1}|_{\tilde{p}} (|U^{-1}(\nabla U)^2|_{\tilde{p}}+ 1) \Big]
\leq C|W|_{\tilde{p}} |\tilde{U}- U|_{\tilde{p}} 
}
\end{equation}

Similarly, we get
\begin{equation}\label{2.36}
{J_{11}\leq C|W|_{\tilde{p}} |\tilde{U}- U|_{\tilde{p}} 
}
\end{equation}

We have $J_{12}\leq J_{12.1}+ J_{12.2}$, where
\begin{align}
J_{12.1}= C\sum_{\alpha}\int (\tilde{g}^{-1})^{|\alpha|+ 2}\cdot \tilde{\nabla}^{\alpha}W\cdot \tilde{\nabla}^{\alpha -1} [\tilde{\nabla}(\tilde{U}^{-2}\tilde{\nabla}U\chi^2) \tilde{\nabla}(U- \tilde{U})] \nonumber\\
J_{12.2}= C\sum_{\alpha}\int (\tilde{g}^{-1})^{|\alpha|+ 2}\cdot \tilde{\nabla}^{\alpha}W \cdot (\tilde{U}^{-2}\tilde{\nabla} U\chi^2) \cdot \tilde{\nabla}^{\alpha+ 1}(U- \tilde{U})  \nonumber
\end{align}
\begin{equation}\label{2.37}
{J_{12.1}\leq C|W|_{\tilde{p}} \sum_{\alpha}\big|\tilde{\nabla}^{\alpha- 1}[f\tilde{\nabla}g]\big|_{L^2} 
}
\end{equation}
where $f=\tilde{\nabla}(\tilde{U}^{-2} \chi^2 \tilde{\nabla}U)$, $g= U- \tilde{U}$, and note $|f|_{\tilde{p}}\leq C(|u|_{\tilde{p}+ 2}+ |\tilde{u}|_{\tilde{p}+ 2}+ 1)\leq C$.

Then
\begin{equation}\label{2.38}
\left.
\begin{array}{rl}
|\tilde{\nabla}^{\alpha- 1}[f\tilde{\nabla}g]|_{L^2} &\leq C\sum_{\beta+ \gamma= \alpha} |\tilde{\nabla}^{\beta}f \tilde{\nabla}^{\gamma}g|_{L^2} \\
&\leq C\big( (|f|_{L^{\infty}}+ 1) |\tilde{\nabla}^{\alpha} g|_{L^2}+ |g|_{L^{\infty}} (|\tilde{\nabla}^{\alpha}f|_{L^2}+ 1) \big)\\
&\leq C|g|_{\tilde{p}}= C|U- \tilde{U}|_{\tilde{p}}
\end{array} \right.
\end{equation}

With (\ref{2.37}) and (\ref{2.38}), we get
\begin{equation}\label{2.39}
{J_{12.1}\leq C|W|_{\tilde{p}} |U- \tilde{U}|_{\tilde{p}}
}
\end{equation}

It is easy to get
\begin{equation}\label{2.40}
\left.
\begin{array}{rl}
J_{12.2}\leq \frac{1}{1000} \sum_{|\alpha|= \tilde{p}}\int |\tilde{\nabla}^{\alpha +1}W|^2\chi^2 + C|U- \tilde{U}|_{\tilde{p}}^2 + C|W|_{\tilde{p}}|U- \tilde{U}|_{\tilde{p}}
\end{array} \right.
\end{equation}

by (\ref{2.39}) and (\ref{2.40}),
\begin{equation}\label{2.41}
{J_{12}\leq \frac{1}{1000} \sum_{|\alpha|= \tilde{p}}\int |\tilde{\nabla}^{\alpha +1}W|^2\chi^2 + C|U- \tilde{U}|_{\tilde{p}}^2 + C|W|_{\tilde{p}}|U- \tilde{U}|_{\tilde{p}}
}
\end{equation}

Similarly, 
\begin{equation}\label{2.42}
{J_{13}\leq \frac{1}{1000} \sum_{|\alpha|= \tilde{p}}\int |\tilde{\nabla}^{\alpha +1}W|^2\chi^2 + C|U- \tilde{U}|_{\tilde{p}}^2 + C|W|_{\tilde{p}}|U- \tilde{U}|_{\tilde{p}}
}
\end{equation}

By (\ref{2.35}), (\ref{2.36}), (\ref{2.41}), (\ref{2.42}) and (\ref{2.34}), 
\begin{equation}\label{2.43}
{J_8\leq \frac{1}{500} \sum_{|\alpha|= \tilde{p}}\int |\tilde{\nabla}^{\alpha +1}W|^2\chi^2 + C|U- \tilde{U}|_{\tilde{p}}^2 + C|W|_{\tilde{p}}|U- \tilde{U}|_{\tilde{p}}
}
\end{equation}
\begin{equation}\label{2.44}
{J_9\leq J_{14}+ J_{15}
}
\end{equation}
where
\begin{equation}\nonumber
{J_{14}= \sum_{\alpha}\int (\tilde{g}^{-1})^{|\alpha|+ 2}\cdot \tilde{\nabla}^{\alpha}W\cdot \tilde{\nabla}^{\alpha}
[(\chi \tilde{\nabla}\chi \tilde{\nabla}U) (U^{-1}- \tilde{U}^{-1})]
}
\end{equation}
\begin{equation}\nonumber
{J_{15}= \sum_{\alpha}\int (\tilde{g}^{-1})^{|\alpha|+ 2}\cdot \tilde{\nabla}^{\alpha}W\cdot \tilde{\nabla}^{\alpha}
[(\chi \tilde{\nabla}\chi \tilde{U}^{-1}) \tilde{\nabla}(U- \tilde{U})]
}
\end{equation}

It is easy to get
\begin{equation}\label{2.45}
{J_{14}\leq C|W|_{\tilde{p}}|U- \tilde{U}|_{\tilde{p}}
}
\end{equation}

Similar with the argument for $J_{12}$, we can get
\begin{equation}\label{2.46}
{J_{15}\leq \frac{1}{1000} \sum_{|\alpha|= \tilde{p}}\int |\tilde{\nabla}^{\alpha +1}W|^2\chi^2 + C|U- \tilde{U}|_{\tilde{p}}^2 + C|W|_{\tilde{p}}|U- \tilde{U}|_{\tilde{p}}
}
\end{equation}

By (\ref{2.45}), (\ref{2.46}) and (\ref{2.44}), we get 
\begin{equation}\label{2.47}
{J_9\leq \frac{1}{1000} \sum_{|\alpha|= \tilde{p}}\int |\tilde{\nabla}^{\alpha +1}W|^2\chi^2 + C|U- \tilde{U}|_{\tilde{p}}^2 + C|W|_{\tilde{p}}|U- \tilde{U}|_{\tilde{p}}
}
\end{equation}

By (\ref{2.43}) and (\ref{2.47}), we get  
\begin{equation}\label{2.48}
{J_7\leq \Big( \frac{9}{500} +\frac{6}{1000} \Big) \sum_{|\alpha|= \tilde{p}}\int |\tilde{\nabla}^{\alpha +1}W|^2\chi^2 + C|U- \tilde{U}|_{\tilde{p}}^2 + C|W|_{\tilde{p}}|U- \tilde{U}|_{\tilde{p}}
}
\end{equation}

By (\ref{2.32}), (\ref{2.33}), (\ref{2.48}) and (\ref{2.31}), we get 
\begin{equation}\label{2.49}
{J_4\leq \frac{3}{100}  \sum_{|\alpha|= \tilde{p}}\int |\tilde{\nabla}^{\alpha +1}W|^2\chi^2 + C|U- \tilde{U}|_{\tilde{p}}^2 + C|W|_{\tilde{p}}|U- \tilde{U}|_{\tilde{p}}
}
\end{equation}

By (\ref{2.28}), (\ref{2.29}), (\ref{2.30}), (\ref{2.49}), and (\ref{2.27}), we get  
\begin{equation}\label{2.50}
{\frac{\partial}{\partial t}(W|_{\tilde{p}}^2) \leq C_1 |W|_{\tilde{p}}^2+ C_2 |U- \tilde{U}|_{\tilde{p}}^2
}
\end{equation}
where $C_1= C_1(n,\tilde{p}, G,\Vert\chi\Vert_{\tilde{p}},\Vert\widetilde{Rm}\Vert_{\tilde{p}}, C_s(\Omega, g_0))> 0$ and $C_2$ depends on the same parameter as $C_1$. 

By Gronwall's inequality and (\ref{2.50}), note $|W(x,0)|_{\tilde{p}}= 0$, we get
\begin{equation}\label{2.51}
{|W|_{\tilde{p}}^2(t)\leq e^{C_1 t}\cdot \int_0^t C_2|U- \tilde{U}|_{\tilde{p}}^2(s) ds\leq C_2 t e^{C_1 t}|U- \tilde{U}|_{t, \tilde{p}}^2 
}
\end{equation}

We can choose $t_6= t_6(\delta, C_1, C_2)= t_6(n, \tilde{p}, G, ||\chi ||_{\tilde{p}}, ||\widetilde{Rm}||_{\tilde{p}}, C_s(\Omega, g_0), \delta)> 0$, such that $C_2\cdot t_6\cdot e^{C_1 t_6}= \delta^2< 1$, then we choose 
\begin{equation}\label{2.52}
{t_7= \min\{t_3, t_6 \}= t_7(n, \tilde{p}, G, ||\chi ||_{\tilde{p}+ 2}, ||\widetilde{Rm}||_{\tilde{p}+ 2}, C_s(\Omega , g_0), \delta)
}
\end{equation}
If $t\in [0, t_7]$, we get $|W|_{\tilde{p}}^2(t)\leq \delta^2 |U- \tilde{U}|_{t_7, \tilde{p}}^2$. Hence
\begin{equation}\nonumber
{|W|_{t_7, \tilde{p}}\leq \delta |U- \tilde{U}|_{t_7, \tilde{p}}
}
\end{equation}

That is 
\begin{equation}\nonumber
{|\Phi_{\epsilon}(U)-\Phi_{\epsilon}(\tilde{U})|_{t_7,\tilde{p}}\leq \delta |U- \tilde{U}|_{t_7, \tilde{p}}
}
\end{equation}

}
\qed

\begin{theorem}\label{thm 2.7}
{There exists $T= T(n, \tilde{p}, G, ||\chi ||_{2\tilde{p}+ 2}, ||\widetilde{Rm}||_{2\tilde{p}+ 2}, C_s(\Omega, g_0))>0$, such that on $\Omega\times [0, T]$, degenerate parabolic system (\ref{2.4}) has a smooth solution.
}
\end{theorem}

\pf
{Choose $m= 2\tilde{p}+ 2$ in Proposition \ref{prop 2.5} and $\delta= \frac{1}{2}$ in Proposition \ref{prop 2.6}, then we choose 
\[T= \min\{t_5, t_7 \}= T(n, \tilde{p}, G, ||\chi ||_{2\tilde{p}+ 2}, ||\widetilde{Rm}||_{2\tilde{p}+ 2}, C_s(\Omega, g_0)) \]

Set $U^{(0)}(x,t)= \tilde{g}(x)$, and $U_{\epsilon}^{(i+1)}= \Phi_{\epsilon}(U^{(i)})$ where $U_{\epsilon}^{(i)}$ are $(2,0)$-tensor on $\Omega\times [0, T]$, $i=0,1,2 \cdots$. Note $U^{(0)}\in \mathfrak A_T\cap \mathbb B_m(G)_T$ for any $m\geq \tilde{p}+ 2$. By Proposition \ref{prop 2.5}, we get $U_{\epsilon}^{(i)}\in \mathfrak A_T\cap \mathbb B_{2\tilde{p}+ 2}(G)_T$ for any $i$. Then we have the following:
\begin{equation}\label{2.53}
{\sup_{(x, t)\in \Omega\times [0, T]}\sup_{|\alpha|\leq \tilde{p}+ 2}|\tilde{\nabla}^{\alpha}U_{\epsilon}^{(i)}(x,t)| \leq C_1|U_{\epsilon}^{(i)}|_{T, 2\tilde{p}+ 2} \leq C_1G
}
\end{equation}
for any $i\geq 0$, where $C_1= C_1(C_s(\Omega, g_0), \tilde{p})$, and note $C_1$ is independent of $i$.

By (\ref{2.5}), we get
\begin{equation}\label{2.54}
{\sup_{|\beta |\leq \frac{\tilde{p}+ 2}{2}}\sup_{(x, t)\in \Omega\times [0, T]}|D_t^{\beta} U_{\epsilon}^{(i)}(x,t)| \leq C_2 
}
\end{equation}
for any $i\geq 0$, where $C_2= C_2(C_s(\Omega, g_0), n, \tilde{p}, G, \Vert\chi\Vert_{2}, \Vert\widetilde{Rm}\Vert_{0})$, note $C_2$ is independent of $i$.

Hence $\{U_{\epsilon}^{(i)}\}_{i=1}^{\infty}$ are uniformly bounded in $\mathbb B_{2\tilde{p}+ 2}(G)_{T}$, we can choose subsequence $U_{\epsilon}^{(k_i)}\rightarrow U_{\epsilon}^{\infty}$ in $C^{\tilde{p}+ 2, \frac{\tilde{p}+ 2}{2}}(\Omega\times [0, T])$ by Rellich-Kondrachov Compactness theorem, and $U_{\epsilon}^{\infty}\in \mathbb B_{\tilde{p}+ 2}(G)_{T}$. That is why we choose $m= 2\tilde{p}+ 2$ in Proposition \ref{prop 2.5} at the beginning of the proof. For $k_i\geq m> 0$, we get
\begin{equation}\nonumber
{|U_{\epsilon}^{(m)}- U_{\epsilon}^{(k_i)}|_{T, \tilde{p}}= |\Phi_{\epsilon}^m(U^{(0)})- \Phi_{\epsilon}^{k_i}(U^{(0)})|_{T, \tilde{p}}\leq \delta^m|U^{(0)}- \Phi_{\epsilon}^{k_i- m}(U^{(0)})|_{T, \tilde{p}}\leq 2(\frac{1}{2})^m G
}
\end{equation}
Hence
\begin{equation}\nonumber
{\lim_{m\rightarrow \infty}|U_{\epsilon}^{(m)}- U_{\epsilon}^{\infty}|_{T, \tilde{p}}\leq \lim_{m\rightarrow \infty}\Big( |U_{\epsilon}^{(m)}- U_{\epsilon}^{(k_i)}|_{T, \tilde{p}}+ |U_{\epsilon}^{(k_i)}- U_{\epsilon}^{\infty}|_{T, \tilde{p}} \Big)= 0
}
\end{equation}

Now 
\begin{align}
& |U_{\epsilon}^{\infty}- \Phi_{\epsilon}(U_{\epsilon}^{\infty})|_{T, \tilde{p}}= \lim_{m\rightarrow \infty}|U_{\epsilon}^{(m)}- \Phi_{\epsilon}(U_{\epsilon}^{\infty})|_{T, \tilde{p}} \nonumber \\
&\quad \quad = \lim_{m\rightarrow \infty}|\Phi_{\epsilon}(U_{\epsilon}^{(m- 1)})- \Phi_{\epsilon}(U_{\epsilon}^{\infty})|_{T, \tilde{p}} \leq \frac{1}{2} \lim_{m\rightarrow \infty}|U_{\epsilon}^{(m- 1)}- U_{\epsilon}^{\infty}|_{T, \tilde{p}}= 0 \nonumber 
\end{align}

then we have
\begin{equation}\nonumber
{|U_{\epsilon}^{\infty}- \Phi_{\epsilon}(U_{\epsilon}^{\infty})|_{C^{0}}\leq C_s(\Omega , g_0)|U_{\epsilon}^{\infty}- \Phi_{\epsilon}(U_{\epsilon}^{\infty})|_{T, \tilde{p}}= 0
}
\end{equation}

hence
\begin{equation}\label{2.55}
{U_{\epsilon}^{\infty}= \Phi_{\epsilon}(U_{\epsilon}^{\infty})
}
\end{equation}

By (\ref{2.53}) and (\ref{2.54}), we get $\{U_{\epsilon}^{\infty}\}_{0< \epsilon\leq 1}$ is uniformly bounded in $C^{\tilde{p}+ 2, \frac{\tilde{p}+ 2}{2}}(\Omega\times [0, T])$. So there exists subsequence $\{U_{\epsilon_i}^{\infty}\}_{i=1}^{\infty}$, such that
\begin{equation}\nonumber
{\lim_{i\rightarrow \infty}U_{\epsilon_i}^{\infty}= \hat{U}^{\infty} \quad in \quad C^{\tilde{p}, \frac{\tilde{p}}{2}}(\Omega\times [0, T])
}
\end{equation}
where $\lim_{i\rightarrow \infty}\epsilon_i= 0$.

On the other hand, by (\ref{2.55}) and (\ref{2.5}), we have 
\begin{equation}\label{2.56}
{(U_{\epsilon_i}^{\infty})_t= (\chi^2 + \epsilon_i)(U_{\epsilon_i}^{\infty})^{\alpha \beta}\tilde{\nabla}_{\alpha}\tilde{\nabla}_{\beta}(U_{\epsilon_i}^{\infty})+ D(U_{\epsilon_i}^{\infty})+ B(U_{\epsilon_i}^{\infty}, \tilde{\nabla} U_{\epsilon_i}^{\infty})
}
\end{equation}

Let $i\rightarrow \infty$ in (\ref{2.56}), we get 
\begin{equation}\nonumber
{(\hat{U}^{\infty})_t= \chi^2(\hat{U}^{\infty})^{\alpha \beta}\tilde{\nabla}_{\alpha}\tilde{\nabla}_{\beta}(\hat{U}^{\infty})+ D(\hat{U}^{\infty})+ B(\hat{U}^{\infty}, \tilde{\nabla}\hat{U}^{\infty})
}
\end{equation}

Then $\hat{U}^{\infty}$ is a solution of (\ref{2.4}) on $\Omega\times [0, T]$, and $\hat{U}^{\infty}\in C^{\tilde{p}, \frac{\tilde{p}}{2}}(\Omega\times [0, T])$. By regularity theory of parabolic system (note we have local regularity property where $\chi(x)> 0$, and $\hat{U}^{\infty}(x, t)= \tilde{g}_{ij}(x)$ where $\chi(x)= 0$), we have
\[\hat{U}^{\infty}\in C^{\infty}(\Omega\times [0, T])\]
 
Hence $\hat{U}^{\infty}$ is a smooth solution of (\ref{2.4}) on $\Omega\times [0, T]$.
}
\qed


\section{Evolution equations of curvature under the local Ricci flow}
Under local Ricci flow, the behavior of curvature tensors are affected by their evolution equations. These equations have simpler versions in Ricci flow context, just taking the cut-off function $\chi (x)\equiv 1$.

Now we calculate the evolution equation of the curvature tensor $Rm$. We use the following notations: $R^{l}_{ijk}= g^{lp}R_{ijkp}$, $R^r_i= g^{rp} R_{ip}$, and $R$ is the scalar curvature. 

\begin{lemma} \label{lem 3.1}
{Under the local Ricci flow (\ref{2.1}),
\begin{equation}\label{3.2}
{\frac{\partial}{\partial t} R_{ijkl}= \chi^2\Delta R_{ijkl}+ Q+ I_1+ I_2+ I_3 
}
\end{equation}
where $Q$, $I_1$, $I_2$, $I_3$ are defined in (\ref{3.9}), (\ref{3.5}),(\ref{3.6}),(\ref{3.7}).
}
\end{lemma}

\pf
{Firstly if $\frac{\partial}{\partial t}g_{ij}=h_{ij}$, then (see \cite{RI})
\begin{equation}\nonumber
{\frac{\partial}{\partial t}R_{ijk}^{l}=\frac{1}{2}g^{lp}
[\nabla_i \nabla_k h_{jp} + \nabla _j \nabla_p h_{ik} - \nabla_i \nabla_p h_{jk} - \nabla_j \nabla_k h_{ip}    -R_{ijk}^{q} h_{qp} -R_{ijp}^{q} h_{kq}]
}
\end{equation}

In local Ricci flow case,
\begin{equation}\nonumber
{h_{ij}=-2\chi^2 R_{ij}. 
}
\end{equation}

Now 
\begin{equation}\label{3.3}
{ \frac{\partial}{\partial t} R_{ijkl}= (-2\chi^2 R_{lm})R_{ijk}^m + g_{lm} (\frac{\partial}{\partial t} R_{ijk}^m).
}
\end{equation}

We calculate $g_{lm}\frac{\partial}{\partial t}R_{ijk}^m$ firstly:
\begin{equation}\nonumber
{\left.
\begin{array}{rl} 
g_{lm}\frac{\partial}{\partial t}R_{ijk}^m = \frac{1}{2}[\nabla_i \nabla_k h_{jl} + \nabla _j \nabla_l h_{ik} - \nabla_i \nabla_l h_{jk} - \nabla_j \nabla_k h_{il} -R_{ijk}^{q} h_{ql} -R_{ijl}^{q} h_{kq}]. 
\end{array}\right.
}
\end{equation}

We calculate $\nabla_i \nabla_k h_{jl}$ as a sample, using the notation $\chi_i=\nabla_i \chi$ and $\chi_{ij}= \nabla_i \nabla_j \chi$:
\begin{equation}\nonumber
{\left.
\begin{array}{rl} 
\nabla_i \nabla_k h_{jl}= -2(\chi^2 \nabla_i \nabla_k R_{jl} + 2\chi \chi_i \nabla_k R_{jl}+ 2\chi \chi_k \nabla_i R_{jl}+ 2\chi_i \chi_k R_{jl} + 2\chi \chi_{ik} R_{jl})
\end{array}\right.
}
\end{equation}

A straightforward computation gives
\begin{equation}\label{3.4}
{\left.
\begin{array}{rl} 
g_{lm}\frac{\partial}{\partial t}R_{ijk}^m&= \chi^2( \nabla_i \nabla_l R_{jk} + \nabla _j \nabla_k R_{il} - \nabla_i \nabla_k R_{jl} - \nabla_j \nabla_l R_{ik} )\\
&\quad + \chi^2 (R_{ijk}^q R_{ql} + R_{ijl}^q R_{kq})+ I_1+ I_2+ I_3
\end{array}\right.
}
\end{equation}

\noindent
where 
\begin{align}
I_1 &= 2\chi [ (\chi_k \nabla_j R_{il} + \chi_j \nabla_k R_{il} + \chi_l \nabla_i R_{jk} \label{3.5}\\
    &+ \chi_i \nabla_l R_{jk}) - (\chi_k \nabla_i R_{jl} + \chi_i \nabla_k R_{jl} + \chi_l \nabla_j R_{ik} + \chi_j \nabla_l R_{ik})] \nonumber\\
I_2 &= 2[ (\chi_i \chi_l R_{jk} + \chi_j \chi_k R_{il}) - (\chi_i \chi_k R_{jl} + \chi_j \chi_l R_{ik})] \label{3.6}\\
I_3 &= 2\chi [(\chi_{il}R_{jk} + \chi_{jk}R_{il}) - (\chi_{ik}R_{jl} + \chi_{jl}R_{ik})]. \label{3.7}
\end{align}

Now we compute $\Delta R_{ijkl}$. From the formula for $\Delta R_{ijk}^l$ in \cite{RI}, 
\begin{equation}\label{3.8}
{\left.
\begin{array}{rl} 
& \Delta R_{ijkl}= \Delta (g_{lm}R_{ijk}^m) = g_{lm} \Delta R_{ijk}^m \\
&= ( \nabla_i \nabla_l R_{jk} + \nabla _j \nabla_k R_{il} - \nabla_i \nabla_k R_{jl} - \nabla_j \nabla_l R_{ik} ) \\
&\quad + (R_j^r R_{irkl} - R_i^r R_{jrkl})- g^{pq}(R_{ijp}^r R_{rqkl} -2R_{pik}^r R_{jqrl} + 2R_{jqk}^r R_{pirl})
\end{array}\right.
}
\end{equation}

By (\ref{3.3}), (\ref{3.4}) and (\ref{3.8}), 
\begin{equation}\nonumber
{\frac{\partial}{\partial t} R_{ijkl}= \chi^2\Delta R_{ijkl}+ Q+ I_1+ I_2+ I_3,
}
\end{equation}

\noindent
where 
\begin{align}
Q  &= Q_1 + Q_2 \label{3.9}\\
Q_1&= \chi^2 g^{pq}(R_{ijp}^r R_{rqkl} -2R_{pik}^r R_{jqrl} + 2R_{jqk}^r R_{pirl}) \label{3.10}\\
Q_2&= -\chi^2 (R_{mjkl} R_i^m + R_{imkl} R_j^m + R_{ijml} R_k^m + R_{ijkm} R_l^m). \label{3.11}
\end{align}
}
\qed

For our applications, we need to estimate $\frac{\partial}{\partial t}|Rm|$.
 
\begin{lemma} \label{lem 3.2}
{Under the local Ricci flow (\ref{2.1}), we have
\begin{align}
\frac{\partial}{\partial t}|Rm| &\leq \chi^2 \Delta |Rm| + \Big(\frac{\chi^2 \big|\nabla |Rm|\big|^2}{|Rm|}\Big)- (1- \epsilon)\Big(\frac{\chi^2 |\nabla Rm|^2}{|Rm|}\Big) \label{3.12}\\
& + 10\chi^2 |Rm|^2 + \Big(\frac{10^5}{\epsilon}\Big) n|\nabla \chi|^2 |Rm| + \Big(\frac{8\chi g^{ri}g^{sj}g^{pk}g^{ql} R_{rspq} \chi_{il} R_{jk}}{|Rm|}\Big)\nonumber
\end{align}
where $\epsilon$ is any positive constant satisfying $0< \epsilon< 1$.
}
\end{lemma}

\pf
{\begin{equation}\nonumber
{\frac{\partial}{\partial t} |Rm|^2= 2g^{ri}g^{sj}g^{pk}g^{ql} R_{rspq} \Big(\frac{\partial}{\partial t} R_{ijkl} \Big) + \frac{\partial}{\partial t} (g^{ri}g^{sj}g^{pk}g^{ql})R_{rspq} R_{ijkl}
}
\end{equation}

Note 
\begin{equation}\nonumber
{\left.
\begin{array}{rl}
\frac{\partial}{\partial t} (g^{ri}) g^{sj}g^{pk}g^{ql}R_{rspq} R_{ijkl}= 2\chi^2 g^{ri}g^{sj}g^{pk}g^{ql} R_{rspq} R_{mjkl} R_{i}^m
\end{array}\right.
}
\end{equation}

Similarly, 
\begin{equation}\nonumber
{\left.
\begin{array}{rl}
\frac{\partial}{\partial t} (g^{ri}g^{sj}g^{pk}g^{ql})R_{rspq} R_{ijkl}= -2 g^{ri}g^{sj}g^{pk}g^{ql} R_{rspq}Q_2
\end{array}\right.
}
\end{equation}

Hence
\begin{equation}\nonumber
{\left.
\begin{array}{rl}
\frac{\partial}{\partial t} |Rm|^2= \chi^2 [\Delta (|Rm|^2) - 2|\nabla Rm|^2] + 2g^{ri}g^{sj}g^{pk}g^{ql} R_{rspq}[Q_1+ I_1+ I_2+ I_3]
\end{array}\right.
}
\end{equation}

Let $f= |Rm|$,
\begin{equation}\nonumber
{\frac{\partial}{\partial t} (f^2)= \chi^2 \Delta (f^2) - 2\chi^2 |\nabla Rm|^2 + 2g^{ri}g^{sj}g^{pk}g^{ql} R_{rspq}[Q_1+ I_1+ I_2+ I_3]
}
\end{equation}

Firstly we have the following estimate:
\begin{equation}\nonumber
{2g^{ri}g^{sj}g^{pk}g^{ql} R_{rspq}Q_1\leq 10\chi^2 f^3.
}
\end{equation}

It is easy to get
\begin{equation}\nonumber
{2g^{ri}g^{sj}g^{pk}g^{ql} R_{rspq}I_1 \leq 32\sqrt{n} \, \chi |\nabla \chi| |\nabla Rm|f\leq 2\epsilon(\chi |\nabla Rm|)^2 + (\frac{10^4}{\epsilon}) n|\nabla \chi|^2 f^2 ;
}
\end{equation}

also
\begin{equation}\nonumber
{2g^{ri}g^{sj}g^{pk}g^{ql} R_{rspq}I_2\leq 16\sqrt{n}\, |\nabla \chi|^2 f^2
}
\end{equation}

We can use the normal coordinate system to simplify: 
\begin{equation}\nonumber
{\left.
\begin{array}{rl}
2g^{ri}g^{sj}g^{pk}g^{ql} R_{rspq}I_3= 16\chi \sum_{ijkl}(\chi_{il} R_{jk} R_{ijkl})= 16\chi g^{ri}g^{sj}g^{pk}g^{ql} R_{rspq} \chi_{il} R_{jk}
\end{array}\right.
}
\end{equation}

Using normal coordinates it is also easy to see 
\begin{equation} \nonumber
{\big|\nabla |Rm| \big|\leq |\nabla Rm|.
}
\end{equation}

Then 
\begin{equation}\nonumber
{\left.
\begin{array}{rl}
\frac{\partial}{\partial t}(f^2) & \leq 2f\chi^2 \Delta f+ 2\chi^2|\nabla f|^2- (2- 2\epsilon)\chi^2 |\nabla Rm|^2 + 10\chi^2 f^3 \\
& \quad + (\frac{10^5}{\epsilon}) n|\nabla \chi|^2 f^2 + 16\chi g^{ri}g^{sj}g^{pk}g^{ql} R_{rspq} \chi_{il} R_{jk},
\end{array}\right.
}
\end{equation}

Finally
\begin{equation}\nonumber
{\left.
\begin{array}{rl}
\frac{\partial}{\partial t}f &\leq \chi^2 \Delta f + (\frac{\chi^2 |\nabla f|^2}{f})- (1- \epsilon)(\frac{\chi^2 |\nabla Rm|^2}{f}) \\
& \quad + 10\chi^2 f^2 + (\frac{10^5}{\epsilon}) n|\nabla \chi|^2 f + \Big(\frac{8\chi g^{ri}g^{sj}g^{pk}g^{ql} R_{rspq} \chi_{il} R_{jk}}{f} \Big).
\end{array}\right.
}
\end{equation}

Replace $f$ with $|Rm|$, we get our conclusion.
}
\qed

Using the formula in Chapter 3 of \cite{RI}, we can get the following formula (because the calculation is very similar with the reduction of Lemma \ref{lem 3.1}, we omit the calculation here, just give the formula):

\begin{lemma}\label{lem 3.3}
{Under the local Ricci flow (\ref{2.1}), 
\begin{align}
\frac{\partial}{\partial t}R_{ij} &=\chi^2 \Delta (R_{ij})+ J_1+ J_2+ J_3+ J_4, \label{3.12.1}\\
where \quad J_1 &= 2\chi g^{pq}\chi_p\nabla_q R_{ij}+ \chi(\chi_i\nabla_j R+ \chi_j\nabla_i R) \nonumber\\
 &- 2\chi\chi_p g^{pq}(\nabla_i R_{qj}+ \nabla_j R_{qi}), \nonumber \\
J_2 &=-2g^{pq}\chi(\chi_{ip}R_{qj}+ \chi_{jp}R_{qi}) + 2\chi\chi_{ij}R, \label{3.13}\\
J_3 &=-2g^{pq}(\chi_i\chi_p R_{qj}+ \chi_j\chi_p R_{qi})+ 2\chi_i\chi_j R,  \nonumber\\
J_4 &= 2\chi^2 g^{pq}(R_{qij}^r R_{rp}- R_{ip}R_{qj}). \nonumber
\end{align}
}
\end{lemma}

Imitating the reduction of Lemma \ref{lem 3.2}, we can also get the following theorem about $|Rc|$:

\begin{lemma} \label{lem 3.4}
{Under the local Ricci flow (\ref{2.1}),
\begin{align}
\frac{\partial}{\partial t}|Rc| &\leq \chi^2 \Delta |Rc| + \Big(\frac{\chi^2 \big|\nabla |Rc| \big|^2}{|Rc|}\Big)- (1- \epsilon)\Big(\frac{\chi^2 |\nabla Rc|^2}{|Rc|}\Big) \label{3.15}\\
& + (\frac{10^5}{\epsilon})n|\nabla \chi|^2 |Rc| + (2+ \sqrt{n})\chi^2 |Rc| |Rm|+ \Big(\frac{g^{ik}g^{jl}R_{kl} J_2}{|Rc|}\Big) \nonumber
\end{align}
where $J_2$ in the last line comes from (\ref{3.13}), and $\epsilon$ is any positive constant satisfying $0< \epsilon< 1$.
}
\end{lemma}

\pf
{Similar to the proof of Lemma \ref{lem 3.2}.
}
\qed

If $T_1$ and $T_2$ are two tensors, we use $T_1*T_2$ to denote all possible tensor product between them.

\begin{lemma}\label{lem W1}
{\begin{equation}\nonumber
{\frac{\partial}{\partial t}(\nabla^m Rm)= \chi^2 \Delta(\nabla^m Rm)
+ \sum_{i+ j+ k= m+2, \ k\leq m+1}\nabla^i\chi *\nabla^j\chi *\nabla^k Rm 
}
\end{equation}
\begin{equation}\nonumber
{\quad + \sum_{i+ j+ k= m}\nabla^i (\chi^2) *\nabla^j Rm* \nabla^k Rm
}
\end{equation}
}
\end{lemma}

\pf
{By induction, when $m= 0$, it is true from Lemma \ref{lem 3.1}. Suppose the conclusion holds for $m -1$, then we have:
\begin{align}
\frac{\partial}{\partial t}(\nabla \nabla^{m-1}Rm) &=\nabla^{m- 1}Rm* \nabla(\chi^2 Rc)+ \nabla(\chi^2\Delta \nabla^{m- 1}Rm) \nonumber \\
 &+ \nabla\Big[\sum_{i+ j+ k= m+ 1, \ k\leq m}\nabla^{i}\chi *\nabla^{j} \chi* \nabla^{k}Rm  \Big] \nonumber \\
& + \nabla\Big[\sum_{i+ j+ k= m- 1}\nabla^{i}(\chi^2) *\nabla^{j} Rm* \nabla^{k}Rm  \Big] \nonumber
\end{align}
And 
\begin{align}
\nabla\Big[\chi^2 \Delta \nabla^{m- 1}Rm \Big] &= \chi^2 \Delta \nabla^{m}Rm+ \chi^2\nabla^{m- 1}Rm* \nabla Rm \nonumber \\
&\quad + \chi^2 \nabla^m Rm* Rm+ \nabla(\chi^2)* \Delta \nabla^{m- 1}Rm\nonumber
\end{align}

We can consider $\nabla^{m- 1}Rm* \nabla(\chi^2 Rc)$, $\chi^2 \nabla^{m- 1}Rm* \nabla Rm$, $\chi^2 \nabla^{m}Rm* Rm$ as the terms in $\nabla\Big[\sum_{i+ j+ k= m- 1}\nabla^i(\chi^2)*\nabla^j Rm*\nabla^k Rm \Big]$ and $\nabla(\chi^2)*\Delta \nabla^{m- 1}Rm $ as one of the terms in $\nabla\Big[\sum_{i+ j+ k= m+ 1, k\leq m} \nabla^i \chi *\nabla^j \chi* \nabla^{k}Rm \Big]$, then we get our equation.
}
\qed

Similarly, we have
\begin{lemma}\label{lem W2}
{\begin{align}
\frac{\partial}{\partial t}\nabla^m (\chi^2 Rm) &= \chi^2 \Delta\nabla^m (\chi^2 Rm) \nonumber \\
 &+ \sum_{i+ j+ k= m+ 2,\ k\leq m+ 1}\nabla^i \chi* \nabla^j \chi* \nabla^k (\chi^2 Rm) \nonumber \\
&+ \sum_{i+ j+ k= m,\ k\leq m- 1}\nabla^i (\chi^2)* \nabla^j (\chi^2 Rm)* \nabla^k Rm+ \chi^2 Rm* \nabla^m (\chi^2 Rm)  \nonumber
\end{align}
and 
\[\frac{\partial}{\partial t}|\nabla^m \chi|\leq C(n, k)\sum_{k\leq m- 1} |\nabla^k (\chi^2 Rc)|\cdot |\nabla^{m- k}\chi|\]
}
\end{lemma}


\section{Extension of the local Ricci flow}

Firstly we prove the following ODE comparison lemmas, which will be used in the later proofs of this section.

\begin{lemma}\label{lem W3}
{Assume $f(t)$ and $g(t)$ are nonnegative functions, $t\in [0, T)$, $b$ is a nonnegative constant, and 
\[\frac{d f}{dt}\leq g+ bf\]
Then on $[0, T)$:
\[f(t)\leq \Big[f(0)+ \int_0^t e^{-bs}g(s) ds\Big]e^{bt} \]
}
\end{lemma}

\pf
{Let $h(t)= \Big[f(0)+ \int_0^t e^{-bs}g(s) ds\Big]e^{bt}$. Then $h(0)= f(0)$, and $h'(t)\geq f'(t)$ by assumption. Hence $h(t)\geq f(t)$.
}
\qed

\begin{lemma}\label{lem W4}
{Assume $f(t)$ is a nonnegative function, $t\in [0, T)$, $a$, $b$, $c$ are nonnegative constants, and 
\[\frac{d f}{dt}\leq a\int_0^t f(s)ds+ bf(t)+ c\]
Then there exist positive constants $w$, $k$ which depend on $a$, $b$, $c$ and $f(0)$ such that:
\[f(t)\leq w e^{kt}, \quad t\in [0, T). \]
}
\end{lemma}

\pf
{It suffices to choose $w$, $k$ large enough to ensure:
\[(w e^{kt})'> a\Big[\int_0^t (w e^{ks})ds\Big]+ b(w e^{kt})+ c \quad and \quad w> f(0) \]  
to get the comparison function.
}
\qed

We define $|\nabla \chi |_{(g(t), \infty)}= \max_{x\in \Omega} |\nabla\chi (x)|_{g(t)}$. Recall that we have defined $C_s(\Omega, g(t))$ in (\ref{2.7}), for simplicity we use the notation $C_s(\Omega, t)$ to replace it.  

\begin{lemma}\label{lem W5}
{If $|\chi^2 Rm|_{(g(t), \infty)}$ is uniformly bounded along a given solution $g(t)$ of local Ricci flow, $t\in [0, T)$, $0< T< +\infty$, then $|\nabla\chi|_{(g(t), \infty)}$ and the Sobolev constant $C_s(\Omega, t)$ are uniformly bounded on $[0, T)$.
}
\end{lemma}

\pf
{Because $\Omega$ is a bounded domain in $(M,g_0)$, by the proof of Theorem $3.5$ in \cite{NAM}, we get $C_s(\Omega,0)$ is finite. Then one only need to notice that the metrics $\{g(t)\}_{t\in [0, T)}$ are uniformly equivalent to each other by assumption.
}
\qed

\begin{lemma}\label{lem W6}
{Assume $g(t)$, $t\in [0, T)$, $0< T< +\infty$, is a smooth solution to local Ricci flow on $(M^n, g_0)$, $|\chi^2 Rm|_{(g(t), \infty)}$ is uniformly bounded on $[0, T)$. Then for any $p> 100$, $|Rm|_{(g(t), L^p(\Omega))}\triangleq \int_{\Omega} |Rm|_{g(t)}^p dV_{g(t)}$ is uniformly bounded on $[0, T)$. 
}
\end{lemma}

\pf
{Using Lemma \ref{lem 3.1}, 
\begin{align}
\frac{\partial}{\partial t}\Big(\int_{\Omega} |Rm|_{g(t)}^p dV_{g(t)} \Big)= p (I_1+ I_2+ I_3+ I_4+ I_5) \nonumber 
\end{align}
where 
\[ I_1=- \int_{\Omega}\chi^2 |\nabla Rm|^2 |Rm|^{p- 2}, \quad I_2= \int_{\Omega} |Rm|^{p- 2}\langle \chi\nabla \chi *\nabla Rm, Rm\rangle, \] 
\[I_3= \int_{\Omega} |Rm|^{p- 2}\langle \chi^2 Rm*Rm, Rm\rangle, \quad I_4= \int_{\Omega} |Rm|^{p- 2}\langle \nabla \chi *\nabla \chi *Rm, Rm\rangle, \]
\[I_5= \int_{\Omega} |Rm|^{p- 2}\langle \chi \nabla^2 \chi *Rm, Rm\rangle  \]

And 
\[I_2\leq C\int_{\Omega}|Rm|^p+ \frac{1}{16}\int_{\Omega} \chi^2 |\nabla Rm|^2 |Rm|^{p- 2} ,\]
\[I_3\leq \max_{t\in [0, T)} |\chi^2 Rm|_{(g(t), \infty)}\cdot \int_{\Omega} |Rm|^p\leq C\int_{\Omega}|Rm|^p ,\]
\[I_4\leq C\max_{t\in [0, T)} |\nabla\chi|_{(g(t), \infty)}\int_{\Omega}|Rm|^p\leq C\int_{\Omega}|Rm|^p\]
Applying integration by parts and Lemma \ref{lem W5} yields:
\begin{align}
I_5=\int_{\Omega} |Rm|^{p- 2}\langle \chi \nabla^2 \chi *Rm, Rm\rangle \leq C \int_{\Omega} |Rm|^{p}+ \frac{1}{16} \int_{\Omega} \chi^2 |\nabla Rm|^2 |Rm|^{p-2} \nonumber
\end{align}
Then
\begin{align}
\frac{\partial}{\partial t}\int_{\Omega} |Rm|^p\leq C\int_{\Omega} |Rm|^p \nonumber
\end{align}
we get 
\begin{align}
\int_{\Omega} |Rm|_{g(t)}^p dV_{g(t)}\leq e^{CT}\int_{\Omega} |Rm|_{g_0}^p dV_{g_0} \nonumber
\end{align}
Hence $|Rm|_{(g(t), L^p(\Omega))}$ are uniformly bounded on $[0, T)$.
}
\qed

\begin{prop}\label{prop W7}
{If $g(t)$ is a smooth solution of the local Ricci flow on $(M^n, g_0)$, $t\in [0, T)$, $0< T< +\infty$ and $m$ is a nonnegative integer. Assume on $[0, T)$, for any $q> 100$, $|\nabla^{i} \chi|_{(g(t), L^q(\Omega))}$ are uniformly bounded, where $0\leq i\leq m+ 1$; and $|\chi^2 Rm|_{(g(t), \infty)}$ are uniformly bounded. Then for any $p> 100$, $|\nabla^m Rm|_{(g(t), L^p(\Omega))}$ is uniformly bounded on $[0, T)$.
}
\end{prop}

\pf
{By induction. When $m= 0$, the conclusion follows from Lemma \ref{lem W6}. For $m\geq 1$, assume that for any $q> 100$, $\int_{\Omega} |\nabla^k Rm|_{g(t)}^q dV_{g(t)}$ is uniformly bounded on $[0, T)$, where $0\leq k\leq m- 1$. For simplicity, we use $\int$ instead of $\int_{\Omega}$ in the rest of the proof. 

Use Lemma \ref{lem W1}, 
\begin{equation}\label{w8}
{\begin{array}{rl}
\frac{\partial}{\partial t}\int |\nabla^m Rm|_{g(t)}^{p} dV_{g(t)} \leq p(I_1+ I_2+ I_3)+ C\int |\nabla^m Rm|_{g(t)}^{p} dV_{g(t)}
\end{array}
}
\end{equation}
where 
\begin{equation}\nonumber
I_1= \int |\nabla^m Rm|^{p- 2} \langle \chi^2 \Delta \nabla^m Rm, \nabla^m Rm \rangle \ ,
\end{equation}
\begin{equation}\nonumber
I_2= \int |\nabla^m Rm|^{p- 2} \sum_{i+ j+ k= m+ 2,\ k\le m+ 1}\langle \nabla^i\chi *\nabla^j \chi *\nabla^k Rm, \nabla^m Rm \rangle \ ,
\end{equation}
\begin{equation}\nonumber
I_3= \int |\nabla^m Rm|^{p- 2} \sum_{i+ j+ k= m}\langle \nabla^i(\chi^2) *\nabla^j(Rm) *\nabla^k Rm, \nabla^m Rm \rangle 
\end{equation}

Use integration by parts and H\"older inequality,
\begin{align}
I_2&\leq C\int |\nabla^m Rm|^{p- 1}\cdot |\nabla^{m+ 1}Rm| \cdot |\nabla \chi | \cdot |\chi |+ C\int |\nabla^{m} Rm|^{p}\cdot |\nabla\chi |^2 \nonumber \\
&\quad + C\sum_{i+ j+ k= m+ 2,\ 1\leq k\leq m- 1} \Big[ \Big(\int |\nabla^m Rm|^p\Big)^{\frac{p- 1}{p}}\cdot \Big(\int |\nabla^k Rm|^{3p}\Big)^{\frac{1}{3p}} \nonumber \\
&\quad \quad \cdot \Big(\int |\nabla^i \chi|^{3p}\Big)^{\frac{1}{3p}} \cdot \Big(\int |\nabla^j \chi|^{3p}\Big)^{\frac{1}{3p}}\Big]+ I_4 \nonumber \\
& + C\sum_{i+ j= m+ 2,\ i,\ j\leq m+ 1} \Big[ \Big(\int |\nabla^m Rm|^p\Big)^{\frac{p- 1}{p}}\cdot \Big(\int |Rm|^{3p}\Big)^{\frac{1}{3p}} \nonumber \\
&\quad \quad \quad \cdot \Big(\int |\nabla^i \chi|^{3p}\Big)^{\frac{1}{3p}} \cdot \Big(\int |\nabla^j \chi|^{3p}\Big)^{\frac{1}{3p}}\Big]  \nonumber  
\end{align}

Use Schwartz inequality, 
\begin{equation}\label{w8.0}
{\left.
\begin{array}{rl}
I_2 &\leq \frac{1}{16} \int \chi^2  |\nabla^{m+ 1} Rm|^2 |\nabla^m Rm|^{p- 2} + C\int |\nabla^m Rm|^{p}   \\
&\quad + C|\nabla^m Rm|_{(g(t), L^{p}(\Omega))}^{p- 1}+ I_4 
\end{array}\right.
}
\end{equation}
where $I_4\triangleq C \int |\nabla^m Rm|^{p- 2} \cdot \langle \chi \nabla^{m+ 2}\chi *Rm, \nabla^m Rm \rangle$. Use integration by parts on $I_4$, 
\begin{align}
I_4&\leq C\int \chi |\nabla^m Rm|^{p- 2} \cdot |\nabla^{m+ 1}Rm| \cdot |\nabla^{m+ 1}\chi |\cdot |Rm| \nonumber \\
&\quad + C\int |\nabla^m Rm|^{p- 1} \cdot |\nabla^{m+ 1}\chi | \cdot |\nabla\chi |\cdot |Rm|+ I_5
 \nonumber \\
&\leq C \Big(\int \chi^2 |\nabla^m Rm|^{p- 2}|\nabla^{m+ 1} Rm|^2\Big)^{\frac{1}{2}}\cdot \Big(\int |\nabla^m Rm|^p\Big)^{\frac{p- 1}{2p}}+ C|\nabla^m Rm|_{(g(t), L^p(\Omega))}^{p- 1}+ I_5 \nonumber
\end{align}
in the last inequality, we used the assumption of the proposition and Lemma \ref{lem W6}. Then
\begin{equation}\label{w8.1}
{I_4\leq \frac{1}{16} \int \chi^2 |\nabla^m Rm|^{p- 2}|\nabla^{m+ 1} Rm|^2+ C|\nabla^m Rm|_{(g(t), L^p(\Omega))}^{p- 1}+ I_5
}
\end{equation}
where $I_5\triangleq C \Big|\int |\nabla^m Rm|^{p- 2} \langle \chi\nabla^{m+ 1}\chi * \nabla Rm,  \nabla^m Rm \rangle \Big|$.

To estimate $I_5$, there are two cases:

If $m= 1$, then $\nabla Rm= \nabla^m Rm$, and by definition of $I_5$:
\begin{align}
I_5&\leq C \int |\nabla^m Rm|^{p- 1} |\nabla^{m+ 1}Rm |\cdot |\nabla^m \chi |\cdot |\chi |+ C\int |\nabla^m Rm|^p \cdot |\nabla \chi |^2  \nonumber \\
&\leq \frac{1}{16} \int \chi^2 |\nabla^m Rm|^{p- 2} |\nabla^{m+ 1}Rm |^2 +  C\int |\nabla^m Rm|^p \Big[ |\nabla^m \chi |^2+ |\nabla \chi |^2 \Big] \nonumber \\
&\leq \frac{1}{16} \int \chi^2 |\nabla^m Rm|^{p- 2} |\nabla^{m+ 1}Rm |^2 +  C\int |\nabla^m Rm|^p \label{w9}
\end{align}
the last inequality above uses the fact that $m= 1$ and $|\nabla \chi|_{(g(t), \infty)}$ are uniformly bounded on $[0, T)$.

If $m\geq 2$, by induction assumption, for any $q> 100$, $\int |\nabla Rm|^q$ is uniformly bounded on $[0, T)$.
\begin{equation}\nonumber
{I_5 \leq C\int |\nabla^m Rm|^{p- 1} \cdot |\nabla^{m+ 1} \chi |\cdot |\nabla Rm|
}
\end{equation}
\begin{equation}\nonumber
{\quad \leq C\Big(\int |\nabla^m Rm|^{p}\Big)^{\frac{p- 1}{p}} \cdot \Big(\int |\nabla^{m+ 1} \chi |^{2p} \Big)^{\frac{1}{2p}} \cdot \Big(\int |\nabla Rm|^{2p} \Big)^{\frac{1}{2p}}
}
\end{equation}
\begin{equation}\label{w10}
{I_5 \leq C |\nabla^m Rm|_{(g(t), L^p(\Omega))}^{p- 1}
}
\end{equation}

By (\ref{w9}) and (\ref{w10}), 
\begin{equation}\label{w11}
{I_5 \leq \frac{1}{16} \int \chi^2 |\nabla^m Rm|^{p- 2} |\nabla^{m+ 1}Rm |^2 +  C\int |\nabla^m Rm|^p+ C |\nabla^m Rm|_{(g(t), L^p(\Omega))}^{p- 1}
}
\end{equation}

By (\ref{w8.1}) and (\ref{w11}), 
\begin{equation}\label{w12}
{I_4\leq \frac{1}{8} \int \chi^2 |\nabla^m Rm|^{p- 2}|\nabla^{m+ 1} Rm|^2+  C\int |\nabla^m Rm|^p+ C|\nabla^m Rm|_{(g(t), L^p(\Omega))}^{p- 1}
}
\end{equation}

By (\ref{w8.0}) and (\ref{w12}), 
\begin{equation}\label{w13}
{I_2 \leq \frac{3}{16} \int \chi^2 |\nabla^{m+ 1} Rm|^2\cdot |\nabla^m Rm|^{p- 2} + C\int |\nabla^m Rm|^{p} + C|\nabla^m Rm|_{(g(t), L^{p}(\Omega))}^{p- 1}
}
\end{equation}

To estimate $I_3$,
\begin{equation}\nonumber
{I_3\leq C\int |\nabla^m Rm|^{p} |\chi^2 Rm| + C\sum_{i+ j+ k= m,\ j,\ k\leq m- 1}\int |\nabla^{m} Rm|^{p- 1} |\nabla^i(\chi^2)| |\nabla^j Rm| |\nabla^k Rm| 
}
\end{equation}
\begin{equation}\nonumber
{\quad \leq C\int |\nabla^m Rm|^{p} + C\sum_{i_1+ i_2+ j+ k= m,\ j,\ k\leq m- 1}\Big[ \Big(\int |\nabla^{m} Rm|^{p- 1}\Big)^{\frac{p- 1}{p}}\cdot \Big(\int |\nabla^{i_1}\chi |^{4p}\Big)^{\frac{1}{4p}} 
}
\end{equation}
\begin{equation}\nonumber
{\quad \quad \cdot \Big(\int |\nabla^{i_2}\chi |^{4p}\Big)^{\frac{1}{4p}} \cdot \Big(\int |\nabla^{j}Rm |^{4p}\Big)^{\frac{1}{4p}} \Big(\int |\nabla^{k}Rm|^{4p}\Big)^{\frac{1}{4p}} \Big]
}
\end{equation}

By induction assumption and the assumption of the proposition, 
\begin{equation}\label{w14}
{I_3\leq C\int |\nabla^m Rm|^{p}+ C|\nabla^m Rm|_{(g(t), L^p(\Omega))}^{p- 1}
}
\end{equation}

For $I_1$, integration by parts yields:
\begin{equation}\nonumber
{I_1\leq 2\int |\nabla^m Rm|^{p- 1}\chi |\nabla\chi | |\nabla^{m+ 1}Rm|- \int \chi^2 |\nabla^m Rm|^{p- 2}|\nabla^{m+ 1}Rm|^2 
}
\end{equation}

Use H\"older inequality and Schwartz inequality, 
\begin{equation}\label{w15}
{I_1\leq -\frac{15}{16}\int \chi^2 |\nabla^m Rm|^{p- 2}|\nabla^{m+ 1}Rm|^2+ C\int |\nabla^m Rm|^p
}
\end{equation}

By (\ref{w13}), (\ref{w14}), (\ref{w15}) and (\ref{w8}), 
\begin{align}
\frac{\partial}{\partial t} \int |\nabla^m Rm|^p &\leq -\frac{3}{4}\int \chi^2 |\nabla^m Rm|^{p- 2}|\nabla^{m+ 1}Rm|^2 \nonumber \\
&\quad + C\int |\nabla^m Rm|^p+ C|\nabla^m Rm|_{(g(t), L^p(\Omega))}^{p- 1} \nonumber 
\end{align}
By Young's inequality, 
\begin{equation}\nonumber
{\frac{\partial}{\partial t} \int |\nabla^m Rm|^p \leq C\Big(\int |\nabla^m Rm|^p+ 1 \Big)
}
\end{equation}

Then $\int |\nabla^m Rm|_{g(t)}^p\leq \Big[\int |\nabla^m Rm|_{g_0}^p+ 1\Big]e^{CT}- 1$, $t\in [0, T)$. Hence 
$|\nabla^m Rm|_{(g(t), L^{p}(\Omega))}$ is uniformly bounded on $[0, T)$.
}
\qed

\begin{theorem}\label{thm W8}
{If $g(t)$, $t\in [0, T)$, $0< T< +\infty$ is a smooth solution of the local Ricci flow on $(M^n, g_0)$. Assume $|\chi^2 Rm|_{(g(t), \infty)}$ is uniformly bounded on $[0, T)$, then for any $m\geq 0$, $|\nabla^m (\chi^2 Rm)|_{(g(t), \infty)}$ is uniformly bounded on $[0, T)$.
}
\end{theorem}

\pf
{By Lemma \ref{lem W5}, $|\nabla\chi |_{(g(t), \infty)}$ and $C_s(\Omega, t)$ are uniformly bounded on $[0, T)$, $g(t)$ are uniformly equivalent to each other, then it suffices to prove 

$(C1)$: ``For any $p> 100$, $m\geq 0$, $|\nabla^m (\chi^2 Rm)|_{(g(t), L^p(\Omega))}$ is uniformly bounded on $[0, T)$."

We prove $(C1)$ by induction on $m$. When $m= 0$, the conclusion follows from the assumption of the theorem. Assume 

$(A1)$: ``For any $p> 100$, $|\nabla^{l}(\chi^2 Rm)|_{(g(t), L^p(\Omega))}$ is uniformly bounded on $[0, T)$, where $0\leq l\leq m- 1$ and $m\geq 1$."

We will prove 

$(C2)$: ``For any $p> 100$, $|\nabla^{m}(\chi^2 Rm)|_{(g(t), L^p(\Omega))}$ is uniformly bounded on $[0, T)$." 

Firstly we claim: 

$(C3)$: ``For any $p> 100$, $\int |\nabla^k \chi |^p$ is uniformly bounded on $[0, T)$, where $0\leq k\leq m$." 

We prove $(C3)$ by induction on $k$. When $k= 0$, the conclusion follows from the fact that $g(t)$ are uniformly equivalent to each other under the assumption in the theorem. Assume 

$(A2)$: ``For any $q> 100$, $\int |\nabla^i \chi |^q$ is uniformly bounded on $[0, T)$, where $0\leq i\leq k- 1$." 

Note $k\leq m$ in our claim. Use Lemma \ref{lem W2}, 
\begin{align}
\frac{\partial}{\partial t}\int |\nabla^k \chi |^p&\leq C\sum_{1\leq i\leq k- 1} \Big( \int |\nabla^k \chi |^{p}\Big)^{\frac{p- 1}{p}} \cdot \Big( \int |\nabla^i (\chi^2 Rm)|^{2p}\Big)^{\frac{1}{2p}} \cdot \Big( \int |\nabla^{k- i}\chi |^{2p} \Big)^{\frac{1}{2p}} \nonumber \\
&\quad \quad + C \int |\nabla^k \chi |^p \cdot |\chi^2 Rm|  \nonumber \\
&\leq C |\nabla^k \chi |_{(g(t), L^p(\Omega))}^{p- 1}+ C \int |\nabla^k \chi |^p\leq C\Big(\int |\nabla^k \chi |^p + 1\Big) \nonumber 
\end{align}
For the last two inequalities, we used the assumption of the theorem, $(A1)$, $(A2)$ and Young's inequality. Hence $\int |\nabla^k \chi |_{g(t)}^p dV_{g(t)}$ is uniformly bounded on $[0, T)$.  From induction method, the claim $(C3)$ is proved. 

From the assumption of the theorem, $(C3)$ and Proposition \ref{prop W7}, we get

$(C4)$: ``For any $p> 100$, $|\nabla^i Rm |_{g(t), L^p(\Omega)}$ is uniformly bounded on $[0, T)$ where $0\leq i\leq m- 1$."

By Lemma \ref{lem W2} and integration by parts, 
\begin{equation}\label{w17}
{\frac{\partial}{\partial t}\int |\nabla^m (\chi^2 Rm)|^p \leq I_1+ I_2+ I_3+ I_4
}
\end{equation}
where 
\[I_1= -\int \chi^2 |\nabla^{m+ 1}(\chi^2 Rm)|^2 |\nabla^m (\chi^2 Rm)|^{p- 2}, \]
\[I_2= \sum_{i+ j+ k= m+ 2,\ k\leq m+ 1}\int |\nabla^m (\chi^2 Rm)|^{p- 2} \langle \nabla^m (\chi^2 Rm), \nabla^i \chi *\nabla^j \chi *\nabla^k(\chi^2 Rm) \rangle, \]
\[I_3= \sum_{i+ j+ k= m,\ k\leq m- 1}\int |\nabla^m (\chi^2 Rm)|^{p- 2} \langle \nabla^m (\chi^2 Rm), \nabla^i (\chi^2) *\nabla^j (\chi^2 Rm) *\nabla^k(Rm) \rangle, \]
\[I_4= \int |\nabla^m (\chi^2 Rm)|^{p- 2} \langle \nabla^m(\chi^2 Rm), \chi^2 Rm *\nabla^m (\chi^2 Rm) \rangle \]

To estimate $I_2$, 
\begin{equation}\nonumber
{I_2= \int |\nabla^m (\chi^2 Rm)|^{p- 2} \langle \nabla^m(\chi^2 Rm), \chi\nabla^{m+ 2}\chi *(\chi^2 Rm) \rangle
}
\end{equation}
\begin{equation}\nonumber
{\quad + \int |\nabla^m (\chi^2 Rm)|^{p- 2} \langle \nabla^m(\chi^2 Rm), \nabla \chi *\nabla^{m+ 1}\chi *(\chi^2 Rm) \rangle
}
\end{equation}
\begin{equation}\nonumber
{\quad + \int |\nabla^m (\chi^2 Rm)|^{p- 2} \langle \nabla^m(\chi^2 Rm),  \chi \nabla^{m+ 1}\chi *\nabla(\chi^2 Rm) \rangle
}
\end{equation}
\begin{equation}\nonumber
{\quad + \int |\nabla^m (\chi^2 Rm)|^{p- 2} \langle \nabla^m(\chi^2 Rm),  \chi \nabla\chi *\nabla^{m+ 1}(\chi^2 Rm) \rangle
}
\end{equation}
\begin{equation}\nonumber
{\quad + \sum_{i+ j+ k= m+ 2,\ i,\ j\leq m,\ k\leq m- 1}\int |\nabla^m (\chi^2 Rm)|^{p- 2} \langle \nabla^m(\chi^2 Rm),  \nabla^i \chi *\nabla^j \chi *\nabla^k (\chi^2 Rm) \rangle
}
\end{equation}
\begin{equation}\nonumber
{\quad + \int |\nabla^m (\chi^2 Rm)|^{p- 2} \langle \nabla^m(\chi^2 Rm),  (\chi\nabla^2\chi+ \nabla\chi* \nabla\chi) *\nabla^m (\chi^2 Rm) \rangle
}
\end{equation}

Integration by parts, 
\begin{equation}\label{w18}
{I_2\leq \Gamma_1+ \Gamma_2+ \Gamma_3+ \Gamma_4+ \Gamma_5
}
\end{equation}
where 
\[\Gamma_1= C\int |\nabla^m (\chi^2 Rm)|^{p- 2} \cdot |\nabla^{m+ 1}(\chi^2 Rm)| \cdot \chi |\nabla^{m+ 1}\chi| \cdot |\chi^2 Rm|\]
\[\Gamma_2= C\int |\nabla^m (\chi^2 Rm)|^{p- 1} \cdot |\nabla \chi| \cdot |\nabla^{m+ 1} \chi| \cdot |\chi^2 Rm|\]
\[\Gamma_3= \Big|\int |\nabla^m(\chi^2 Rm)|^{p- 2} \langle \nabla^m (\chi^2 Rm), \chi\nabla^{m+ 1}\chi *\nabla(\chi^2 Rm) \rangle \Big|\]
\[\Gamma_4= C\int \chi|\nabla^{m+ 1}(\chi^2 Rm)|\cdot |\nabla^m(\chi^2 Rm)|^{p- 1}+ C\int |\nabla^m(\chi^2 Rm)|^p \]
\[\Gamma_5= C\sum_{i+ j+ k= m+2,\ i,j\leq m,\ k\leq m- 1} \int |\nabla^m (\chi^2 Rm)|^{p- 2} \cdot |\nabla^i \chi| \cdot |\nabla^j \chi|\cdot |\nabla^k (\chi^2 Rm)| \]

Use H\"older inequality and Young's inequality,
\begin{align}
\Gamma_1  \leq \frac{1}{16} \int \chi^2 |\nabla^m (\chi^2 Rm)|^p |\nabla^{m+ 1}(\chi^2 Rm)|^2+ C \int |\nabla^m(\chi^2 Rm)|^p + C\int |\nabla^{m+ 1} \chi |^p \label{w19}
\end{align}

We can also estimate $\Gamma_2$, $\Gamma_4$ and $\Gamma_5$ in similar way:
\begin{align}
\Gamma_2 \leq C \int |\nabla^m(\chi^2 Rm)|^p+ C\int |\nabla^{m+ 1} \chi |^p \label{w20}
\end{align}

\begin{equation}\label{w21}
{\left.
\begin{array}{rl}
\Gamma_4 \leq \frac{1}{16} \int \chi^2 |\nabla^m (\chi^2 Rm)|^p |\nabla^{m+ 1}(\chi^2 Rm)|^2+ C \int |\nabla^m(\chi^2 Rm)|^p 
\end{array}\right.
}
\end{equation}
\begin{equation}\label{w22}
{\left.
\begin{array}{rl}
\Gamma_5 &\leq C \int |\nabla^m(\chi^2 Rm)|^p+ C\sum_{i+ j+ k= m+ 2,\ i,j\leq m,\ k\leq m- 1} \Big[\int |\nabla^i \chi |^{3p} \\
&\quad  + \int |\nabla^j \chi |^{3p} + \int |\nabla^k (\chi^2 Rm) |^{3p} \Big] \\
&\leq C \int |\nabla^m(\chi^2 Rm)|^p+ C
\end{array}\right.
}
\end{equation}
the last inequality in estimating $\Gamma_5$ follows from $(A1)$ and $(C3)$.

To estimate $\Gamma_3$, there are two cases: Case $(1)$ and Case $(2)$.

Case $(1)$: If $m= 1$, do integration by parts, 
\begin{equation}\nonumber
\Gamma_3\leq C\int |\nabla^m(\chi^2 Rm)|^{p- 1} |\nabla^{m+ 1} (\chi^2 Rm)| (\chi |\nabla\chi |)+ C\int  |\nabla^m (\chi^2 Rm)|^p |\nabla \chi |^2
\end{equation}
Use the fact that $|\nabla \chi|_{(g(t), \infty)}$ is uniformly bounded, we get that if $m= 1$,
\begin{equation}\label{w23}
\Gamma_3 \leq \frac{1}{16} \int \chi^2 |\nabla^m (\chi^2 Rm)|^p |\nabla^{m+ 1}(\chi^2 Rm)|^2+ C\int |\nabla^m(\chi^2 Rm)|^{p} 
\end{equation}

Case $(2)$: If $m\geq 2$, use integration by parts, 
\begin{equation}\nonumber
{\left.
\begin{array}{rl}
\Gamma_3 &\leq C \int |\nabla^m(\chi^2 Rm)|^{p- 2} \cdot |\nabla^{m+ 1}(\chi^2 Rm)| \cdot |\nabla(\chi^2 Rm)| \cdot \chi |\nabla^m \chi | \\
&\quad + C \int |\nabla^m(\chi^2 Rm)|^{p- 1} |\nabla\chi | |\nabla^{m}\chi | |\nabla(\chi^2 Rm)| + \Gamma_{3,1}
\end{array}\right.
}
\end{equation}
where 
\[\Gamma_{3,1}= \Big| \int |\nabla^m(\chi^2 Rm)|^{p- 2} \langle \nabla^m (\chi^2 Rm), \chi\nabla^m \chi *\nabla^2(\chi^2 Rm) \rangle \Big|\]

Use H\"older inequality and Young's inequality, 
\begin{equation}\nonumber
{\left.
\begin{array}{rl}
\Gamma_3 &\leq \frac{1}{16} \int \chi^2 |\nabla^m (\chi^2 Rm)|^p |\nabla^{m+ 1}(\chi^2 Rm)|^2+ C\int |\nabla^m(\chi^2 Rm)|^{p} \\
&\quad + C\int |\nabla(\chi^2 Rm)|^{2p}+ C\int |\nabla^m \chi |^{2p}+ \Gamma_{3,1} 
\end{array}\right.
}
\end{equation}
Note $m\geq 2$, use $(A1)$ and $(C3)$, 
\begin{equation}\label{w24}
\Gamma_3 \leq \frac{1}{16} \int \chi^2 |\nabla^m (\chi^2 Rm)|^p |\nabla^{m+ 1}(\chi^2 Rm)|^2+ C\int |\nabla^m(\chi^2 Rm)|^{p} + C+ \Gamma_{3,1} 
\end{equation}

To estimate $\Gamma_{3,1}$, note that $m\geq 2$, and there are still two cases: Case $(2.1)$ and Case $(2.2)$.

Case $(2.1)$: If $m= 2$, use integration by parts and the fact that $|\nabla\chi |_{(g(t), \infty)}$ is uniformly bounded, 
\begin{equation}\label{w25}
{\left.
\begin{array}{rl}
\Gamma_{3,1}\leq  \frac{1}{16} \int \chi^2 |\nabla^m (\chi^2 Rm)|^p |\nabla^{m+ 1}(\chi^2 Rm)|^2+ C\int |\nabla^m(\chi^2 Rm)|^{p}
\end{array}\right.
}
\end{equation}

Case $(2.2)$: If $m\geq 3$, use H\"older inequality and Young's inequality, 
\begin{equation}\label{w26}
{\left.
\begin{array}{rl}
\Gamma_{3,1}&\leq  C\int |\nabla^m(\chi^2 Rm)|^{p}+ C\int |\nabla^{m}\chi |^{2p}+ C\int |\nabla^2 (\chi^2 Rm) |^{2p}  \\
&\leq  C\int |\nabla^m(\chi^2 Rm)|^{p}+ C
\end{array}\right.
}
\end{equation}
the last inequality follows from $(A1)$ and $(C3)$.

By (\ref{w25}) and (\ref{w26}), 
\begin{equation}\label{w27}
{\left.
\begin{array}{rl}
\Gamma_{3,1}\leq \frac{1}{16} \int \chi^2 |\nabla^m (\chi^2 Rm)|^p |\nabla^{m+ 1}(\chi^2 Rm)|^2+ C\int |\nabla^m(\chi^2 Rm)|^{p}+ C
\end{array}\right.
}
\end{equation}

From (\ref{w24}) and (\ref{w27}), we get that if $m\geq 2$,
\begin{equation}\label{w28}
{\left.
\begin{array}{rl}
\Gamma_{3}\leq \frac{1}{8} \int \chi^2 |\nabla^m (\chi^2 Rm)|^p |\nabla^{m+ 1}(\chi^2 Rm)|^2+ C\int |\nabla^m(\chi^2 Rm)|^{p}+ C
\end{array}\right.
}
\end{equation}

From (\ref{w23}) and (\ref{w28}), when $m\geq 1$, (\ref{w28}) holds.

By (\ref{w18}), (\ref{w19}), (\ref{w20}), (\ref{w21}), (\ref{w22}) and (\ref{w28}),  
\begin{equation}\label{w29}
{\left.
\begin{array}{rl}
I_2&\leq \frac{1}{4} \int \chi^2 |\nabla^m (\chi^2 Rm)|^p |\nabla^{m+ 1}(\chi^2 Rm)|^2+ C\int |\nabla^m(\chi^2 Rm)|^{p}\\
&\quad  + C\int |\nabla^{m+ 1}\chi |^p+ C
\end{array}\right.
}
\end{equation}

To estimate $I_3$, 
\begin{align}
I_3& \leq C \Big[ \sum_{i_1+ i_2+ j+ k= m,\ k\leq m- 1,\ j\leq m- 1} \Big(\int |\nabla^m (\chi^2 Rm)|^{p}\Big)^{\frac{1}{p}}\cdot \Big(\int |\nabla^{i_1}\chi |^{4p} \Big)^{\frac{1}{4p}}   \nonumber \\
&\quad \cdot \Big(\int |\nabla^{i_2}\chi |^{4p} \Big)^{\frac{1}{4p}} \cdot \Big(\int |\nabla^{j} (\chi^2 Rm)|^{4p} \Big)^{\frac{1}{4p}} \cdot \Big(\int |\nabla^{k}Rm|^{4p} \Big)^{\frac{1}{4p}} \Big] \nonumber \\
&\quad + C\int |\nabla^{m} (\chi^2 Rm)|^p \nonumber
\end{align}
Hence from $(A1)$, $(C3)$, $(C4)$ and Young's inequality, 
\begin{equation}\label{w30}
{I_3\leq C\int |\nabla^{m} (\chi^2 Rm)|^p+ C
}
\end{equation}

From the assumption of the theorem,
\begin{equation}\label{w31}
{I_4\leq C\int |\nabla^{m} (\chi^2 Rm)|^p
}
\end{equation}

By (\ref{w29}), (\ref{w30}), (\ref{w31}) and (\ref{w17}), 
\begin{equation}\label{w32}
{\left.
\begin{array}{rl}
\frac{\partial}{\partial t} \int |\nabla^m (\chi^2 Rm)|^p &\leq -\frac{3}{4} \int \chi^2 |\nabla^m (\chi^2 Rm)|^p |\nabla^{m+ 1}(\chi^2 Rm)|^2\\
&\quad  + C\int |\nabla^m(\chi^2 Rm)|^{p} + C\int |\nabla^{m+ 1}\chi |^p+ C
\end{array}\right.
}
\end{equation}

By Lemma \ref{lem W2}, 
\begin{align}
\frac{\partial}{\partial t} \int |\nabla^{m+ 1} \chi |^p &\leq C \int |\nabla^{m+ 1}\chi |^{p}\cdot |\chi^2 Rc|+ C\int |\nabla^{m+ 1}\chi |^{p- 1} \cdot |\nabla^m (\chi^2 Rm)|\cdot |\nabla \chi | \nonumber \\
&\quad + C\sum_{1\leq k\leq m- 1}\int |\nabla^{m+ 1}\chi |^{p- 1} \cdot |\nabla^k (\chi^2 Rm)|\cdot |\nabla^{m+ 1- k}\chi |  \nonumber \\
&\leq C\int |\nabla^{m+ 1}\chi |^{p} + C\Big(\int |\nabla^{m+ 1}\chi |^{p}\Big)^{\frac{1}{p}} \cdot \Big(\int |\nabla^m (\chi^2 Rm)|^p\Big)^{\frac{1}{p}} \nonumber \\
&+ C\sum_{1\leq k\leq m- 1}\Big[ \Big(\int |\nabla^{m+ 1}\chi |^{p}\Big)^{\frac{1}{p}}\cdot \Big(\int |\nabla^{k}\chi |^{2p}\Big)^{\frac{1}{2p}} \cdot \Big(\int |\nabla^{m+ 1- k}\chi |^{2p}\Big)^{\frac{1}{2p}} \Big] \nonumber 
\end{align}
the last inequality follows from H\"older inequality and the assumption of theorem. 

Hence by Young's inequality, $(A1)$ and $(C3)$, 
\begin{equation}\label{w33}
{\frac{\partial}{\partial t} \int |\nabla^{m+ 1} \chi |^p\leq C_1\int |\nabla^{m+ 1}\chi |^{p}+ C_2 \int |\nabla^m (\chi^2 Rm)|^p + C_3 
}
\end{equation}

By (\ref{w33}) and Lemma \ref{lem W3}, on $[0, T)$, 
\begin{align}
\int |\nabla^{m+ 1}\chi |_{g(t)}^p dV_{g(t)}&\leq e^{C_1 t} \Big[\int |\nabla^{m+ 1}\chi |_{g_0}^p dV_{g_0} \nonumber \\
&\quad + \int_0^t e^{-C_1 s}\Big(C_3+ C_2|\nabla^m(\chi^2 Rm)|_{(g(s), L^p(\Omega))}^p \Big) ds\Big] \nonumber 
\end{align}

Hence on $[0, T)$,
\begin{equation}\label{w34}
{\int |\nabla^{m+ 1}\chi |^p\leq C+ C\int_0^t |\nabla^m(\chi^2 Rm)|_{(g(s), L^p(\Omega))}^p ds
}
\end{equation}

By (\ref{w32}) and (\ref{w34}), 
\begin{align}
\frac{\partial}{\partial t} |\nabla^m (\chi^2 Rm)|_{(g(t), L^p(\Omega))}^p &\leq C_4 \int_0^t |\nabla^m (\chi^2 Rm)|_{(g(s), L^p(\Omega))}^p ds \nonumber \\
&\quad + C_5 |\nabla^m (\chi^2 Rm)|_{(g(t), L^p(\Omega))}^p + C_6 \label{w35}
\end{align}
By (\ref{w35}) and Lemma \ref{lem W4}, $|\nabla^m (\chi^2 Rm)|_{(g(t), L^p(\Omega))}^p$ is uniformly bounded on $[0, T)$. Hence $(C2)$ is proved. By induction method, $(C1)$ is proved. 
}
\qed

\begin{prop}\label{prop W9}
{If $g(t)$, $t\in [0, T)$ is a solution to the local Ricci flow on $(M^n, g_0)$, and for any $m\geq 0$, $|\nabla^m (\chi^2 Rm)|_{(g(t),\infty)}$ is uniformly bounded on $[0, T)$. Then the local Ricci flow can be extended smoothly through time $T$.
}
\end{prop}

\pf
{Use similar argument in section $7$ of chapter $6$ of \cite{RI}, the difference is that $\nabla^m (\chi^2 Rm)$ plays the key role in our local Ricci flow case instead of $\nabla^m Rm$ in Ricci flow case.
}
\qed

\begin{theorem}\label{thm W10}
{If $g(t)$, $t\in [0, T)$, $0< T< +\infty$ is a smooth solution of the local Ricci flow on $(M^n, g_0)$. Assume $|\chi^2 Rm|_{(g(t), \infty)}$ is uniformly bounded on $[0, T)$, then the local Ricci flow can be extended smoothly through time $T$.
}
\end{theorem}

\pf
{By Theorem \ref{thm W8} and Proposition \ref{prop W9}, we get our conclusion.
}
\qed

\section{Existence time estimates (The first case: $p_0> \frac{n}{2}$)}

In this section, we use the product of two different cut-off functions to do curvature estimates, so we call them ``Local local'' curvature estimates. Also in this section, we assume that $(M^n, g_0)$ is a $n$-dimensional $(n\geq 3)$ complete noncompact Riemannian manifold satisfying the following assumptions:
\begin{equation}\label{5.1}
{\left\{
\begin{array}{rl}
\Big(\fint_{B_x(4r_0)} h^{\frac{2n}{n-2}} dV_{g_0} \Big)^{\frac{n-2}{n}} &\leq  A_0 r_0^2 \fint_{B_x(4r_0)}|\nabla h|^2 dV_{g_0}\ , \\
\Big(\fint_{B_x(r_0)} |Rm(g_0)|^{p_0} dV_{g_0} \Big)^{\frac{1}{p_0}} &\leq K_1
\end{array}\right.
}
\end{equation}
for any $x\in M$ and $h\in C_0^{\infty}(B_x(4r_0))$, where $A_0\geq 1$, $p_0> \frac{n}{2}$, $K_1$, $r_0$ are some positive constants, and $B_x(r)$ is the open geodesic ball of radius $r$ centered at $x$ on $(M, g_0)$. And $\fint$ means the usual average of integrand.

We always assume $\Vert\nabla \chi\Vert_{\infty}\triangleq \Vert\nabla \chi\Vert_{(g_0,\infty)}= \sup_{x\in \Omega}\big(g_0^{ij}\nabla_i \chi \nabla_j \chi\big)^{\frac{1}{2}}\leq 1$. Now we state some lemmas about local Sobolev constants under scaling of metric. There is well-known local Sobolev inequality
\begin{equation}\label{5.2}
{\Big(\fint_{B_x(r)} f^{\frac{2n}{n-2}} dV_{g(t)}\Big)^{\frac{n-2}{n}}\leq  A(B_x(r),t) r^2 \fint_{B_x(r)}|\nabla f|^2 dV_{g(t)}, \quad \forall f\in C_0^{\infty}(B_x(r)) 
}
\end{equation}
with respect to each metric $g(t)$, and $A(B_x(r),t)$ depends on both $t$ and $B_x(r)$. Note in (\ref{5.2}), although $B_x(r)$ is the geodesic ball with respect to $g_0$, $\fint_{B_x(r)} f^{\frac{2n}{n- 2}}dV_{g(t)}\triangleq \frac{1}{V_{g(t)}(B_x(r))}\int_{B_x(r)} f^{\frac{2n}{n- 2}}dV_{g(t)}$. When $r$ is fixed, $A(t)\triangleq \sup_{x\in M^n} A(B_x(r), t)$.

\begin{lemma} \label{lem 5.1}
{If $a^{-1}g_0 \leq g(t) \leq ag_0$, then 
\begin{equation}\nonumber
{a^{-(n+1)}A_0 \leq A(t) \leq a^{n+1} A_0 
}
\end{equation}
where $A_0= A(0)$ is respect to $g_0$, $a\geq 1$ is some positive constant.
}
\end{lemma}

\pf
{It is straightforward from scaling argument and (\ref{5.2}).
}
\qed

\begin{remark}
{We had defined two Sobolev constants, the Sobolev constant for $(\Omega, g(t))$ is $C_s(\Omega, g(t))$ defined in (\ref{2.7}) and local Sobolev constant $A(B_x(r), t)$ is defined in (\ref{5.2}). They are different from each other although they satisfy similar property as in Lemma \ref{lem 5.1}.
}
\end{remark}

\begin{lemma}\label{lem 5.3}
{Assume that the inequality:
\[ \Big(\fint_{B_x(4r_0)} h^{\frac{2n}{n-2}} dV_{g_0} \Big)^{\frac{n-2}{n}} \leq  A_0 r_0^2 \fint_{B_x(4r_0)}|\nabla h|^2 dV_{g_0} \]
where $h\in C_0^{\infty}(B_x(4r_0))$, $A_0\geq 1$ and $r_0$ are fixed positive constants. Then
\[V_{g_0}\big(B_x(\rho)\big) \geq N_1 \Big(\frac{\rho}{4r_0}\Big)^n V_{g_0}\big(B_x(4r_0)\big), \quad 0< \rho\leq 4r_0 \]
where $N_1= 2^{2n-\frac{n^2}{2}}A_0^{-\frac{n}{2}}$ and $N_1\leq 1$.
}
\end{lemma}

\pf
{For simplicity, we use $V(4r_0)$ instead of $V_{g_0}(B_x(4r_0))$ in the proof. By assumption, 
\[\Big( \int_{B(4r_0)}|h|^{\frac{2n}{n- 1}}\Big)^{\frac{n- 2}{2n}}\leq A_0^{\frac{1}{2}}r_0V(4r_0)^{-\frac{1}{n}}\Big(\int_{B(4r_0)}|\nabla h|^2\Big)^{\frac{1}{2}}\]
We recall Theorem $3.1.5$ in \cite{Saloff-1}:

{\it \textbf{Theorem}:}~
{Assume that the inequality 
\[\forall f\in C_0^{\infty}(M)\ , \ \quad ||f||_r\leq (C||\nabla f||_p)^{\theta} ||f||_{s}^{1-\theta}\]
is satisfied for some $r$, $s$, $\theta$ with $0< s\leq r\leq \infty$ and $0< \theta\leq 1$. Assume also that 
\[\frac{\theta}{p}+ \frac{1- \theta}{s}- \frac{1}{r}> 0\]
Then $V(x,\rho)\geq c\rho^{\nu}$ with $\nu$ defined by
\[\frac{\theta}{\nu}= \frac{\theta}{p}+ \frac{1- \theta}{s}- \frac{1}{r}\]
and the constant $c$ given by $c= 2^{-\frac{\nu^2}{\theta r}- \frac{\nu}{\theta}}C^{-\nu}$.
}
\qed

Let $\theta= 1$, $s=r= \frac{2n}{n-2}$, $p= 2$ in the above theorem, $C= \frac{A_0^{\frac{1}{2}} r_0}{V^{\frac{1}{n}}}$, where $V= V(4r_0)$, and $\nu= n$. Then $c= 2^{-\frac{n^2}{2}} (A_0^{\frac{1}{2}}r_0)^{-n}V$, by the above theorem, 
\[V(\rho)\geq c\rho^{\nu}= 2^{2n- \frac{n^2}{2}}A_0^{-\frac{n}{2}}\Big(\frac{\rho}{4r_0}\Big)^n V\geq N_1\Big(\frac{\rho}{4r_0}\Big)^n V(4r_0)\]
}
\qed

\begin{lemma}\label{lem 5.4}
{$(M^n, g)$ is a smooth, complete Riemannian manifold with the property:
\[V_{g}\big(B_x(\rho)\big) \geq N_1 \Big(\frac{\rho}{4r_0}\Big)^n Vol\big(B_x(4r_0)\big), \quad 0< \rho\leq 4r_0, \quad \forall x\in M \]
where $r_0$ and $N_1$ are fixed positive constants. Then there exists a sequence $(x_i)$ of points of $M$ such that \\
(I) $M= \bigcup_{i=1}^{\infty} B_{x_i}(r_0)$. \\
(II) For any $i\neq j$, $B_{x_i}(\frac{r_0}{2})\cap B_{x_j}(\frac{r_0}{2})= \emptyset$.\\
(III) For any $j$, $B_{x_j}(r_0)$ intersects at most $N$ balls $B_{x_i}(r_0)$, $i\neq j$, where $N= 2^{4n}N_1^{-2}= 2^{n^2}A_0^n$.
}
\end{lemma}

\pf
{Let $X_{r_0}= \{(x_i)_I, x_i\in M$, s.t. $I$ is countable and $\forall i\neq j,\  d_g(x_i, x_j)\geq r_0\}$. As one can easily check, $X_{r_0}$ is partially ordered by inclusion and every chain in $X_{r_0}$ has an upper bound. Hence, by Zorn's lemma, $X_{r_0}$ contains a maximal element $(x_i)$ and $(x_i)$ satisfies (I) and (II). From now on, let $(x_i)$ be such that (I) and (II) are satisfied. For $x\in M$, we define
\[I_{2r_0}(x)= \{i: x\in B_{x_i}(2r_0)\}\]
Then 
\begin{equation}\label{5.3}
V(B_x(2r_0))\geq \frac{N_1 }{2^{n}} V(B_x(4r_0))\geq \frac{N_1 }{2^{n}} \sum_{i\in I_{2r_0}(x)}V(B_{x_i}(\frac{r_0}{2}))
\end{equation}
The last inequality follows from $\bigcup_{i\in I_{2r_0}(x)}B_{x_i}(\frac{r_0}{2})\subseteq B_x(4r_0)$ and $B_{x_i}(\frac{r_0}{2})\cap B_{x_j}(\frac{r_0}{2})= \emptyset$ if $i\neq j$. By the assumption, $V(B_{x_i}(\frac{r_0}{2}))\geq \frac{N_1}{8^n}V(B_{x_i}(4r_0))$. And for $i\in I_{2r_0}(x)$, $B_x(2r_0)\subseteq B_{x_i}(4r_0)$, 
\begin{equation}\label{5.4}
V(B_{x_i}(\frac{r_0}{2}))\geq \frac{N_1}{8^n}V(B_{x}(2r_0))
\end{equation}
By (\ref{5.3}) and (\ref{5.4}), 
\[V(B_x(2r_0))\geq N_1^2 2^{-4n}Card(I_{2r_0}(x))V(B_x(2r_0))\]
where $Card$ stand for the cardinality. For any $x\in M$, $Card(I_{2r_0}(x))\leq 2^{4n}N_1^{-2}$. For fixed $x_j$, if $B_{x_j}(r_0)\cap B_{x_i}(r_0)\neq \emptyset$, then $x_i\in B_{x_j}(2r_0)$, $i\in I_{2r_0}(x_j)$ and we get that $N\leq Card(I_{2r_0}(x_j))\leq 2^{4n}N_1^{-2}$
}
\qed

\begin{lemma}\label{lem 5.5}
{For the covering $\Big(B_{x_i}(r_0)\Big)$ of $(M, g)$ as in Lemma \ref{lem 5.4}, we can find $\xi_i\in C_0^{0, 1}(B_{x_i}(r_0))$, $0\leq \xi_i\leq 1$, $\sum_i \xi_i^2= 1$, $|\nabla \xi_i|^2\leq C(r_0, N)$, where $C(r_0, N)= \frac{8}{r_0^2}(N+ 1)$ and $N$ is as in Lemma \ref{lem 5.4}. 
}
\end{lemma}

\pf
{Let $\rho: [0, \infty) \rightarrow [0, 1]$ be defined by 
\begin{equation}\label{5.5}
\rho(t) = \left\{
\begin{array}{rl}
1                              &  0   \leq  t     \leq \frac{r_0}{2}    \\
(2- \frac{2}{r_0}t)^2   &  \frac{r_0}{2}  \leq  t     \leq r_0    \\
0                              &  r_0 \leq t                                   \\
\end{array} \right.
\end{equation}

Let $\alpha_i(x)= \rho(d_g(x, x_i))$, where $d_g$ denotes the distance associated to $g$ and $x\in M$. Clearly, $\alpha_i$ is Lipschitz with compact support. Note for any $m$, $|\nabla \alpha_m|\leq 2(2- \frac{2}{r_0}t)\frac{2}{r_0}= \frac{4}{r_0}\alpha_m^{\frac{1}{2}}$, that is $\frac{|\nabla \alpha_m|^2}{|\alpha_m|}\leq \Big(\frac{4}{r_0}\Big)^2$. Let $\eta_i= \frac{\alpha_i}{\sum_m \alpha_m}$. Then $\eta_i$ is a partition of unity subordinate to the covering $\Big(B_{x_i}(r_0)\Big)$.
\begin{align} 
|\nabla \eta_i| &\leq \Big|\frac{\nabla \alpha_i}{(\sum_m \alpha_m)}\Big|+ \frac{|\alpha_i|\cdot |\sum_m \nabla \alpha_m|}{(\sum_m \alpha_m)^2} \leq |\nabla \alpha_i|+ \frac{4\eta_i}{r_0}\cdot\frac{\sum_m \alpha_m^{\frac{1}{2}}}{\sum_m \alpha_m} \nonumber \\
&\leq |\nabla \alpha_i|+ \frac{4\eta_i}{r_0}\sqrt{N} \nonumber
\end{align}
in the last equality, we use Cauchy-Schwartz inequality and property (III) in Lemma \ref{lem 5.4}.

Let $\xi_i= \eta_i^{\frac{1}{2}}$, then $\xi_i\in C_0^{0, 1}(B_{x_i}(r_0))$, $0\leq \xi_i\leq 1$, $\sum_i \xi_i^2= 1$. Hence
\[|\nabla \xi_i|^2= \frac{1}{4}\cdot \frac{|\nabla \eta_i|^2}{|\eta_i|}\leq \frac{1}{4|\alpha_i|}\cdot 2\Big(|\nabla \alpha_i|^2+ \big(\frac{4}{r_0}\eta_i\sqrt{N}\big)^2\Big)\leq \frac{8}{r_0^2}(N+ 1)\]
}
\qed

For simplicity reason, we use notation $B_i$ to replace $B_{x_i}(r_0)$ in the rest of this section. Note that $(x_i)$ are choosen as in Lemma \ref{lem 5.4} and are fixed in the rest of this section and section $6$. We also use $\xi$ to replace $\xi_i$ in some proofs of this section, it will be clear from the context when we do such replacement. We use the following notation in the rest of this section:
\[\fint h dV_{g(t)}\triangleq \frac{1}{V_{g(t)}(B_i)} \int_{B_i\cap\Omega} h dV_{g(t)}\]
for any $h\in C_{0}^{\infty}(B_i)$, and it will be clear from the context which $B_i$ we choose.

From Theorem \ref{thm 2.7} and DeTurck's trick, local Ricci flow (\ref{2.1}) has a smooth solution on a sufficiently small time interval starting at $t=0$. By (\ref{5.1}), we can assume that $[0, T_{\max})$ is the maximal time interval, on which local Ricci flow (\ref{2.1}) has a smooth solution and such that the following hold for each metric $g(t)$ (where $g(0)= g_0$):
\begin{equation}\label{5.6}
\Big(\fint_{B_x(4r_0)} h^{\frac{2n}{n-2}} dV_{g(t)} \Big)^{\frac{n-2}{n}}\leq  4A_0 r_0^2 \fint_{B_x(4r_0)}|\nabla h|^2 dV_{g(t)}\ ;
\end{equation}
\begin{equation}\label{5.7}
\frac{1}{2}g_0 \leq g(t) \leq 2g_0\ ;
\end{equation}
\begin{equation}\label{5.8}
\Big(\frac{1}{V_{g(t)}(B_i)}\int_{B_i \cap \Omega} |Rm(g(t))|^{p_0}dV_{g(t)}\Big)^{\frac{1}{p_0}}\leq (2N)^{\frac{1}{p_0}} K_1 \quad i= 1, 2, 3, \cdots
\end{equation}
for any $x\in M$, $h\in C_0^{\infty}(B_x(4r_0))$ and $B_i$ are chosen as in Lemma \ref{lem 5.4}.

In the following lemma, we get the parabolic version of energy estimates on $[0, T_{\max})$ under local Ricci flow.

\begin{lemma}\label{lem 5.6}
{For $0\leq t< T_{\max}$,  any $i$, $q\geq 1$, $p\geq \frac{n}{2}$, $p'\geq 0$ and $p\geq p'$, there exists some positive constant 
\begin{equation}\label{5.9}
{C_1(A_0,K_1,n,p_0,r_0)= C(n,p_0) A_0^{\frac{(2p_0+ 3)n}{2p_0- n}}K_1^{\frac{2p_0}{2p_0- n}} \Big[r_0^{-2}+ r_0^{\frac{2n}{2p_0- n}} \Big]
}
\end{equation}
such that:
\begin{equation}\label{5.10}
{\left.
\begin{array}{rl}
& \frac{\partial}{\partial t}\Big(\fint \xi_i^{2q} \chi^{2p'}|Rm|^p dV_{g(t)}\Big)+ \frac{1}{10} 
 \Big(\fint \big|\nabla (\xi_i^{q} \chi^{p'+1}|Rm|^{\frac{p}{2}}) \big|^2 dV_{g(t)} \Big)\\
& \leq  C_1(A_0,K_1,n,p_0,r_0) p^{\frac{n}{2p_0- n}+ 3} q^2 \Big( \fint \xi_i^{2q- 2}\chi^{2p'}|Rm|^p dV_{g(t)} \Big)
\end{array}\right.
}
\end{equation}
}
\end{lemma}

\pf
{Set $f(x,t)=|Rm(x)|_{g(t)}$, we get the following:
\begin{equation}\label{5.11}
{\frac{\partial}{\partial t}\Big(\fint \xi^{2q} \chi^{2p'}f^p \Big) \leq \sum_{k=1}^{7}I_k  
}
\end{equation}

\noindent
where
\[ I_1 = p\fint \xi^{2q} \chi^{2p'+2}f^{p-1}\Delta f , \quad I_2 = p\fint \xi^{2q} \chi^{2p'+2}f^{p-2}|\nabla f|^2 \]
\[ I_3 = -(1- \epsilon)p\fint \xi^{2q} \chi^{2p'+2}f^{p-2} |\nabla Rm|^2,\quad I_4 = 10p \fint \xi^{2q} \chi^{2p'+2}f^{p+1}- \fint \xi^{2q} \chi^{2p'+ 2}f^{p} R \]
\[ I_5 = \Big(\frac{10^5}{\epsilon} \Big)np \fint \xi^{2q} \chi^{2p'}|\nabla \chi|^2 f^{p}, \quad I_6 = 8p \fint \xi^{2q} \chi^{2p'+1}f^{p-2} g^{ri}g^{sj}g^{pk}g^{ql} R_{rspq} R_{jk} \chi_{il}  \]
\[I_7= n(n- 1)\Big( \fint\xi^{2q} \chi^{2p'} f^p \Big) \cdot \Big(\frac{1}{V_{g(t)}(B_i)}\int_{B_i}\chi^2 f\Big) \]
For simplification, we set the following notations:
\begin{equation}\nonumber
{\left.
\begin{array}{rl}
& \tilde{a}= \xi^{q}\chi^{p'}f^{\frac{p}{2}} |\nabla \chi|; \quad \tilde{b}= \xi^{q}\chi^{p'+1} f^{\frac{p}{2}-1}|\nabla f|;\\
& \tilde{c}= \xi^{q}\chi^{p'+1} f^{\frac{p}{2}-1}|\nabla Rm|; \quad \tilde{d}= \fint \big|\nabla (\xi^{q}\chi^{p'+1}f^{\frac{p}{2}}) \big|^2; \quad \tilde{e}= \xi^{q- 1}\chi^{p'+ 1} f^{\frac{p}{2}}|\nabla \xi|
\end{array}\right.
}
\end{equation}

Then
\begin{equation}\nonumber
{\left.
\begin{array}{rl}
I_1 \leq -p(p-1)\fint \tilde{b}^2+ 2\epsilon p\fint \tilde{b}^2+ \frac{pq^2}{\epsilon}\fint \tilde{e}^2+ \frac{p(p'+ 1)^2}{\epsilon}\fint \tilde{a}^2
\end{array}\right.
}
\end{equation}
\begin{equation}\nonumber
{I_2\leq p\fint \tilde{b}^2, \quad I_3= -(1- \epsilon)p\fint \tilde{c}^2, \quad I_5\leq \Big(\frac{10^5}{\epsilon} \Big)np\fint \tilde{a}^2 
}
\end{equation}
where $0< \epsilon< \frac{1}{8}$ is some constant to be determined later.

Using integration by parts and Young's inequality, 
\begin{equation}\nonumber
{I_6\leq \epsilon p\fint \tilde{c}^2+ \frac{1}{\epsilon}10^6 n^3 p^3 q^2\Big( \fint \tilde{a}^2+ \fint \tilde{e}^2\Big) 
}
\end{equation}

And
\begin{align}
I_7&\leq n(n- 1)\fint \xi^{2q}\chi^{2p'}f^p \cdot \Big(\frac{1}{V_{g(t)}(B_i)}\int_{B_i\cap\Omega} f^{p_0}\Big)^{\frac{1}{p_0}} \nonumber \\
&\leq C(n, p_0)A_0^{\frac{n}{p_0}}K_1\Big(\fint \xi_i^{2q- 2} \chi^{2p'} f^{p} dV_{g(t)} \Big) \nonumber
\end{align}
in the last inequality above, we use the fact that $N= 2^{n^2}A_0^n$ by Lemma \ref{lem 5.3} and Lemma \ref{lem 5.4}.

Note that $|R|\leq n|Rm|\leq 2p f$, $p_0'< \frac{n}{n- 2}$ where $\frac{1}{p_0}+ \frac{1}{p_0'}= 1$,
\begin{equation}\nonumber
{\left.
\begin{array}{rl}
I_4&\leq 12p \fint \xi^{2q} \chi^{2p'+2}f^{p+1}\leq 12p \nparallel f\parallel_{p_0, B_i\cap \Omega}\cdot \Big[\fint \big(\xi^{2q}\chi^{2p'+ 2}f^p\big)^{p_0'} \Big]^{\frac{1}{p_0'}} \\
& \leq 12p (2N)^{\frac{1}{p_0}} K_1 \cdot \Big[C(\epsilon, p_0, n)(2N p K_1 A_0 r_0^2)^{\frac{n}{2p_0- n}} \fint \xi^{2q}\chi^{2p'+ 2}f^p \\
&\quad + \frac{\epsilon}{2N p K_1 A_0 r_0^2}\Big(\fint \big(\xi^{2q}\chi^{2p'+ 2}f^p\big)^{\frac{n}{n- 2}} \Big)^{\frac{n- 2}{n}} \Big] \\
& \leq 100 \epsilon \tilde{d}+ C(\epsilon,n,p_0) p^{\frac{n}{2p_0- n}+ 1} q^2 A_0^{\frac{n^2+ 3n}{2p_0- n}}K_1^{\frac{2p_0}{2p_0- n}} r_0^{\frac{2n}{2p_0- n}}\fint_{B_i\cap\Omega} \xi^{2q- 2}\chi^{2p'+ 2}f^p
\end{array}\right.
}
\end{equation}
where $\nparallel f\parallel_{p_0, B_i\cap \Omega}= \Big(\frac{1}{V_{g(t)}(B_i)}\int_{B_i\cap\Omega} f^{p_0}\Big)^{\frac{1}{p_0}}$.

Combining the above estimates, note $\tilde{b}\leq \tilde{c}$, then
\begin{equation}\label{5.12}
{\left.
\begin{array}{rl}
\frac{\partial}{\partial t}\Big(\fint \xi^{2q}\chi^{2p'}f^p \Big) &\leq p(1+ 4\epsilon- p)\fint \tilde{b}^2+ 100\epsilon \tilde{d} \\
& + \frac{2}{\epsilon}10^6 n^3 p^3 q^2\Big( \fint \tilde{a}^2+ \fint \tilde{e}^2\Big) \\
& + C(\epsilon ,n,p_0) p^{\frac{n}{2p_0- n}+ 1} q^2 A_0^{\frac{n^2+ 3n}{2p_0- n}}K_1^{\frac{2p_0}{2p_0- n}}(r_0^{\frac{2n}{2p_0- n}}+ 1) \\
&\quad \cdot \Big(\fint \xi^{2q- 2}\chi^{2p'+ 2}f^p \Big)
\end{array}\right.
}
\end{equation}

On the other hand, 
\begin{equation}\nonumber
{\left.
\begin{array}{rl}
\Big|\nabla (\xi^{q}\chi^{p'+1}f^{\frac{p}{2}})\Big|^2 &\leq (q\tilde{e}+ (p'+ 1)\tilde{a}+ \frac{p}{2}\tilde{b})^2 \\
& \leq \big(2+ \frac{1}{2\epsilon}\big)q^2\tilde{e}^2+ \big(2+ \frac{1}{2\epsilon}\big)(p'+ 1)^2\tilde{a}^2+ \big(\frac{1}{4}+ \epsilon \big)p^2\tilde{b}^2
\end{array}\right.
}
\end{equation}

Then 
\begin{equation}\label{5.13}
{\fint \tilde{b}^2 \geq \Big[(\frac{1}{4}+ \epsilon)p^2 \Big]^{-1}\tilde{d}- \Big[\big(\frac{1}{4}+ \epsilon \big)p^2 \Big]^{-1} \Big(2+ \frac{1}{2\epsilon} \Big) (q+ p'+ 1)^2 \fint (\tilde{a}^2+ \tilde{e}^2)
}
\end{equation}

By (\ref{5.12}), (\ref{5.13}), and estimate $\fint \tilde{e}^2$ using $|\tilde{\nabla}\xi_i|^2 \leq C(r_0, N)= \frac{8}{r_0^2}(N+ 1)$ as in Lemma \ref{lem 5.5}, $0\leq \xi_i\leq 1$, $0\leq \chi\leq 1$, also note $1+ 4\epsilon- p< 0$,
\begin{align}
&\frac{\partial}{\partial t}\Big(\fint \xi^{2q}\chi^{2p'}f^p \Big)+ \Big(\frac{4(p- 1- 4\epsilon)}{(1+ 4\epsilon)p}- 100\epsilon \Big) \tilde{d}  \nonumber \\
&\quad \leq  C(\epsilon ,n,p_0) p^{\frac{n}{2p_0- n}+ 3} q^2 A_0^{\frac{(2p_0+ 3)n}{2p_0- n}}K_1^{\frac{2p_0}{2p_0- n}}[r_0^2+ r_0^{\frac{2n}{2p_0- n}}]\fint \xi^{2q- 2}\chi^{2p'}f^p \label{5.14}
\end{align}

Choose $\epsilon= \frac{1}{1000}$ in (\ref{5.14}), then
\begin{equation}\label{5.15}
{\left.
\begin{array}{rl}
& \frac{\partial}{\partial t}\Big(\fint \xi^{2q}\chi^{2p'}f^p \Big)+ \frac{1}{10} \Big(\fint |\nabla(\xi^{q}\chi^{p'+ 1} f^{\frac{p}{2}})|^2 \Big) \\
& \quad \leq C(n,p_0) p^{\frac{n}{2p_0- n}+ 3} q^2 A_0^{\frac{(2p_0+ 3)n}{2p_0- n}}K_1^{\frac{2p_0}{2p_0- n}}\Big[r_0^2+ r_0^{\frac{2n}{2p_0- n}} \Big]\fint \xi^{2q- 2}\chi^{2p'}f^p 
\end{array}\right.
}
\end{equation}

Replace $f$ with $|Rm|$, the lemma is proved.
}
\qed

For simplicity reason, we use $C_1$ intead of $C_1(A_0,K_1,n,p_0,r_0)$ in the rest of this section, similar for $C_2$ etc. In the next lemma, we use the energy estimate and modified Moser iteration to get local $C^0$-estimate of $|\chi^2 Rm(g(t))|$.

\begin{lemma}\label{lem 5.7}
{If $0< t< T_{\max}$, then there exists some positive constant 
\begin{equation}\label{5.16}
{C_2(A_0,K_1,n,p_0,r_0)= C(n, p_0) A_0^{\Big(\frac{3}{2p_0}+ 1\Big)n}\Big(r_0^{\frac{n}{p_0}}K_1\Big)
}
\end{equation}
such that: 
\begin{equation}\label{5.17}
{\left.
\begin{array}{rl}
|\chi^2 Rm(g(t))| \leq C_2 \Big[C_1^{\frac{n+ 2}{2p_0}}t^{\frac{1}{p_0}}+ t^{-\frac{n}{2p_0}}\Big]
\end{array}\right.
}
\end{equation}
}
\end{lemma}

\pf
{Let $f(x,t)=|Rm(x)|_{g(t)}$. Given any $t_1$ and $t_2$ such that $0<t_1 < t_2 < T_{\max}$, set
\begin{equation} \nonumber
\psi(t) = \left\{
\begin{array}{rl}
0                           &  0   \leq  t     \leq t_1     \\
\frac{t- t_1}{t_2- t_1} &  t_1 \leq  t     \leq t_2    \\
1                           &  t_2 \leq t     < T_{\max} \\
\end{array} \right.
\end{equation}

By Lemma \ref{lem 5.6}
\begin{equation}\label{5.18}
{\left.
\begin{array}{rl}
& \frac{\partial}{\partial t}\Big(\fint \xi_i^{2q} \chi^{2p'} f^p dV_{g(t)}\Big)+ \frac{1}{10} 
 \Big(\fint \big|\nabla (\xi_i^{q} \chi^{p'+1} f^{\frac{p}{2}}) \big|^2 dV_{g(t)} \Big)\\
& \quad \leq C_1 p^{\frac{n}{2p_0- n}+ 3} q^2 \Big(\fint \xi_i^{2q- 2} \chi^{2p'} f^{p} dV_{g(t)} \Big)
\end{array}\right.
}
\end{equation}

Multiply (\ref{5.18}) by $\psi(t)$, integrate it from $0$ to $t_0$ with respect to $t$, where $t_2\leq t_0< T_{\max}$. Then
\begin{equation} \label{5.19}
{\left.
\begin{array}{rl}
& \fint (\xi^{2q}\chi^{2p'} f^{p}) dV_{g(t_0)} + \frac{1}{10} \int_{t_2}^{t_0} \fint |\nabla (\xi^q \chi^{p'+ 1} f^{\frac{p}{2}})|^2 \\
& \leq C A_0^{\frac{(2p_0+ 3)n}{2p_0- n}}K_1^{\frac{2p_0}{2p_0- n}} \Big[r_0^2+ r_0^{\frac{2n}{2p_0- n}} \Big]p^{\frac{n}{2p_0- n}+ 3} q^2 \Big(1 + \frac{1}{t_2- t_1}\Big) \int_{t_1}^{t_0} \\
&\quad \cdot \Big( \fint \xi^{2(q-1)} \chi^{2p'} f^{p} \Big) 
\end{array}\right.
}
\end{equation}
Note (\ref{5.19}) is in fact valid for any $0< t_1< t_2\leq t_0< T_{\max}$.

For $0< t\leq t_0$, denote 
\begin{equation} \nonumber
{H(p,p',q,t)= \int_{t}^{t_0} \fint \xi^{2q}\chi^{2p'} f^p 
}
\end{equation}
 
then for any $0< \iota< \iota'\leq t_0< T_{\max}$, using (\ref{5.6}), (\ref{5.19}) and H\"older inequality,
\begin{equation} \nonumber
{\left.
\begin{array}{rl}
& H(p(1+ \frac{2}{n}), p'(1+ \frac{2}{n})+ 1, q(1+ \frac{2}{n}), \iota')= \int_{\iota'}^{t_0} \fint \big( \xi^{2q} \chi^{2p'} f^p\big)^{\frac{2}{n}} \big( \xi^q \chi^{p'+ 1} f^{\frac{p}{2}}\big)^2 \\

&\quad \leq \int_{\iota'}^{t_0} \Big[\fint (\xi^{2q} \chi^{2p'} f^{p}) \Big]^{\frac{2}{n}}\cdot 4A_0 r^2 \Big[ \fint |\nabla (\xi^q \chi^{p'+1} f^{\frac{p}{2}})|^2 \Big] dt\\

&\quad \leq 40A_0r_0^2 \Big[C_1p^{\frac{n}{2p_0- n}+ 3}q^2+ \frac{1}{\iota '- \iota}\Big]^{1+ \frac{2}{n}} H(p,p',q-1,\iota)^{1+ \frac{2}{n}}
\end{array}\right.
}
\end{equation}

Denote $\alpha= \frac{n}{2p_0- n}+ 3$, $\nu= 1+ \frac{2}{n}$, $\eta= \nu^{n+2}$, fix $t\in (0, t_0)$, set 
\begin{equation}\nonumber
{\left.
\begin{array}{rl}
& p_k= p_0 \nu^k \ , \quad p_k'= p_k- \frac{n}{2}\ , \quad p_0'= p_0- \frac{n}{2}\ , \\
& q_k= n\nu^k- (\frac{n}{2}+ 1)\ , \quad q_0= \frac{n}{2}- 1\ , \quad \iota_k= t(1- \eta^{-k})\ , \\
& H_k= H(p_k,p_k',q_k,\iota_k)\ ,\quad  \Phi_k= H_k^{\frac{1}{p_k}}
\end{array}\right.
}
\end{equation}

\begin{equation} \nonumber
{\Phi_0= \Big(\int _0^{t_0} \fint \xi^{n-2}\chi^{2p_0- n} f^{p_0} \Big)^{\frac{1}{p_0}}
}
\end{equation}

Then for any $k\geq 0$,
\begin{equation} \nonumber
{\left.
\begin{array}{rl}
H_{k+1}& = H(p_{k+1},p_{k+1}',q_{k+ 1}, \iota_{k+1})\\
& = H\Big(p_k(1+ \frac{2}{n}), p_k'(1+ \frac{2}{n})+ 1, (q_k+ 1)(1+ \frac{2}{n}), \iota_{k+ 1}\Big) \\

&\leq 40 A_0 r_0^2 \Big[C_1 p_k^{\alpha} (q_k+ 1)^2+ \frac{1}{\iota_{k+ 1}- \iota_k}\Big]^{\nu} \cdot H(p_k,p_k',q_k,\iota_k)^{\nu} \\

&\leq 40 A_0 r_0^2 \Big[C_1 p_0^{\alpha} n^2 \nu^{(\alpha + 2)k}+ t^{-1}\frac{\eta}{\eta - 1}\eta^k \Big]^{\nu} H_k^{\nu} \\

&\leq C(n, p_0)A_0r_0^2 [C_1+ t^{-1}]^{\nu} \nu^{(\alpha + n+ 2)k\nu} H_k^{\nu}
\end{array}\right.
}
\end{equation}

Taking $p_{k+ 1}$-root on both sides of the above inequality,
\begin{equation} \nonumber
{\Phi_{k+1}\leq \Big[C(n, p_0)A_0r_0^2\Big]^{p_0^{-1}\nu^{-(k+ 1)}} \nu^{(\alpha + n+ 2)\frac{k\nu^{-k}}{p_0}} \Big[C_1+ t^{-1}\Big]^{p_0^{-1}\nu^{-k}} \Phi_k
}
\end{equation}

By induction,
\begin{equation} \nonumber
{\Phi_{k+1}\leq \Big[C(n, p_0)A_0r_0^2\Big]^{p_0^{-1}(\sigma_{k+ 1}- 1)} \nu^{(\alpha + n+ 2)\frac{\sigma_k '}{p_0}} \Big[C_1+ t^{-1}\Big]^{p_0^{-1}\sigma_k} \Phi_0
}
\end{equation}
where $\sigma_k= \sum_{i= 0}^{k}\nu^{- i}$ and $\sigma_k'= \sum_{i= 0}^{k}i \nu^{- i}$, note that $\frac{\sigma_{k+ 1}- 1}{p_0}\leq \frac{n}{2p_0}$

Then 
\begin{align}
|\xi_i^{\frac{n}{p_0}} \chi^2 f(x,t)| &\leq \lim_{k\rightarrow \infty} \Phi_{k+1}\nonumber \\
&\leq C(n, p_0)A_0^{\frac{n}{2p_0}} \Big[C_1+ t^{-1}\Big]^{\frac{n+ 2}{2p_0}} r_0^{\frac{n}{p_0}}\cdot \Big(\int_0^{t_0}\fint_{B_i\cap\Omega} \xi_i^{n- 2} \chi^{2p_0- n} f^{p_0} \Big)^{\frac{1}{p_0}} \nonumber
\end{align}

Let $t_0\rightarrow t$ in the above and using (\ref{5.8}), 
\begin{align}
|\xi_i^{\frac{n}{p_0}} \chi^2 f(x,t)| &\leq C(n, p_0)A_0^{\frac{n}{2p_0}} \Big[C_1+ t^{-1}\Big]^{\frac{n+ 2}{2p_0}} r_0^{\frac{n}{p_0}} \Big[(2N)^{\frac{1}{p_0}}K_1\Big] t^{\frac{1}{p_0}} \nonumber \\
&\leq C(n, p_0)A_0^{\frac{3n}{2p_0}} \Big(r_0^{\frac{n}{p_0}}K_1\Big) \Big[C_1^{\frac{n+ 2}{2p_0}} t^{\frac{1}{p_0}}+ t^{-\frac{n}{2p_0}}\Big]   \nonumber 
\end{align}

Now for any $x\in \Omega$, we assume $x\in B_{i_0}\cap\Omega$, note 
$|\xi_i^2\chi^2 f|\leq |\xi_i^{\frac{n}{p_0}}\chi^2 f|$, then
\begin{align}
|\chi^2 f|(x,t)= |\sum_{j= 1}^{N}\xi_j^2 \chi^2 f(x,t)|\leq C(n, p_0)A_0^{\Big(\frac{3}{2p_0}+ 1\Big)n} \Big(r_0^{\frac{n}{p_0}}K_1\Big)\Big[C_1^{\frac{n+ 2}{2p_0}} t^{\frac{1}{p_0}}+ t^{-\frac{n}{2p_0}}\Big] \nonumber 
\end{align} 
Replace $f$ with $|Rm|$, this is our conclusion.
}
\qed

Straightforward from Lemma \ref{lem 5.7}, we have the following two corollaries.

\begin{cor}\label{cor 5.8}
{If $0< t< T_{\max}$, there exists some positive constant 
\begin{align}
C_3(A_0,K_1,n,p_0,r_0)= C(n, p_0)A_0^{\Big(\frac{3}{2p_0}+ 1\Big)n} \Big(r_0^{\frac{n}{p_0}}K_1\Big) \label{5.20}
\end{align} 
such that 
\begin{equation}\label{5.21}
{\left.
\begin{array}{rl}
|\chi^2 Rc(g(t))| \leq C_3 \Big(t^{-\frac{n}{2p_0}}+ C_1^{\frac{n+ 2}{2p_0}} t^{\frac{1}{p_0}} \Big)
\end{array}\right.
}
\end{equation}
}
\end{cor}

\begin{cor}\label{cor 5.9}
{If $0\leq t< T_{\max}$, there exists some positive constant $C$, which is independent of $t$, such that 
\begin{equation}\nonumber
{|\chi^2 Rm(g(t))| \leq C
}
\end{equation}
}
\end{cor}

Choose $T_1> 0$ such that 
\begin{align}
\exp{\Big(C_1p_0^{\frac{n}{2p_0- n}+ 3}NT_1\Big)}= \frac{3}{2} \nonumber
\end{align}
where $C_1$ is from Lemma \ref{lem 5.6} and $N$ is from Lemma \ref{lem 5.4}. Then
\begin{align}
T_1= C(n, p_0)A_0^{-\frac{5p_0n}{2p_0- n}}K_1^{-\frac{2p_0}{2p_0- n}} \Big[r^{-2}+ r_0^{\frac{2n}{2p_0- n}}\Big]^{-1} \label{5.22}
\end{align}

\begin{prop} \label{prop 5.10} 
{If $0\leq t< T_{\max}< T_1$, then
\begin{equation}\nonumber
{\Big(\frac{1}{V_{g(t)}(B_i)} \int_{B_i\cap \Omega} |Rm(g(t))|^{p_0}dV_{g(t)}\Big)^{\frac{1}{p_0}}\leq \Big( \frac{3}{2}N \Big)^{\frac{1}{p_0}} K_1, \ \quad  i= 1, 2, 3, \cdots
}
\end{equation}
}
\end{prop}

\pf
{Set $f(x,t)\doteqdot |Rm(x)|_{g(t)}$. Let $q= 1$, $p'=0$, $p= p_0$ in Lemma \ref{lem 5.6}, then:
\begin{equation}\nonumber
{\frac{\partial}{\partial t}\Big(\fint_{\Omega} \xi_i^{2} f^{p_0} dV_{g(t)}\Big)\leq  C_1 p_0^{\frac{n}{2p_0- n}+ 3} \fint_{B_i\cap \Omega} f^{p_0} dV_{g(t)}
}
\end{equation}

For $\Omega$, define $Cov(\Omega)= \{B_i| B_i\cap \bar{\Omega}\neq \emptyset \}$, note that $Cov(\Omega)$ has only a finite number of elements. We define
\begin{equation}\nonumber
{\phi(t)= \sup_{B_i\in Cov(\Omega)} \Big[ \fint_{B_i\cap \Omega} \xi_i^2 |Rm|^{p_0} dV_{g(t)}\Big]
}
\end{equation}

Assume $\phi(t_0)= \fint_{B_{i_0}\cap \Omega} \xi_{i_0}^2 |Rm|^{p_0} dV_{g(t_0)}$, then 
\begin{equation}\label{5.23}
{\left.
\begin{array}{rl}
\frac{\partial}{\partial t}\phi(t_0) &\leq \frac{\partial}{\partial t} \Big( \fint_{\Omega} \xi_{i_0}^2 f^{p_0}\Big)(t_0) \leq C_1 p_0^{\frac{n}{2p_0- n}+ 3} \fint_{B_{i_0}\cap \Omega} f^{p_0} dV_{g(t_0)} \\
& \leq C_1 p_0^{\frac{n}{2p_0- n}+ 3} \sum_{j= 1}^{N} \Big( \int_{B_{i_0}\cap \Omega} \xi_{j}^2 f^{p_0} dV_{g(t_0)} \Big) \leq C_1 p_0^{\frac{n}{2p_0- n}+ 3} N \phi(t_0)
\end{array}\right.
}
\end{equation}

Because $t_0$ is chosen freely,
\begin{equation}\label{5.24}
{\left.
\begin{array}{rl}
\frac{\partial}{\partial t}\phi(t)\leq \Big[C_1 p_0^{\frac{n}{2p_0- n}+ 3} N \Big]\phi(t)
\end{array}\right.
}
\end{equation}

By (\ref{5.1}), $\phi(0)\leq \sup_{B_i}\fint_{B_i\cap\Omega} |Rm|^{p_0}dV_{g_0}\leq K_1^{p_0}$, then by the definition of $T_1$ and (\ref{5.24}),  
\[\phi(t)\leq \frac{3}{2} K_1^{p_0}\]

Now 
\begin{equation}\nonumber
{\left.
\begin{array}{rl}
\fint_{B_i\cap\Omega}|Rm|^{p_0}dV_{g(t)}= \sum_{j=1}^{N} \fint_{B_i\cap\Omega}\xi_j^2 f^{p_0}\leq N\phi(t)\leq \frac{3}{2}N K_1^{p_0}
\end{array}\right.
}
\end{equation}
}
\qed

Define 
\begin{align}
T_{1, 1}\doteqdot \Big[\frac{(2p_0- n)\ln 2}{8p_0(n+ 1)} C_3^{-1}\Big]^{\frac{2p_0}{2p_0- n}}, \quad T_{1, 2}\doteqdot \Big[\frac{p_0\ln 2}{4(p_0+ 1)(n+ 1)}C_1^{-1}C_3^{-(\frac{n+ 2}{2p_0})}\Big]^{\frac{p_0}{p_0+ 1}} \nonumber
\end{align}
and recall $T_1$ is defined in (\ref{5.22}), 
\begin{align}
T_2\doteqdot \min\{T_1, T_{1, 1}, T_{1, 2}\}\label{5.25}
\end{align}

\begin{prop}\label{prop 5.11}
{If $0\leq t< T_{\max}< T_2$, then 
\begin{equation}\label{5.26}
{2^{- \frac{1}{n+1}}g_0 \leq g(t)\leq 2^{\frac{1}{n+ 1}}g_0
}
\end{equation}
}
\end{prop}

\pf
{Choose $0\leq t_0< T_{\max}< T_2$. then by Corollary \ref{cor 5.8}, we have
\begin{equation}\nonumber
{\left.
\begin{array}{rl}
\int_0^{t_0}|\frac{\partial}{\partial t}g| dt = \int_0^{t_0}|2\chi^2 Rc(g(t))|dt \leq \int_0^{t_0} 2C_3 \Big(t^{-\frac{n}{2 p_0}}+ C_1^{\frac{n+ 2}{2p_0}}t^{\frac{1}{p_0}} \Big) dt\leq \frac{\ln 2}{n+ 1}
\end{array}\right.
}
\end{equation}
By Lemma $6.49$ in \cite{RI}, we get our conclusion.
}
\qed

\begin{prop}\label{prop 5.12}
{If $0\leq t< T_{\max}< T_2$, then 
\begin{equation}\label{5.27}
\Big(\fint_{B_x(4r_0)} h^{\frac{2n}{n-2}} dV_{g(t)} \Big)^{\frac{n-2}{n}}\leq  2A_0 r^2 \fint_{B_x(4r_0)}|\nabla h|^2 dV_{g(t)}, \quad \forall h\in C_0^{\infty}(B_x(4r_0))
\end{equation}
}
\end{prop}

\pf
{By Lemma \ref{lem 5.1} and Proposition \ref{prop 5.11}, we get our conclusion.
}
\qed

\begin{prop}\label{prop 5.13}
{If $T_{\max}< T_2$, then on $[0, T_{\max})$, we have
\begin{equation}\label{5.28}
\left\{
\begin{array}{rl}
\Big(\fint_{B_x(4r_0)} h^{\frac{2n}{n-2}} dV_{g(t)} \Big)^{\frac{n-2}{n}}&\leq  2A_0 r^2 \fint_{B_x(4r_0)}|\nabla h|^2 dV_{g(t)}, \\
2^{- \frac{1}{4}}g_0 &\leq g(t)\leq 2^{\frac{1}{4}}g_0 \\
\Big(\fint_{B_i\cap \Omega} |Rm(g(t))|^{p_0}dV_{g(t)}\Big)^{\frac{1}{p_0}} &\leq \Big( \frac{3}{2}N \Big)^{\frac{1}{p_0}} K_1 \\
|\chi^2 Rm(g(t))| &\leq C
\end{array} \right.
\end{equation}
for any $x\in M$ and $h\in C_0^{\infty}(B_x(4r_0))$, where $C$ is independent of $t$.
}
\end{prop}

\pf
{Combining Propositions \ref{prop 5.10}, \ref{prop 5.11}, \ref{prop 5.12} and Corollary \ref{cor 5.9}, we get our conclusion.
}
\qed

\begin{theorem}\label{thm 5.14}
{Assume $(M^n, g_0)$ is a $n$-dimensional $(n\geq 3)$ complete noncompact Riemannian manifold, which satisfies (\ref{5.1}). Then local Ricci flow (\ref{2.1}) has a smooth solution on $[0, \frac{T_2}{2}]$, where $T_2$ is defined in (\ref{5.25}). Moreover, for $t\in [0, \frac{T_2}{2}]$, the metric satisfies (\ref{5.7}) and the curvature tensors satisfy the bounds (\ref{5.17}) and (\ref{5.21}).
}
\end{theorem}

\pf
{Firstly, by Theorem \ref{thm 2.7} and DeTurck's trick, the equation (\ref{2.1}) has a smooth solution on a sufficiently small time interval starting at $t=0$. Let $[0, T_{\max})$ be a maximal time interval on which (\ref{2.1}) has a smooth solution and such that (\ref{5.6}), (\ref{5.7}), (\ref{5.8}) hold for each metric $g(t)$. We claim that $T_{\max}> \frac{T_2}{2}$, we prove it by contradiction. 

If $T_{\max}\leq \frac{T_2}{2}$, then $T_{\max}< T_2$. By Proposition \ref{prop 5.13}, $|\chi^2 Rm|_{(g(t),\infty)}$ is uniformly bounded on $[0, T_{\max})$. Then by Theorem \ref{thm W10}, we can extend the solution of local Ricci flow to $[0, T_{\max}+ \delta)$, where $\delta> 0$ is some constant. By Proposition \ref{prop 5.13}, we can furthermore assume that (\ref{5.6}), (\ref{5.7}), (\ref{5.8}) hold on $[0, T_{\max}+ \delta)$, so it is the contradiction with the definition of $T_{\max}$. Then we have $T_{\max}> \frac{T_2}{2}$, and (\ref{2.1}) has a smooth solution on $[0, \frac{T_2}{2}]$. We prove the first conclusion and $g(t)$ satisfies (\ref{5.7}).

Note $\frac{T_2}{2}< T_{\max}$, then by Lemma \ref{lem 5.7} and Corollary \ref{cor 5.8}, we get (\ref{5.17}), (\ref{5.21}).
}
\qed

\begin{remark}\label{rem 5.15}
{In fact we have the following conclusion by the observation that the power ration between $K_1$ and $r_0$ is exactly $\frac{p_0}{n}$ in local Ricci flow’s existence time $\frac{T_2}{2}$ (in fact in $T_1$, $T_{1, 1}$ and $T_{1, 2}$ ).

Assume $(M^n, g_0)$ is a $n$-dimensional $(n\geq 3)$ complete noncompact Riemannian manifold, which satisfies:
\begin{equation}\label{5.29}
\left\{
\begin{array}{rl}
Rc \geq 0, \\
\lim_{r\rightarrow \infty} \frac{V(B_{x_0}(r))}{r^n}= \delta > 0, \\
\Big(\int_{M^n} |Rm(g_0)|^{p_0} dV_{g_0}\Big)^{\frac{1}{p_0}} \leq C_0< \infty
\end{array} \right.
\end{equation}
where $x_0\in M^n$ is a fixed point, $p_0 > \frac{n}{2}$ and $C_0$ are positive constants.

Let $\Omega_r \doteqdot B_{x_0}(r)$, $r\geq 1$ is positive constant, then local Ricci flow (\ref{2.1}) with repect to $\Omega_r$ has a smooth solution on $[0, T]$, where $T > 0$ is independent of $r$.

The proof of the above conclusion is similar to Theorem \ref{thm 5.14}, but not using $\{\xi_i\}_{i= 1}^{\infty}$ in the whole procedure.
}
\end{remark}


\section{Existence time estimates (The second case: Scale-Invariant Exponent)}  

In this section, we assume that $(M^n, g_0)$ is a $n$-dimensional $(n\geq 3)$ complete noncompact Riemannian manifold satisfying the following assumptions:
\begin{equation}\label{6.1}
{\left\{
\begin{array}{rl}
V_{g_0}(B_x(\rho)) &\geq \widehat{N}_1 \Big(\frac{\rho}{4r_0}\Big)^n V_{g_0}(B_x(4r_0)), \quad 0<\rho\leq 4r_0, \\
\Big(\int_{B_{x}(r_0)} h^{\frac{2n}{n- 2}} dV_{g_0} \Big)^{\frac{n- 2}{n}} &\leq \widehat{A}_0 \int_{B_{x}(r_0)} |\nabla h|^2 dV_{g_0} , \\
\Big(\int_{B_x(r_0)} |Rm(g_0)|^{\frac{n}{2}} dV_{g_0} \Big)^{\frac{2}{n}} &\leq \Big(\tau n \widehat{A}_0\Big)^{-1} , \\
\Big(\int_{B_x(r_0)} |Rc(g_0)|^{p_0} dV_{g_0} \Big)^{\frac{1}{p_0}} &\leq  \widehat{K}_1
\end{array}\right.
}
\end{equation}
for any $x\in M^n$, $h(x)\in C_0^{\infty} (B_x(r_0))$, where $p_0> \frac{n}{2}$, $\widehat{A}_0\geq 1$, $\widehat{N}_1\leq 1$, $\widehat{K}_1$ and $r_0$ are positive constants, $\tau= 300\widehat{N} p_0$ and $\widehat{N}= 2^{4n}\widehat{N}_1^{-2}\geq 1$ as in Lemma \ref{lem 5.4}, but using $\widehat{N}_1$ not $N_1$.

By the doubling property assumption in (\ref{6.1}), we can choose $B_i$ as in Lemma
\ref{lem 5.4}. By similar argument as in the section $5$, we can assume that $[0, T_{\max})$ is the maximal time interval on which the local Ricci flow (\ref{2.1}) has a smooth solution and the following holds for each metric $g(t)$ (where $g(0) = g_0$):
\begin{align}
\Big(\int_{B_x(r_0)} h^{\frac{2n}{n- 2}} dV_{g(t)}\Big)^{\frac{n- 2}{n}} &\leq 4\widehat{A}_0 \int_{B_x(r_0)} |\nabla h|^2 dV_{g(t)} ; \label{6.2} \\
\frac{1}{2}g_0\leq g(t)\leq 2g_0 ; \label{6.3} \\
\Big(\int_{B_i\cap \Omega} |Rm(g(t))|^{\frac{n}{2}} dV_{g(t)}\Big)^{\frac{2}{n}} &\leq \Big(2\widehat{N}\Big)^{\frac{2}{n}} \Big(\tau n \widehat{A}_0\Big)^{-1}, \quad i=1, 2, 3, \cdots \label{6.4}
\end{align}
for any $x\in M^n$, $h\in C_0^{\infty}(B_x(r_0))$.

\begin{lemma}\label{lem 6.1}
{When $0\leq t< T_{\max}$, for any $i$, $q\geq 1$, $p\geq \frac{n}{2}$, $p'\geq 0$ and $p\geq p'$, there exists constant $\widehat{C}_1(n, \widehat{N}_1, r_0)= C(n, \widehat{N}_1)(r_0^{-2}+ 1)> 0$ such that:
\begin{align}
&\frac{\partial}{\partial t}\Big(\int_{\Omega} \xi_i^{2q} \chi^{2p'} |Rm|^p dV_{g(t)}\Big)+ \theta \Big(\int_{\Omega} |\nabla (\xi_i^q \chi^{p'+ 1} |Rm|^{\frac{p}{2}})|^2 dV_{g(t)} \Big) \nonumber \\
&\quad \leq \frac{1}{\epsilon} \widehat{C}_1(n, \widehat{N}_1, r_0)p^3q^2 \Big(\int_{B_i\cap \Omega} \xi_i^{2q- 2} \chi^{2p'} |Rm|^p dV_{g(t)} \Big) \label{6.5} 
\end{align}
where 
\begin{align}
\theta= \Big(\frac{4(p- 1- 4\epsilon)}{(1+ 4\epsilon)p}\Big)- 48p\widehat{A}_0 |Rm|_{\big(g(t),\  L^{\frac{n}{2}}(B_i\cap \Omega)\big)} \label{6.6}
\end{align}
and $|Rm|_{\big(g(t),\  L^{\frac{n}{2}}(B_i\cap \Omega)\big)}\doteqdot \Big(\int_{B_i\cap \Omega} |Rm(x, t)|^{\frac{n}{2}} dV_{g(t)}\Big)^{\frac{2}{n}}$, $\epsilon> 0$ is any positive constant.
}
\end{lemma}

\pf
{Let $f(x, t)= |Rm(x, t)|_{g(t)}$, 
\begin{align}
\frac{\partial}{\partial t}\Big(\int_{\Omega} \xi^{2q} \chi^{2p'} f^p\Big) \leq \sum_{k= 1}^6 I_k \label{6.7}
\end{align}
where 
\begin{align}
& I_1= p\int \xi^{2q} \chi^{2p'+ 2} f^{p- 1} \Delta f ,\quad I_2= p\int \xi^{2q} \chi^{2p'+ 2} f^{p- 2} |\nabla f|^2 \nonumber \\
& I_3= -(1- \epsilon) p\int \xi^{2q} \chi^{2p'+ 2} f^{p- 2} |\nabla Rm|^2, \quad I_4= 10p \int \xi^{2q} \chi^{2p'+ 2} f^{p+ 1}- \int \xi^{2q} \chi^{2p'+ 2} f^p R \nonumber \\
& I_5= \Big(\frac{10^5}{\epsilon}\Big) np\int \xi^{2q} \chi^{2p'} |\nabla \chi|^2 f^p ,\quad I_6= 8p\int \xi^{2q} \chi^{2p'+ 1} f^{p- 2} g^{ri} g^{sj} g^{pk} g^{ql} R_{rspq} \chi_{il} R_{jk} \nonumber 
\end{align}
For simplification, we set the following notations:
\begin{align}
& \tilde{a}= \xi^{q}\chi^{p'} f^{\frac{p}{2}} |\nabla \chi | ; \quad \tilde{b}= \xi^{q} \chi^{p'+ 1} f^{\frac{p}{2}- 1} |\nabla f|; \nonumber  \\
& \tilde{c}= \xi^{q}\chi^{p'+ 1} f^{\frac{p}{2}- 1} |\nabla Rm | ; \quad \tilde{d}= \int \big|\nabla (\xi^q\chi^{p'+ 1} f^{\frac{p}{2}}) \big|^2 ; \quad  \tilde{e}= \xi^{q- 1} \chi^{p'+ 1} f^{\frac{p}{2}} |\nabla \xi |; \nonumber  
\end{align}

Note $|R|\leq n|Rm|= 2pf$, 
\begin{align}
I_4 &\leq 12p \int \xi^{2q} \chi^{2p'+ 2} f^{p+ 1} \leq 12p |f|_{\big(g(t), L^{\frac{n}{2}}(B_i\cap \Omega)\big)}\cdot 4\widehat{A}_0\cdot \int \big| \nabla (\xi^q \chi^{p'+ 1} f^{\frac{p}{2}})\big|^2 \nonumber \\
& \leq 48p \widehat{A}_0 |f|_{\big(g(t), L^{\frac{n}{2}}(B_i\cap \Omega)\big)} \cdot \tilde{d} \nonumber 
\end{align}
Similar to the proof of Lemma \ref{lem 5.1}, 
\begin{align}
\frac{\partial}{\partial t} \Big(\int_{\Omega} \xi^{2q} \chi^{2p'} f^p\Big) &\leq p(1+ 4\epsilon- p) \int \tilde{b}^2+ 48p\widehat{A}_0 || f||_{\frac{n}{2},\ B_i\cap \Omega} \cdot \tilde{d} \nonumber \\
& \quad + \frac{2}{\epsilon} 10^6n^3p^3q^2 \Big(\int \tilde{a}^2+ \int \tilde{e}^2 \Big) \label{6.8}
\end{align}
Then by (\ref{5.13}), 
\begin{align}
&\frac{\partial}{\partial t}\Big(\int \xi^{2q} \chi^{2p'} f^p\Big)+ \theta \tilde{d} \leq \frac{4}{\epsilon} 10^6n^3p^3q^2 \int (\tilde{a}^2+ \tilde{e}^2) \nonumber \\
&\quad \leq \frac{8}{\epsilon} 10^6n^3 p^3 q^2 \Big(\int_{\Omega} \xi^{2q} \chi^{2p'} |Rm|^p dV_{g(t)}+ \int_{\Omega} \xi^{2q- 2}\chi^{2p'+ 2} |Rm|^p |\nabla \xi|^2 dV_{g(t)} \Big) \nonumber 
\end{align} 
where $\theta$ is defined in (\ref{6.6}). Replace $f$ with $|Rm|$, by Lemma \ref{lem 5.5}, $|\tilde{\nabla} \xi_i|^2\leq \frac{8}{r_0^2} (\widehat{N}+ 1)$, the lemma is proved.
}
\qed

\begin{lemma}\label{lem 6.2}
{When $0< t< T_{\max}$, for $i= 1, 2, 3, \cdots$, there exists a constant $C(n, p_0, \widehat{N}_1)$ such that 
\begin{align}
\int_{\Omega} \xi_i^4 \chi^2 |Rm(x)|^{1+ \frac{n}{2}} dV_{g(t)} \leq C(n, p_0, \widehat{N}_1) \big(r_0^{-2}+ 1\big) \big(\widehat{A}_0\big)^{-\frac{n}{2}} (t+ 1)^2 t^{-1}
\end{align}
}
\end{lemma}

\pf
{Let $f(x, t)= |Rm(x)|_{g(t)}$. In Lemma \ref{lem 6.1}, let $q= 1$, $p'= 0$, $p= \frac{n}{2}$, $\epsilon= \frac{1}{16}$, plug the above values into (\ref{6.6}), we get 
$\theta\geq \frac{1}{6}$, then 
\begin{align}
\frac{\partial}{\partial t}\Big(\int_{\Omega} \xi^2 f^{\frac{n}{2}} dV_{g(t)}\Big)+ \frac{1}{6} \int_{\Omega} |\nabla (\xi \chi f^{\frac{n}{4}})|^2 dV_{g(t)} \leq \frac{1}{\epsilon_0} C(n, \widehat{N}_1)(r_0^{-2}+ 1)\Big(\int_{B_{i}\cap\Omega} f^{\frac{n}{2}} dV_{g(t)} \Big) \nonumber
\end{align}

Multiply the above by $\psi_{t_1, t_2}(t)$, which is define in the proof of Lemma \ref{5.7}. Integrate from $0$ to $t_0$ with respect to $t$, where $t_0$ satisfies $t_2\leq t_0< T_{\max}$. Then let $t_1= \frac{t_0}{4}$, $t_2= \frac{t_0}{2}$, by (\ref{6.4}),
\begin{align}
\int_{\frac{t_0}{2}}^{t_0} \int_{\Omega} |\nabla (\xi \chi f^{\frac{n}{4}})|^2 &\leq C(n, \widehat{N}_1) (r_0^{-2}+ 1) (1+ \frac{1}{t_0}) \int_{\frac{t_0}{4}}^{t_0} \int_{B_i\cap\Omega} f^{\frac{n}{2}} \nonumber \\
& \leq C(n, p_0, \widehat{N}_1) (r_0^{-2}+ 1)\big(\widehat{A}_0\big)^{-\frac{n}{2}} \Big(1+ \frac{1}{t_0}\Big)t_0 \label{6.10}
\end{align} 

Again in Lemma \ref{lem 6.1}, let $q= 2$, $p'= 1$, $p= 1+ \frac{n}{2}$, choose $\epsilon= \frac{1}{16}$, then $\theta\geq \frac{1}{3}> 0$ and 
\begin{align}
\frac{\partial}{\partial t}\Big(\int_{\Omega} \xi^4 \chi^2 f^{1+ \frac{n}{2}} dV_{g(t)}\Big) \leq C(n, \widehat{N}_1) (r_0^{-2}+ 1) \Big(\int_{\Omega} \xi^2 \chi^2 f^{1+ \frac{n}{2}} dV_{g(t)} \Big) \nonumber 
\end{align}

Multiply the above with $\psi$, integrate from $0$ to $t_0$, where $t_2\leq t_0< T_{\max}$, then let $t_1= \frac{t_0}{2}$, $t_2= t_0$, from (\ref{6.2}), (\ref{6.4}) and (\ref{6.10})
\begin{align}
\int_{\Omega} \xi^4 \chi^2 f^{1+ \frac{n}{2}} dV_{g(t_0)} &\leq C(n, \widehat{N}_1) (r_0^{-2}+ 1) \big(1+ \frac{1}{t_0}\big) \Big(\int_{\frac{t_0}{2}}^{t_0} \int_{\Omega} \xi^2 \chi^2 f^{1+ \frac{n}{2}} \Big) \nonumber \\
&\leq C(n, p_0, \widehat{N}_1) (r_0^{-2}+ 1) \big(\widehat{A}_0 \big)^{-\frac{n}{2}} \Big(1+ \frac{1}{t_0}\Big)^2 t_0  \nonumber
\end{align}

Because we choose $t_0$ arbitrarily, for any $0< t< T_{\max}$ 
\begin{align}
\int_{\Omega} \xi^4 \chi^2 f^{1+ \frac{n}{2}} dV_{g(t)} \leq C(n, p_0, \widehat{N}_1) (r_0^{-2}+ 1) \big(\widehat{A}_0\big)^{-\frac{n}{2}} (t+ 1)^2 t^{-1} \nonumber
\end{align}
Replace $f$ with $|Rm|$, we get our conclusion.
}
\qed

\begin{lemma}\label{lem 6.3}
{When $0< t< T_{\max}$, $q\geq 2$, $p\geq \frac{n}{2}$, $p'\geq 0$ and $p\geq p'$, there exists a constant $C(n, p_0, \widehat{N}_1)$ such that 
\begin{align}
& \frac{\partial}{\partial t}\Big(\int_{\Omega} \xi_i^{2q} \chi^{2p'} |Rc|^p\Big)+ \frac{1}{9} \Big(\int_{\Omega} \big| \nabla(\xi_i^q\chi^{p'+ 1} |Rc|^{\frac{p}{2}})\big|^2\Big) \nonumber \\
& \leq C(n, p_0, \widehat{N}_1)p^{2+ \frac{n}{2}} q^2 \big(r_0^{-2}+ 1\big) (t+ 1)^2 t^{-1}\Big( \int_{B_i\cap \Omega} \xi_i^{2(q- 2)} \chi^{2p'} |Rc|^p \Big) \label{6.11}
\end{align}
There also exists $\widehat{C}_2(n, \widehat{N}_1)> 0$ such that:
\begin{align}
&\frac{\partial}{\partial t}\Big(\int_{\Omega} \xi_i^2 \chi^{2p'} |Rc|^{p} dV_{g(t)}\Big) + \theta \Big(\int_{\Omega} \Big|\nabla \big(\xi_i \chi^{p'+ 1} |Rc|^{\frac{p}{2}}\big)\Big|^2 dV_{g(t)} \Big) \nonumber \\
& \quad \leq \frac{1}{\epsilon} \widehat{C}_2(n, \widehat{N}_2) p^3 (r_0^{-2}+ 1) \Big(\int_{B_i\cap \Omega} \chi^{2p'} |Rc|^p \Big) \label{6.12}
\end{align}
where $\theta= \Big[\frac{4(p- 1- 4\epsilon)}{(1+ 4\epsilon)p}- \frac{32p \widehat{N}^{\frac{2}{n}}}{\tau \sqrt{n}} \Big]$ and $\epsilon> 0$ is any positive constant.
}
\end{lemma}

\pf
{Let $f(x, t)= |Rc(x)|_{g(t)}$. By Lemma \ref{lem 3.4} 
\begin{align}
\frac{\partial}{\partial t} \Big(\int_{\Omega} \xi^{2q} \chi^{2p'} f^p \Big)\leq \sum_{k= 1}^{6} I_k \label{6.13}
\end{align}
where
\begin{align}
& I_1= p \int \xi^{2q} \chi^{2p'+ 2} f^{p- 1} \Delta f , \quad I_2= p\int \xi^{2q}\chi^{2p'+ 2} f^{p- 2} |\nabla f|^2 \nonumber \\
& I_3= -(1- \epsilon)p \int \xi^{2q} \chi^{2p'+ 2} f^{p -2} |\nabla Rc|^2, \quad I_4= 2(1+ \sqrt{n})p \int \xi^{2q} \chi^{2p'+ 2} f^p |Rm| \nonumber \\
& I_5= \Big(\frac{10^5}{\epsilon}\Big)np \int \xi^{2q}\chi^{2p'} |\nabla \chi|^2 f^p, \quad I_6= p\int \xi^{2q} \chi^{2p'} g^{ik}g^{jl} R_{kl} J_2 \nonumber 
\end{align}
where $J_2$ in $I_6$ is from (\ref{3.13}).

Similar to the proof of Lemma \ref{thm 1.1}, we set 
\begin{align}
& \tilde{a}= \xi^{q}\chi^{p'} f^{\frac{p}{2}} |\nabla \chi | ; \quad \tilde{b}= \xi^{q} \chi^{p'+ 1} f^{\frac{p}{2}- 1} |\nabla f|; \nonumber  \\
& \tilde{c}= \xi^{q}\chi^{p'+ 1} f^{\frac{p}{2}- 1} |\nabla Rc | ; \quad \tilde{d}= \int \big|\nabla (\xi^q\chi^{p'+ 1} f^{\frac{p}{2}}) \big|^2 ; \quad  \tilde{e}= \xi^{q- 1} \chi^{p'+ 1} f^{\frac{p}{2}} |\nabla \xi |; \nonumber  
\end{align}

Then 
\begin{align}
I_1+ I_2+ I_3+ I_5+ I_6\leq \frac{4(1+ 4\epsilon- p)}{(1+ 4\epsilon )p}\tilde{d}+ \frac{4}{\epsilon}10^6 n^3 p^3 q^2\Big(\int \tilde{a}^2+ \int \tilde{e}^2 \Big) \label{6.14}
\end{align}

Finally for $I_4$, when $q\geq 2$ 
\begin{align}
I_4 &\leq 2(1+ \sqrt{n})p\Big(\int_{\Omega} \xi^4 \chi^2 |Rm|^{1+ \frac{n}{2}}\Big)^{\frac{2}{n+ 2}} \cdot \Big(\int_{B_i\cap \Omega} \xi^{2(q- 2)} \xi^{2p'} f^{p} \Big)^{\frac{2}{n+ 2}}  \nonumber \\
&\quad \quad \cdot \Big[\int_{\Omega} \Big(\xi^q \chi^{p'+ 1} f^{\frac{p}{2}}\Big)^{\frac{2n}{n- 2}}\Big]^{\frac{n- 2}{n+ 2}} \nonumber \\
& \leq \hat{a}\cdot \hat{b}  \nonumber 
\end{align}
in the last inequality we used Lemma \ref{lem 6.2} and the following notations:
\begin{align}
\hat{a}&= C(n, p_0, \widehat{N}_1)p \Big[\big(r_0^{-2}+ 1\big)(t+ 1)^2 t^{-1}\Big]^{\frac{2}{n+ 2}} \Big(\int_{B_i\cap \Omega} \xi^{2(q- 2)} \chi^{2p'} f^p \Big)^{\frac{2}{n+ 2}} \nonumber \\
\hat{b}&= \Big(\int_{\Omega} \Big|\nabla \big(\xi^q \chi^{p'+ 1} f^{\frac{p}{2}}\big)\Big|^2 \Big)^{\frac{n}{n+ 2}} \nonumber \\
s& = \frac{n+ 2}{2}, \quad s'= \frac{n+ 2}{n}, \quad \frac{1}{s}+ \frac{1}{s'}= 1 \nonumber
\end{align}
For any $\delta> 0$, 
\begin{align}
\hat{a}\hat{b}\leq \frac{1}{s}\Big(\hat{a}\delta^{-\frac{1}{ss'}}\Big)^s+ \frac{1}{s'}\Big(\hat{b}\delta^{\frac{1}{ss'}}\Big)^{s'} \leq \hat{a}^s\delta^{-\frac{1}{s'}}+ \hat{b}^{s'}\delta^{\frac{1}{s}} \nonumber
\end{align}
Let $\delta= k_0^{\frac{n+ 2}{2}}$, where $k_0$ is some positive constant which will be determined later.
\begin{align}
I_4\leq k_0^{-\frac{n}{2}} C(n, p_0, \widehat{N}_1)p^{\frac{n+ 2}{2}} \Big[(r_0^{-2}+ 1) (t+ 1)^2 t^{-1} \Big]\Big(\int_{B_i\cap \Omega } \xi^{2(q- 2)} \chi^{2p'} f^{p} \Big)+ k_0 \tilde{d} \label{6.15}
\end{align}
Now we choose $\epsilon= \frac{1}{16}$, $k_0= \frac{1}{9}$ in (\ref{6.14}) and (\ref{6.15}), from (\ref{6.13}) 
\begin{align}
&\frac{\partial}{\partial t}\Big(\int_{\Omega} \xi^{2q} \chi^{2p'} f^p\Big)+ \frac{1}{3}\int_{\Omega} \big|\nabla (\xi^q\chi^{p'+ 1} f^{\frac{p}{2}})\big|^2 \nonumber \\
&\leq C(n, p_0, \widehat{N}_1)p^{2+ \frac{n}{2}} q^2 (r_0^{-2}+ 1) (t+ 1)^2 t^{-1} \Big(\int_{B_i\cap \Omega} \xi^{2(q- 2)} \chi^{2p'} f^p \Big) \nonumber
\end{align}
Hence (\ref{6.11}) is proved.

When $q= 1$, from (\ref{6.2}), (\ref{6.4}) and H\"older inequality, 
\begin{align}
I_4\leq 8(1+ \sqrt{n})p\widehat{A}_0|Rm|_{\big(g(t), \ L^{\frac{n}{2}}(B_i\cap \Omega)\big)}\cdot \tilde{d} \leq \frac{32p\widehat{N}^{\frac{2}{n}}}{\tau \sqrt{n}}\tilde{d} \label{6.16}
\end{align}
From (\ref{6.13}), (\ref{6.14}) and (\ref{6.16}), 
\begin{align}
&\frac{\partial}{\partial t}\Big(\int_{\Omega} \xi_i^2 \chi^{2p'} f^p dV_{g(t)}\Big)+ \theta \Big(\int_{\Omega} \Big|\nabla \big(\xi_i \chi^{p'+ 1} f^{\frac{p}{2}}\big)\Big|^2 dV_{g(t)}\Big) \nonumber \\
& \leq \frac{4}{\epsilon} 10^6 n^3 p^3 \Big(|\nabla \chi|_{\infty}^2 \int_{\Omega} \xi_i^2 \chi^{2p'} f^p+ \int_{B_i\cap \Omega} \chi^{2p'+ 2} f^p |\nabla \xi_i|^2 \Big) \nonumber 
\end{align}
Replace $f$ by $|Rc|$ and simplify it, we get (\ref{6.12}).
}
\qed

Choose $\widehat{T}_1> 0$ such that
\begin{align}
\exp{\Big(16\widehat{C}_2 p_0^3 \widehat{N} (r_0^{-2}+ 1) \widehat{T}_1\Big)}= \frac{3}{2} \label{6.17}
\end{align}

\begin{cor}\label{cor 6.5}
{When $0\leq t< T_{\max}\leq \widehat{T}_1$, 
\begin{align}
\Big(\int_{B_i\cap \Omega} |Rc|^{p_0} dV_{g(t)}\Big)^{\frac{1}{p_0}} \leq \Big(\frac{3}{2}\widehat{N} \Big)^{\frac{1}{p_0}} \widehat{K}_1, \quad i= 1, 2, 3, \cdots \label{6.18}
\end{align}
}
\end{cor}

\pf
{Set $f= |Rc|$, choose $p'= 0$, $p= p_0$ and $\epsilon= \frac{1}{16}$ in (\ref{6.12}), then it is easy to check that $\theta\geq \frac{1}{15}\geq 0$ from the definition of $\theta$ in Lemma \ref{lem 6.3}, then
\begin{align}
\frac{\partial}{\partial t}\int_{\Omega} \xi_i^2 |Rc|^{p_0} \leq 16\widehat{C}_2 p_0^3 (r_0^{-2}+ 1) \Big(\int_{B_i\cap\Omega} |Rc|^{p_0}\Big) \nonumber 
\end{align}

Define $\phi(t)= \sup_{B_i\in Cov(\Omega)} \int_{B_i\cap \Omega} \xi_i^2 |Rc|^{p_0} dV_{g(t)}$, 
\begin{align}
\frac{\partial}{\partial t} \phi(t)\leq 16 \widehat{C}_2 p_0^3 (r_0^{-2}+ 1)\widehat{N} \phi(t) \nonumber
\end{align}

From the last inequality in (\ref{6.1}),
\begin{align}
\phi(t)\leq \exp{\Big(16\widehat{C}_2 p_0^3 \widehat{N} (r_0^{-2}+ 1)t\Big)} \phi(0)\leq \frac{3}{2} \widehat{K}_1^{p_0} \nonumber
\end{align}

Now we have
\begin{align}
\int_{B_i\cap \Omega} |Rc|^{p_0} dV_{g(t)} \leq \widehat{N}\phi(t) \leq \Big(\frac{3}{2}\widehat{N} \Big)\widehat{K}_1^{p_0} \nonumber
\end{align}

The conclusion is proved.
}
\qed

We still use Moser iteration to get local $C^{0}$-estimate $|\chi^2 Rc(g(t))|$.

\begin{prop}\label{prop 6.6}
{If $0< t< T_{\max}\leq \widehat{T}_1$, there exists some constant $\widehat{C}_3(n, \widehat{N}_1, p_0)> 0$ such that
\begin{align}
|\chi^2 Rc(g(t))|\leq \widehat{C}_3 \Big[\big(r_0^{-2}+ 1\big)\widehat{A}_0\Big]^{\frac{n}{2p_0}} \widehat{K}_1 \Big(t^{-\frac{n}{2p_0}}+ t^{\frac{n}{2p_0}}\Big) \label{6.19}
\end{align}
}
\end{prop}

\pf
{Let $f(x, t)= |Rc(x)|_{g(t)}$. From (\ref{6.11}), do similar argument as in Lemma \ref{lem 5.7} except defining $H$ as the following
\begin{align}
H(p, p', q, t)= \int_t^{t_0} \int_{B_i\cap \Omega} \xi^{2q} \chi^{2p'} f^p \nonumber
\end{align}
Denote $\nu= 1+ \frac{2}{n}$, $\eta= \nu^{n+ 2}$ and fix $t\in (0, t_0)$, set
\begin{align}
& p_k= p_0\nu^k, \quad p_k'= p_k- \frac{n}{2}, \nonumber \\
& q_k= 2n\nu^k- (n+ 2), \quad q_0= n- 2, \quad \tau_k= t(1- \eta^{-k}), \nonumber \\ 
& H_k= H(p_k, p_k', q_k, \tau_k), \quad \Phi_k= H_k^{\frac{1}{p_k}} \nonumber
\end{align}
Then 
\begin{align}
\Phi_0= \Big(\int_0^{t_0} \int_{\Omega} \xi^{2(n- 2)}\chi^{2p_0- n} f^{p_0}\Big)^{\frac{1}{p_0}} \nonumber
\end{align}
We can get
\begin{align}
&\Big|\xi_i^{\frac{2n}{p_0}}\chi^2 f(x, t) \Big|\leq \lim_{k\rightarrow \infty} \Phi_{k+ 1} \nonumber \\
&\leq C(n, \widehat{N}_1, p_0)\big(r_0^{-2}+ 1\big)^{\frac{n}{2p_0}} (t_0+ 1)^{\frac{n}{p_0}} \widehat{A}_0^{\frac{n}{2p_0}} t^{-\frac{n+ 2}{2p_0}} \Big(\int_0^{t_0} \int_{B_i\cap\Omega} \xi_i^{2(n- 2)} \chi^{2p_0- n} f^{p_0}\Big)^{\frac{1}{p_0}} \nonumber
\end{align}

Let $t_0\rightarrow t$ in the above and use Corollary \ref{cor 6.5}, 
\begin{align}
\Big| \xi_i^{\frac{2n}{p_0}} \chi^2 f(x, t)\Big|\leq C(n, \widehat{N}_1, p_0) \Big[\Big(r_0^{-2}+ 1\Big) \widehat{A}_0\Big]^{\frac{n}{2p_0}} \widehat{K}_1 \Big(t^{-\frac{n}{2p_0}}+ t^{\frac{n}{2p_0}}\Big) \nonumber
\end{align}

For any $x\in \Omega$, one can assume $x\in B_{i_0}\cap \Omega$, and note $|\xi_i^4\chi^2 f|\leq \Big| \xi_i^{\frac{2n}{p_0}} \chi^2 f\Big|$, then 
\begin{align}
|\chi^2 f|(x, t)&= \Big|\sum_{j= 1}^{\widehat{N}} \xi_j^2 \chi^2 f(x, t)\Big|= \Big|\Big(\sum_{j= 1}^{\widehat{N}}\xi_j^2\Big)\chi^2 f\Big| \nonumber \\
&\leq \widehat{C}(n, \widehat{N}_1, p_0) \Big[\Big(r_0^{-2}+ 1\Big)\widehat{A}_0\Big]^{\frac{n}{2p_0}} \widehat{K}_1 \Big(t^{-\frac{n}{2p_0}}+ t^{\frac{n}{2p_0}} \Big) \nonumber 
\end{align}
Replace $f$ with $|Rc|$, that is our conclusion.
}
\qed

\begin{lemma}\label{lem 6.7}
{When $0< t< T_{\max}$, for $q> 2$, $p\geq \frac{n}{2}$, $p'\geq 0$ and $p\geq p'$, there exists a constant $C(n, p_0, \widehat{N}_1)$ such that 
\begin{align}
&\frac{\partial}{\partial t}\Big(\int \xi_i^{2q} \chi^{2p'} |Rm|^p\Big)+ \frac{1}{9}\Big(\int \Big|\nabla \big(\xi_i^{q} \chi^{p'+ 1} |Rm|^{\frac{p}{2}}\big)\Big|^2\Big) \nonumber \\
& \leq C(n, p_0, \widehat{N}_1) p^{2+ \frac{n}{2}} q^2 (r_0^{-2}+ 1) (t+ 1)^2 t^{-1} \Big(\int_{B_i\cap\Omega} \xi_i^{2(q- 2)} \chi^{2p'} |Rm|^p \Big) \nonumber
\end{align}
}
\end{lemma}

\pf
{ Use Lemma \ref{lem 6.2}, the proof is similar to the proof of (6.11) in Lemma \ref{lem 6.3}.
}
\qed

\begin{lemma}\label{lem 6.8}
{If $0< t< T_{\max}$, there exists some positive constant $\widehat{C}_4(n, \widehat{N}_1, p_0)$ such that 
\begin{align}
|\chi^2 Rm(g(t))|\leq \widehat{C}_4 (r_0^{-2}+ 1)(t^{-1}+ t) \label{6.20}
\end{align}
}
\end{lemma}

\pf
{Similar to the proof of Proposition \ref{prop 6.6}, but replacing $p_0$ with $\frac{n}{2}$ and using Lemma \ref{lem 6.7}, (\ref{6.4}) instead of Lemma \ref{lem 6.3}, Corollary \ref{cor 6.5}.
}
\qed

\begin{cor}\label{cor 6.9}
{If $0< t< T_{\max}$, then there exists some positive constant $C$ independent of $t$ such that 
\begin{align}
|\chi^2 Rm(g(t))|\leq C \nonumber
\end{align}
}
\end{cor}

\pf
{Straightforward from Lemma \ref{lem 6.8}.
}
\qed

Choose $\widehat{T}_2> 0$ such that
\begin{align}
\exp{\Big(2n^3\widehat{C}_1\widehat{N}\widehat{T}_2\Big)}= \frac{3}{2} \label{6.21}
\end{align}
and choose $q= 1$, $p'= 0$, $p= \frac{n}{2}$, $\epsilon= \frac{1}{16}$ in Lemma \ref{lem 6.1}, do similar argument as in Corollary \ref{cor 6.5}, if $0\leq t\leq T_{\max}\leq \widehat{T}_2$,
\begin{align}
\Big(\int_{B_i\cap\Omega} |Rm(g(t))|^{\frac{n}{2}} dV_{g(t)}\Big)^{\frac{2}{n}} \leq \Big(\frac{3}{2}\widehat{N}\Big)^{\frac{2}{n}} \Big(\tau n\widehat{A}_0 \Big)^{-1}, \quad i= 1, 2, 3, \cdots \nonumber 
\end{align}

Define
\begin{align}
\widehat{T}_{3, 1}&= \Big[\frac{(2p_0- n)\ln 2}{32p_0} \widehat{C}_3^{-1}\widehat{K}_1^{-1}\big[(r_0^{-2}+ 1)\widehat{A}_0 \big]^{-\frac{n}{2p_0}}\Big]^{\frac{2p_0}{2p_0- n}}   \nonumber \\
\widehat{T}_{3, 2}&= \Big[\frac{(2p_0+ n)\ln 2}{32p_0} \widehat{C}_3^{-1}\widehat{K}_1^{-1}\big[(r_0^{-2}+ 1)\widehat{A}_0 \big]^{-\frac{n}{2p_0}}\Big]^{\frac{2p_0}{2p_0+ n}}   \nonumber  \\
\widehat{T}_3&= \min\{\widehat{T}_{3, 1}, \ \widehat{T}_{3, 2} \} \label{6.22}
\end{align}

By Proposition \ref{prop 6.6} and similar argument as in Proposition \ref{prop 5.11}, if $0\leq t< T_{\max}\leq \min\{\widehat{T}_1, \widehat{T}_3\}$, then $2^{-\frac{1}{4}}g_0\leq g(t)\leq 2^{\frac{1}{4}}g_0$.

We define 
\begin{align}
\widehat{T}_4= \min\{\widehat{T}_1, \widehat{T}_2, \widehat{T}_3 \} \label{6.23}
\end{align}
By the above argument, similar as Proposition \ref{prop 5.13}, we have

\begin{prop}\label{prop 6.10}
{If $0< T_{\max}\leq \widehat{T}_4$, then on $[0, T_{\max})$, we have
\begin{equation}\label{6.24}
{\left\{
\begin{array}{rl}
\Big(\int_{B_{x}(r_0)} h^{\frac{2n}{n- 2}} dV_{g_0} \Big)^{\frac{n- 2}{n}} &\leq 2\widehat{A}_0 \int_{B_{x}(r_0)} |\nabla h|^2 dV_{g_0} , \\
2^{-\frac{1}{4}}g_0\leq g(t)\leq 2^{\frac{1}{4}}g_0 , \\
\Big(\int_{B_i\cap \Omega } |Rm(g(t))|^{\frac{n}{2}} dV_{g(t)} \Big)^{\frac{2}{n}} &\leq \Big(\frac{3}{2}\widehat{N}\Big)^{\frac{2}{n}} \Big(\tau n \widehat{A}_0\Big)^{-1} , \\
|\chi^2 Rm(g(t))|&\leq C
\end{array}\right.
}
\end{equation}
for any $x\in M^n$ and $h\in C_0^{\infty} (B_x(r_0))$, where $C$ is independent of $t$.
}
\end{prop}

By similar argument to the proof of Theorem \ref{thm 5.14}, we have the following theorem:

\begin{theorem}\label{thm 6.11}
{Assume $(M^n, g_0)$ is a $n$-dimensional $(n\geq 3)$ complete noncompact
Riemannian manifold, which satisfies (\ref{6.1}). Then local Ricci flow (\ref{2.1}) has a
smooth solution on $[0, \frac{\widehat{T}_4}{2}]$. Moreover, for $t \in (0, \frac{\widehat{T}_4}{2}]$, the metric satisfies (\ref{6.3}) and the curvature tensors satisfy the estimates (\ref{6.19}), (\ref{6.20}).
}
\end{theorem}

\section{Short-time existence of the Ricci flow}

In this section we consider the family of local Ricci flows converging to the Ricci flow on a noncompact manifold $M^n$, we can get the short time existence of the Ricci flow on $M^n$ when we have the uniform lower bound of existence time for the family of the local Ricci flows . At many points in this section, we will take a subsequence; to simplify notation, at each stage a sequence such as $\{\chi_k\}$ will be re-indexed to continue to be $\{\chi_k\}$.

{\it \textbf{Proof of Theorem \ref{thm 1.1}}:}~
{Choose 
\begin{equation}\nonumber
{\Omega_i= \{x| \ d_{g_0}(x, \mathbf O)\leq 4^i,\ x\in M \} \quad i=1,2,\cdots
}
\end{equation}
where $\mathbf O$ is some fixed point in $M$. Construct $\chi_i$, which is a smooth cut-off function on $M$, such that 
\begin{equation}\nonumber
\chi_{i+ 1}(x)= \left\{
\begin{array}{rl}
1 \quad  & x\in \Omega_{i} \\
0 \quad  & x\in M\backslash \Omega_{i+ 1}
\end{array} \right.
\end{equation}
and 
\begin{equation}\label{7.1}
{0\leq \chi_i\leq 1, \quad \Vert\nabla \chi_i\Vert_{\infty}\leq 1, \quad \lim_{i\rightarrow \infty}\chi_i(x)=1
}
\end{equation}

Note $T_2$ is independent of $\chi_i$. By Theorem \ref{thm 5.14}, for each $k$, there exists $g_k(t)$ which is the solution of the following local Ricci flow on $[0, \frac{T_2}{2}]$:
\begin{equation}\label{7.2}
\left\{
\begin{array}{rl}
\frac{\partial}{\partial t}g_k &= -2\chi_k^2 Rc, \quad x\in M\\
g_k(x, 0)&= g_0(x), \quad x\in M 
\end{array} \right.
\end{equation}
Then fix $i$, for any $0< t_1< \frac{T_2}{2}$, we have
\begin{equation}\label{7.3}
{\Vert\nabla_k^{m} Rm_k(x, t_1)\Vert_{(g_k(t_1),\infty)}\leq C(m, \Omega_i , t_1),\quad x\in \Omega_i, \quad k\leq i 
}
\end{equation}
where $\nabla_k$ and $Rm_k$ are with respect to $g_k(t)$ and $C_{(m)}^i$ is a sequence of constants independent of $k$. For simplicity, we use the notation $|\nabla_k^m Rm_k(x,t_1)|_{\infty}$ instead of $\Vert\nabla_k^{m} Rm_k(x, t_1)\Vert_{(g_k(t_1),\infty)}$. If $k> i$, on $\Omega_i$, $g_k$ satisfies (\ref{1.1}) which is the Ricci flow. From (\ref{5.17})
\begin{equation}\label{7.4}
{|\chi_k^2 Rm_k(x, t_1)|_{\infty}\leq C_2 \Big(t_1^{-\frac{n}{2p_0}}+ C_1^{\frac{n+ 2}{2p_0}}t^{\frac{1}{p_0}}\Big) \ ,\quad x\in M, \quad \forall k
}
\end{equation}
Note $C_1$, $C_2$ are independent of $k$. Hence by definition of $\chi_i$, 
\begin{equation}\label{7.5}
{|Rm_k(x, t_1)|_{\infty}\leq C_2 \Big(t_1^{-\frac{n}{2p_0}}+ C_1^{\frac{n+ 2}{2p_0}}t^{\frac{1}{p_0}}\Big)  \ ,\quad x\in \Omega_i\ , \ k> i 
}
\end{equation}

Then by Theorem $14.14$ in \cite{RFT(II)},
\begin{equation}\label{7.6}
{\left.
\begin{array}{rl}
|\nabla_k^m Rm_k(x, t_1)|_{\infty}& \leq C(A_0,K_1,n,m,p_0,r_0,t_1,\Omega_i)t_1^{-\frac{m}{2}} \\
&\quad = C(A_0,K_1,n,m,p_0,r_0,t_1,\Omega_i)
\end{array}\right.
}
\end{equation}
for any $x\in \Omega_i$ and $k> i$.

By (\ref{7.3}) and (\ref{7.6}),  
\begin{equation}\label{7.7}
{|\nabla_k^m Rm_k(x, t_1)|_{\infty}\leq C_{(m)}^{i} \quad x\in \Omega_i, \ \forall k\geq 0
}
\end{equation}
where $C_{(m)}^i$ is a sequence of constants independent of $k$.

When $t\in [0, \frac{T_2}{2}]$, by Theorem \ref{thm 5.14} 
\begin{equation}\label{7.8}
{\frac{1}{2} g_0\leq g_k(t)\leq 2 g_0,
}
\end{equation}

then
\begin{equation}\nonumber
{C_1\sqrt{\det{g_0}}\leq \sqrt{\det{g_k(t_1)}}\leq C_2\sqrt{\det{g_0}}
}
\end{equation}

By Lemma $5.2$ of \cite{CRI}, there exists some $\tau_0> 0$, such that
\begin{equation}\label{7.9}
{inj_{g_k(t_1)}(\mathbf O)\geq \tau_0
}
\end{equation}

By (\ref{7.7}) and (\ref{7.9}), Theorem $3.9$ in \cite{RFT(I)} applies for $(\Omega_i, g_k(t_1), \mathbf O)_{k\in \mathbb N}$, then
\begin{equation}\nonumber
{\lim_{k\rightarrow \infty}g_k(x, t_1)= g_{\infty}(x, t_1) 
}
\end{equation}
where the limit is in $C^{\infty}$ sense on $\Omega_i$. Hence
\begin{equation}\label{7.10}
{|\widehat{\nabla}^m g_k(x, t_1)|_{\infty}\triangleq |\nabla_{g_0}^m g_k(x, t_1)|_{\infty}\leq C_{(m)}^i \, \quad x\in \Omega_i
}
\end{equation}
where $C_{(m)}^i$ is a sequence of constants independent of $k$.

From Theorem $14.14$ in \cite{RFT(II)} and the argument of getting (\ref{7.6}), when $k> i$, we have
\begin{equation}\label{7.11}
{|\nabla_k^m Rm_k(x, t)|_{\infty}\leq C_{(m)}^i\ , \quad (x, t)\in \Omega_i\times [t_1, \frac{T_2}{2}]
}
\end{equation}

If $k\leq i$, it is easy to get (\ref{7.11}) is also true. So for any $k$, (\ref{7.11}) is true, where $C_{(m)}^i$ is independent of $k$.

Then by (\ref{7.7}), (\ref{7.10}), (\ref{7.11}) and Lemma $3.11$ in \cite{RFT(I)},
\begin{equation}\nonumber
{\Big|\frac{\partial^q}{\partial t^q} \widehat{\nabla}^m g_k(x, t) \Big|_{\infty}\leq C_{(m, q)}^i, \quad (x, t)\in \Omega_i\times [t_1, \frac{T_2}{2}], \quad \forall k\geq 0
}
\end{equation}
where $C_{(m, q)}^i$ is a sequence of constants independent of $k$.

Now by Arzela-Ascoli theorem, we get
\begin{equation}\nonumber
{\lim_{k\rightarrow \infty}g_k(x,t)= g_{\infty}(x,t) \quad (x, t)\in \Omega_i\times [t_1, \frac{T_2}{2}]
}
\end{equation}
where the limit is in $C^{\infty}$ sense. Because $i$ and $t_1$ are arbitrary, by the diagonal method, 
\begin{equation}\label{7.12}
{g(x,t)\triangleq g_{\infty}(x,t)= \lim_{k\rightarrow \infty}g_k(x,t)  \quad (x, t)\in M\times (0, \frac{T_2}{2}]
}
\end{equation}

Then
\begin{equation}\nonumber
\left.
\begin{array}{rl}
|\lim_{t\rightarrow 0}\big(g(x,t)- g_0(x)\big)|&= |\lim_{t\rightarrow 0}\lim_{k\rightarrow \infty} \big(g_k(x,t)- g_k(x,0)\big)|\\
&= |\lim_{t\rightarrow 0}\lim_{k\rightarrow \infty} \int_0^t \frac{\partial}{\partial s}g_k(x,s)ds|\\
&\leq \lim_{t\rightarrow 0}\lim_{k\rightarrow \infty} \int_0^t |2\chi_k^2 Rc_k|_{\infty}\\
& \leq C\lim_{t\rightarrow 0} (t^{1- \frac{n}{2p_0}}+ t^{1+ \frac{1}{p_0}})= 0
\end{array} \right.
\end{equation}
In the last inequality above, we used (\ref{5.21}). Hence we can define
\begin{equation}\label{7.13}
{g(x,0)= g_0(x)  \quad x\in M^n
}
\end{equation}

Then by (\ref{7.2}), (\ref{7.12}) and (\ref{7.13}), we get $g(t)$ is the solution of the Ricci flow (\ref{1.1}) on $M^n\times [0, T]$. This is the first conclusion of theorem \ref{thm 1.1}; it remains to prove the estimates (\ref{1.3}).

Now by (\ref{7.5}) and (\ref{7.12}), 
\begin{equation}\nonumber
{|Rm_g(x,t_1)|_{\infty}\leq \lim_{k\rightarrow \infty}|Rm_k(x,t_1)|_{\infty}\leq C_2 \Big(t_1^{-\frac{n}{2p_0}}+ C_1^{\frac{n+ 2}{2p_0}}t_1^{\frac{1}{p_0}} \Big) 
}
\end{equation}
for any $x\in M^n$ and $t_1\in (0, \frac{T_2}{2}]$.

We take $(M^n, g(\frac{t_1}{2}))$ as the initial conditions of the Ricci flow, then the Ricci flow has solution on $M^n\times [\frac{t_1}{2}, \frac{T_2}{2}]$ and 
\begin{equation}\label{7.14}
{\Big|Rm_{g(\frac{t_1}{2})} \Big|_{\infty}\leq C(A_0,K_1,n,p_0,r_0) (t_1^{-\frac{n}{2p_0}}+ t_1^{\frac{1}{p_0}}) 
}
\end{equation}

Now define 
\begin{equation}\label{7.15}
{\left.
\begin{array}{rl}
\tilde{g}(x,t)= \Big[C(t_1^{-\frac{n}{2p_0}}+ t_1^{\frac{1}{p_0}})  \Big]\cdot g\Big(x, \frac{t_1}{2}+ t \Big[C\Big(t_1^{-\frac{n}{2p_0}}+ t_1^{\frac{1}{p_0}}\Big) \Big]^{-1} \Big)
\end{array}\right.
}
\end{equation}
where $C= C(A_0,K_1,n,p_0,r_0)$ in (\ref{7.15}) is the same as in (\ref{7.14}). Then from scaling argument,
\begin{equation}\nonumber
{\left\{
\begin{array}{rl}
\frac{\partial}{\partial t}\tilde{g}_{ij}(x,t)&= -2\tilde{R}_{ij}(x,t)  \\
\tilde{g}(x,0)&= \Big[ C(A_0,K_1,n,p_0,r_0)(t_1^{-\frac{n}{2p_0}}+ t_1^{\frac{1}{p_0}})  \Big] g(x,\frac{t_1}{2}) 
\end{array}\right.
}
\end{equation}

We already know that $g_{ij}$ and its Ricci flow exist on $[\frac{t_1}{2}, \frac{T_2}{2}]$, $\tilde{g}_{ij}$ and its Ricci flow exists on $\Big[0, (\frac{T_2}{2}- \frac{t_1}{2})C(A_0,K_1,n,p_0,r_0) (t_1^{-\frac{n}{2p_0}}+ t_1^{\frac{1}{p_0}})  \Big]$. We use $\{\tilde{R}_{ijkl}(x)\}$ and $\tilde{\nabla}$ to denote, repectively, the Riemannian curvature tensor and the covariant derivative with respect to $\tilde{g}_{ij}(x, 0)$.

Note $|\widetilde{Rm}|\leq 1$ on $M$, by Lemma $7.1$ of \cite{Shi}, there exists a positive constant $C(n,m)$, such that for any $\ m\geq 0$, if 
\begin{equation}\nonumber
{t\in \Big(0, (\frac{T_2}{2}- \frac{t_1}{2})C(A_0,K_1,n,p_0,r_0)\Big (t_1^{-\frac{n}{2p_0}}+ t_1^{\frac{1}{p_0}}\Big) \Big]
}
\end{equation}
we have 
\begin{equation}\label{7.16}
{|\tilde{\nabla}^m\widetilde{Rm}(x,t)|^2\leq \frac{C(n,m)}{t^m} 
}
\end{equation}
by (\ref{7.15}) and (\ref{7.16}), for $\frac{t_1}{2}< t\leq \frac{T_2}{2}$,
\begin{equation}\nonumber
{\left.
\begin{array}{rl}
|\nabla^m Rm(x,t)|^2&\leq \frac{C(n,m) \Big[ C(A_0,K_1,n,p_0,r_0) (t_1^{-\frac{n}{2p_0}}+ t_1^{\frac{1}{p_0}})  \Big]^{m+ 2}} {\Big[ C(A_0,K_1,n,p_0,r_0) (t_1^{-\frac{n}{2p_0}}+ t_1^{\frac{1}{p_0}}) (t- \frac{t_1}{2}) \Big]^{m}} \\
& \leq \frac{C(A_0,K_1,m,n,p_0,r_0)}{(t- t_1)^m}(t_1^{-\frac{n}{2p_0}}+ t_1^{\frac{1}{p_0}})^2 
\end{array}\right.
}
\end{equation}

Then
\begin{equation}\nonumber
{|\nabla^m Rm(x,t_1)|^2\leq \frac{C(A_0,K_1,m,n,p_0,r_0)}{t_1^m}(t_1^{-\frac{n}{2p_0}}+ t_1^{\frac{1}{p_0}})^2, \quad m\geq 0
}
\end{equation}
Because $t_1$ is chosen arbitrarily, 
\begin{equation}\nonumber
{|\nabla^m Rm(x,t)|\leq \frac{C(A_0,K_1,m,n,p_0,r_0)}{t^{\frac{m}{2}}}(t^{-\frac{n}{2p_0}}+ t^{\frac{1}{p_0}}), \quad for\ \ 0< t\leq \frac{T_2}{2}, \quad m\geq 0
}
\end{equation}
}
\qed

\begin{remark}\label{rem 7.1}
{In fact, the assumption about the integral of curvature tensor in the
above theorem can be weakened as
\begin{align}
\Big(\fint_{B_{x_i}(r_0)} |Rm|^{p_0} \Big)^{\frac{1}{p_0}} \leq K_1 \nonumber
\end{align}
where $B_{x_i}(r_0)$ is chosen as in Lemma \ref{lem 5.4}.
} 
\end{remark}

\begin{remark}\label{rem 7.2}
{By Remark \ref{rem 5.15} and the argument in Theorem \ref{thm 1.1}, if $(M^n, g_0)$ satisfies (\ref{5.29}), the Ricci flow (\ref{1.1}) has a smooth solution on $[0, T]$, where $T > 0$ is some positive constant.
} 
\end{remark}

{\it \textbf{Proof of Corollary \ref{cor 1.2}}:}~
{Because $Rc(g_0)\geq -Kg_0$, by Theorem $3.1$ in \cite{Saloff} and Corollary $1.1$ in \cite{LiSchoen}, $A_0\leq C(n, K, r_0)$, then apply Theorem \ref{thm 1.1}, we get our conclusion.
}
\qed

We also have the following corollary which is part of Shi's Theorem.
\begin{cor}\label{cor 7.3}
{Let $(M, g_{0})$ be an $n$-dimensional complete noncompact Riemannian manifold with its Riemannian curvature tensor $\{R_{ijkl}\}$ satisfying $|Rm|\leq k_0$ on $M$, where $0< k_0< +\infty$ is a constant. Then there is a constant $T(n, k_0)> 0$ depending only on $n$ and $k_0$ such that the evolution equation (\ref{1.1}) has a smooth solution $g_{ij}(x,t)$ on $[0,\ T(n,k_0)]$.
}
\end{cor}

\pf
{By $|Rm|\leq k_0$, there exists $K= nk_0$ such that 
\[Rc(g_0)\geq -Kg_0\]
and it is obvious
\begin{equation}\nonumber
{\Big( \fint_{\tilde{B}_{x}(1)}|Rm(g_0)|^{n} dV_{g_0} \Big)^{\frac{1}{n}} \leq k_0 
}
\end{equation}

Now we choose 
\[K= nk_0, \ K_1= k_0,\ p_0= n,\ r=1 \]
then by all the above choice of $K$, $K_1$, $p_0$ and $r$, (\ref{1.5}) in Corollary \ref{cor 1.2} are satisfied. We get our conclusion.
}
\qed

By Theorem \ref{thm 6.11}, we can use similar argument in the proof of Theorem \ref{thm 1.1} to get the following theorem.

\begin{theorem}\label{thm 7.4}
{Assume $(M^n, g_0)$ is a $n$-dimensional $(n\geq 3)$ complete noncompact Riemannian manifold satisfying (\ref{6.1}). Then the Ricci flow (\ref{1.1}) has a smooth solution $g_{ij}(x, t)$ on $[0, \widehat{T}_4]$ (where $\widehat{T}_4$ is defined in (\ref{6.23})) and satisfies the following estimates. For any integer $m\geq 0$, there exists a positive constant $C(\widehat{A}_0, \widehat{K}_1, \widehat{N}_1, m, n, r_0)$ such that 
\begin{align}
\sup_{x\in M^n} |\nabla^m Rm(x, t)|\leq \frac{C(\widehat{A}_0, \widehat{K}_1, \widehat{N}_1, m, n, r_0)}{t^{\frac{m}{2}}} (t^{-1}+ t), \quad 0< t\leq \frac{\widehat{T}_4}{2} \nonumber
\end{align}

Especially, 
\begin{align}
\sup_{x\in M^n} |Rm(x, t)|\leq C(n, \widehat{N}_1, r_0) (t^{-1}+ t), \quad 0< t\leq  \frac{\widehat{T}_4}{2}  \nonumber
\end{align}
where $C(n, \widehat{N}_1, r_0)$ depends only on $n$, $\widehat{N}_1$ and $r_0$.
}
\end{theorem}

\end{document}